\documentclass[journal]{IEEEtran}

\usepackage{amssymb}
\usepackage{cite}
\usepackage{color}
\usepackage{amsthm}
\usepackage{booktabs}
\usepackage{graphicx}
\usepackage{subfigure}
\usepackage{epstopdf}
\usepackage{multirow}

\usepackage{mathrsfs}
\usepackage[centertags]{amsmath}
\usepackage{amsfonts}
\usepackage{cite}
\usepackage{nicefrac}
\usepackage{enumerate}
\usepackage{xypic}
\usepackage{ulem}
\usepackage{cases}
\usepackage{diagbox}
\usepackage{verbatim}
\usepackage{diagbox}
\usepackage{longtable}
\usepackage{lscape}
\usepackage{stmaryrd}
\usepackage[figuresright]{rotating} 
\usepackage{tcolorbox}
\usepackage{stmaryrd}
\usepackage{rotating}

\usepackage{algorithm}
\usepackage{algorithmic}
\ifCLASSINFOpdf
\else
\fi
\usepackage{amsmath}
\usepackage{algorithmic}
\usepackage{array}
\usepackage{url}
\usepackage{enumitem}

\graphicspath{ {./Figures/} }

\newtheorem{theorem}{Theorem}

\newtheorem{lemma}{Lemma}
\newtheorem{proposition}{Proposition}

\newtheorem{definition}{Definition}
\newtheorem{remark}{Remark}
\newtheorem{assumption}{Assumption}

\newtheorem{example}{Example}

\newcommand{\ba}{\begin{array}}
\newcommand{\ea}{\end{array}}
\newcommand{\be}{\begin{equation}}
\newcommand{\ee}{\end{equation}}

\newcommand{\RNum}[1]{\lowercase\expandafter{\romannumeral #1\relax}}
\newcommand{\RNumU}[1]{\uppercase\expandafter{\romannumeral #1\relax}}
\begin{document}
\title{\huge Group zero-norm regularized robust loss minimization: proximal MM method and statistical error bound
}


%\author{{Chun-Na~Li,
%        Yuan-Hai~Shao, Wo-Tao~Yin,
%        and~Ming-Zeng~Liu}% <-this % stops a space
%\thanks{C. Li was with the College of Zhijiang, Zhejiang University of Technology, Hangzhou, 310024, P.R.China .}% <-this % stops a space
%\thanks{Y. Shao is with the School of Economics and Management, Hainan University, Haikou, 570228, P.R.China.}% <-this % stops a space
%\thanks{W. T. Yin is with the Math Department, University of California, Los Angeles, 951555, USA}
%\thanks{M. Liu is with the School of Mathematics and Physics Science, Dalian University of Technology at Panjin, Dalian, 124221, P.R.China.}}
\author{Ling Liang$^{1,2}$, Shujun Bi$^{2}$% <-this % stops a space
\IEEEcompsocitemizethanks{ This work is supported by the National Natural Science Foundation of China (No.12371323).\\
1. School of General Education, Guangzhou College of Technology and Business, Guangzhou, China. \\
2. School of Mathematics, South China University of
Technology, Guangzhou, China. \\
 (e-mail: liangling@gzgs.edu.cn; bishj@scut.edu.cn).}}% <-this % stops a space

\markboth{Journal of \LaTeX\ Class Files,~Vol.~ , No.~ ,  ~ }%
{Shell \MakeLowercase{{\em et al.}}: Group zero-norm regularized robust loss minimization:proximal MM method and statistical error bound}
% The on-ly time the second header will appear is for the odd numbered pages



\maketitle

\begin{abstract}
This study focuses on solving group zero-norm regularized robust loss minimization problems. We propose a proximal Majorization-Minimization (PMM) algorithm to address a class of equivalent Difference-of-Convex (DC) surrogate optimization problems. First, we present the core principles and iterative framework of the PMM method. Under the Kurdyka-{\L}ojasiewicz (KL) property assumption of the potential function, we establish the global convergence of the algorithm and characterize its local (sub)linear convergence rate. Furthermore, for linear observation models with design matrices satisfying restricted eigenvalue conditions, we derive statistical estimation error bounds between the PMM-generated iterates (including their limit points) and the ground truth solution. These bounds not only rigorously quantify the approximation accuracy of the algorithm but also extend previous results on element-wise sparse composite optimization from reference [57]. To efficiently implement the PMM framework, we develop a proximal dual semismooth Newton method for solving critical subproblems. Extensive numerical experiments on both synthetic data and the UCI benchmark demonstrate the superior computational efficiency of our PMM method compared to the proximal Alternating Direction Method of Multipliers (pADMM).
\end{abstract}
\begin{IEEEkeywords}
group zero-norm,  KL property, PMM,  semismooth Newton
\end{IEEEkeywords}

% \fi
\IEEEpeerreviewmaketitle

%\numberwithin{equation}{section}
\section{Introduction}
\IEEEPARstart{G}{}roup sparse optimization extends traditional sparse modeling by incorporating structural sparsity constraints. Whereas conventional approaches focus on element-wise sparsity through $\ell_1$-norm regularization, modern applications increasingly require group-wise selection mechanisms to handle clustered variable structures-a critical capability absent in classical formulations. This paradigm shift has positioned group sparse optimization as a pivotal methodology in high-dimensional statistics.

\subsection{Motivation}

Consider a feature space partitioned by disjoint index sets $\bigcup_{i=1}^mJ_i$ satisfying $\bigcup_{i=1}^mJ_i=\{1,2,\ldots,p\}$ with $J_i\cap J_j=\emptyset$ for all $i\neq j\in\{1,2,\ldots,m\}$.  The fundamental group sparse recovery problem seeks solutions to underdetermined systems:
\begin{equation}\label{observation model}
	b=\sum^m_{i=1}A_{J_i}{x^*}_{J_i}+\varpi,
\end{equation}
where $x^*\in\mathbb{R}^p$ denotes the ground true coefficients, $A_{J_i}\in \mathbb{R}^{n\times |J_i|}(i=1,2,\ldots,m)$ the design submatrices, and $\varpi\sim(0,\sigma^2I_n)$ the measurement noise. Notably, this generalizes standard sparse regression when $|J_i|=1$ for arbitrary $\{i\}$. Pioneered by Yuan and Lin\cite{Yuan06}, the canonical formulation employs group zero-norm regularization:
\begin{equation}\label{op0}
	\min_{x\in\mathbb{R}^p}\Theta_{\nu,\mu}(x):=f(x)+\nu\|x\|_{\rm gz},
\end{equation}
where $\|x\|_{\rm gz}:=\sum_{i=1}^m{\rm sign}(\|x_{J_i}\|_2)$ counts nonzero groups. The nonconvex, discontinuous nature of $\|\cdot\|_{\rm gz}$ induces NP-hardness, necessitating specialized optimization strategies.

Group sparse optimization has emerged as a cornerstone methodology with multidisciplinary applications. In computational imaging, it serves as a powerful regularizer for inverse problems including image denoising, deblurring, and inpainting. Through strategic partitioning of image attributes (color channels, texture patches) into coherent groups and enforcing inter-group sparsity constraints, this approach achieves superior reconstruction fidelity by preserving structural coherence \cite{Elad2010,Tropp2010,Duarte2011, Yuan12}. The statistical learning domain benefits from its dual capability in feature selection and model compression. By imposing group-wise sparsity on correlated feature clusters, the framework automatically prunes redundant variables while retaining discriminative feature groups—a mechanism proven to enhance generalization performance through intrinsic dimension reduction. Beyond single-task learning, multi-task variants leverage inter-task correlations via shared sparsity patterns, where parameter groups across tasks are jointly regularized to capture common representations while maintaining task-specific adaptations \cite{Yuan06, Meier08, Obozinski10, Huang10,ZhangGY08,Jenatton2011,Gong13}. Compressive sensing theory further reveals its theoretical advantages: As a structured sparsity paradigm, block or group sparsity not only generalizes standard sparse signal models but also permits tighter sampling bounds and improved recovery guarantees compared to element-wise sparsity \cite{Eldar09, Stojnic09,Davenport2011,Duarte2011}. Emerging applications in computational biology demonstrate its utility for genomic feature selection, where gene regulatory networks naturally exhibit group-structured sparsity \cite{Jenatton2011}. These developments have spurred extensions of classical sparse recovery algorithms to handle block-sparse constraints. Notable approaches include convex relaxation via the $\ell_{2,1}$-norm \cite{Yuan06, Meier08,Deng13,YangZou15}, iterative hard thresholding (IHT) methods \cite{Beck19} and nonconvex surrogate formulations that approximate the original NP-hard problem \eqref{op0} through tractable alternatives \cite{Ling13,Guo15,Hu17}; The nonconvex surrogates typically replace the discontinuous group zero-norm with non-Lipschitz $\ell_\alpha$-norm regularizers ($0<\alpha<1$) \cite{Chartrand07,Chen10}, smoothing concave approximations \cite{Bradley98,Rinaldi10} and folded concave penalties including SCAD and MCP \cite{Fan01,Zhang10}.
 
Recent advancements in nonconvex optimization have inspired new algorithmic developments. D. Sun et al. \cite{Tang2019} pioneered a PMM method for square-root-loss regression problems, which was subsequently extended by Zhang et al. \cite{Zhang-Pan-Bi-Sun23} to handle Difference-of-Convex (DC) surrogates of zero-norm regularized optimization. While these works advanced element-wise sparse modeling, they remain limited in addressing group-structured sparsity.

Motivated by these foundations, the authors propose the following group zero-norm regularized robust loss composite optimization problems
\begin{equation}\label{prob1}	\min_{x\in\mathbb{R}^p}\Theta_{\nu,\mu}(x):=\vartheta(Ax-b)+\frac{\mu}{2}\|x\|^2+\nu\|g(x)\|_{\rm gz}+\mathbb{I}_{\mathcal{X}}(x),
\end{equation}
where $\mathcal{X}\subseteq\mathbb{R}^p$ denotes a closed convex set, $f(x)\!:=\vartheta(Ax\!-\!b)+\frac{\mu}{2}\|x\|^2$ with $\vartheta\!:\mathbb{R}^n\to\mathbb{R}$ being a lower-bounded finite-valued convex function, $g(x)\!:=Bx$ encodes structured sparsity patterns through linear mapping, $A\in\mathbb{R}^{n\times p},B\in\mathbb{R}^{N\times p}$ and~$b\in\mathbb{R}^n$ are given data, $\mu>0$ ensures coercivity of $f$ over non-compact set $\mathcal{X}$ (redundant when $\mathcal{X}$ is compact).

The group sparse structure is explicitly characterized by
\begin{equation}\label{thetax}
	G(x)\!:=(\|g_{\!_{J_1}}\!(x)\|,\ldots,\|g_{\!_{J_m}}\!(x)\|)
\end{equation}
with composite regularization $\theta_{\nu}(x)\!:=\nu\|g(x)\|_{\rm gz}+\mathbb{I}_{\mathcal{X}}(x)$. Noted that $\|g(x)\|_{\rm gz}=\|G(x)\|_0$ represents the nonzero elements of $G(x)$.\\
{\bf Key advances beyond reference \cite{Zhang-Pan-Bi-Sun23}:}\\
1). Group-structured sparsity: $\|\cdot\|_{\rm gz}$ operates on predefined groups $\{J_i\}_{i=1}^m$ rather than individual elements\\
2). Structural encoding: Matrix $B$ models inter-group dependencies\\
3). Generalized constraints: $\mathbb{I}_{\mathcal{X}}(\cdot)$ accommodates convex domain restrictions.

\subsection{Prior Arts and Contributions} \label{Section-relatedlectures}

In fact, the model \eqref{prob1} covers various problems in many pieces of literature previously mentioned. In addition, when the smooth mapping $g$ is taken as an identity mapping namely $g(x)=x$ and $\{J_i\}=i$, the model \eqref{prob1} is reduced to the case in \cite{Zhang-Pan-Bi-Sun23}. When the smooth mapping $g$ is taken as an identity mapping namely $g(x)=x$, this model is a general group zero-norm optimization problem \eqref{op0}. When $g(x)=Bx$, and $B$ takes a special matrix, the model can correspond to the group sparse optimization problem of clustering \cite{David21}. Therefore, this paper considers a general form, namely $g$ is a smooth mapping. The authors provide a unified mechanism for the group zero-norm regularized composite optimization problem \eqref{prob1} in a more general manner.
%\begin{figure*}[!th]
%\centering
%\includegraphics[width=.85\textwidth]{activations.pdf}
%\caption{Six activation functions.}\vspace{-5mm}
%\label{figureactivation}
%\end{figure*}

As is well known, group variable selection models with nonconvex penalties are not without drawbacks. Although nonconvex penalties are beneficial for coefficient estimation, they can lead to nonconvex optimization problems. In nonconvex programming, optimization problems typically have multiple local optimal solutions that are not global. Solutions generated by algorithms such as gradients or coordinate descent methods may be trapped in local optima.
\cite{Pan-Liang-Liu2023} provided several local optimality conditions for a class of stationary points of problem \eqref{prob1}, which is the closest to the one from an algorithm. Based on this, the authors provide an improved PMM algorithm and analyze its global convergence and local (sub)linear rate of the iterative sequence generated, and also explores the statistical properties of group Lasso.

As a matter of fact, many scholars have conducted extensive research on the statistical properties of Lasso. For example, \cite{Steinhaeuser2012} investigated consistency results for estimators with general tree-structured norm regularizer, where group Lasso is its special case. \cite{Poignard2018} demonstrated the asymptotic properties of adaptive Lasso estimates. \cite{Rao2015, Rao2013} explored multi-task learning and classification problems based on group Lasso estimation variants. \cite{Li2015} studied multiple linear regression by using group Lasso. \cite{Ahsen2017} proposed a theoretical framework for developing error bounds for group Lasso and group Lasso with overlapping tree structures. Specifically, their results indicated that if the sample size satisfies a certain amount, the true solution is highly likely to be restored. Early works on the statistical theory of SCAD or MCP penalty, least squares estimators focused on the error bounds of global optima \cite{Zhang-Zhang2012} or local optima \cite{Fan2014}, which were obtained through specific initialization schemes and algorithms. Recently, in respect of the least squares objective function with SCAD or MCP penalty, \cite{Loh2015} established statistical characteristics for all stationary points(although their results are not directly applicable to SCAD or MCP penalty group estimators). Unlike previous results, this paper is inspired by reference \cite{Zhang-Pan-Bi-Sun23} to consider group zero-norm regularized robust loss minimization, where robust loss takes piecewise linear loss and squared strongly convex loss into consideration. Under certain restricted eigenvalue conditions of the design matrix, a statistical error bound is established for the iterative sequence and its stationary points generated by the proximal MM algorithm to the true solution of the problem.

%\subsection{Contributions}

%\begin{itemize}[leftmargin=20pt]
%\item[C1)]

%\item[C2)] 

%\item[C3)]
%\end{itemize}

\subsection{Organization and Notation}
The remainder of this paper is organized as follows. Section \ref{section2} reviews certain preliminaries for utilization. Some theoretical conditions and results are furnished to ensure the global convergence and (sub)linear convergent rate of the algorithm proposed in Section \ref{section3}. Additionally, the PMM method is employed to address the variant characterization of group zero-norm optimizations, with its dual problem solved by semismooth Newton method, and our
main convergence results are derived in Section \ref{section3}. Furthermore, in Section~\ref{section4}, the authors estimate the error bound between the iterates and the ground truth solution in a statistical sense. Section \ref{section5} assesses the performance of PMM-SNCG on both the synthetic and real data sets. Finally, concluding remarks are presented in Section \ref{section6}.

We end this section by summarizing the notation to be employed throughout the paper.
Throughout the article, the authors introduce the following notations. Let $\mathbb{R}$ ($\mathbb{R}_{+}$) be all (nonnegative) real number set. Let $\mathbb{R}^n$ be the real vector space of dimensional $n$ endowed with the Euclidean inner product $\langle\cdot,\cdot\rangle$ and induced norm $\|\cdot\|$. This paper denotes by $e$ a column vector of all 1s and $I$ an identity matrix whose dimension are known from the context. For any $x\in\mathbb{R}^n$ and $q\in[1,+\infty]$, $\|x\|_q$ denotes the $\ell_q$-norm of $x$ and $\|x\|$ represents $\|x\|_2$. $\|x\|_{(s)}$ denotes the sum of the first $s$ maximum elements of a vector $|x|$. ${\rm sign}(x)$ denotes the sign vector of $x$. For given positive integer $k$, denote by $[k]:=\{1,2,\ldots,k\}$. $\mathbb{I}_{\mathcal{X}}$ represents the indicator function with respect to the set $\mathcal{X}$. Let $J_1,J_2,\ldots,J_m$ be a collection of index sets to represent the group structure of explanatory factors, where $J_i\cap J_j\neq\emptyset$ for all $i\neq j\in\{1,2,\ldots,m\}$ and $\bigcup_{i=1}^mJ_i=\{1,2,\ldots,n\}$. For given index set $J\subset [n]$ and for arbitrary $x\in\mathbb{R}^n$, $x_{J}$ is a subvector formed by the elements of the indicator in $J$. Let $A:=[A_{J_1} A_{J_2} \cdots A_{J_m}]\in\mathbb{R}^{n\times p}$ and ${G}(x):=(\|x_{J_1}\|,\|x_{J_2}\|,\ldots,\|x_{J_m}\|)^{\top}$ for $x\in \mathbb{R}^p$, where $A_{J_i}$ for $i=1,2,\ldots,m$ is an $n\times |J_i|$ submatrix composed of columns with indicators in $J$. For a given matrix $A\in\mathbb{R}^{n\times p}$, $\|A\|$ denotes the spectral norm of $A$.
Given a continuously differentiable mapping $g\!:\mathbb{R}^p\to\mathbb{R}^{N}$, denote by $\nabla g(x)=[g'(x)]^{\top}$, where $g'(x)$ represents the~Jacobian matrix of $g$ at~$x$. Given a closed, proper and convex function $f\!:\mathbb{R}^ n\to\overline{\mathbb{R}}$, $f^*(\cdot):=\sup_{x\in\mathbb{R}^n} \{\langle \cdot, x \rangle-f(x)\}$ represents the conjugate function of $f$. For lower semicontinuous and convex function $f\!:\mathbb{R}^n\to\overline{\mathbb{R}}$, $\mathcal{P}_{\gamma}f$ and $e_{\gamma}f$ represent the proximal mapping and Moreau envelop of $f$ with $\gamma>0$, i.e., $\mathcal{P}_{\gamma}f(x):=\mathop{\arg\min}_{z\in\mathbb{R}^n}\Big\{f(z)+\frac{1}{2\gamma}\|z-x\|^2\Big\}$ and $e_{\gamma}f(x):=\min_{z\in\mathbb{R}^n}\Big\{f(z)+\frac{1}{2\gamma}\|z-x\|^2\Big\}$.

\section{Preliminaries}\label{section2}

\subsection{Generalized subdifferentials and subderivatives}
This subsection reviews the notions of subderivative and proximal, regular and limitting subdifferentials.
\begin{definition}\label{subderivative}\cite[Definition 8.1 $\&$ 13.3]{RW98} For a function $f:\mathbb{R}^p\rightarrow[-\infty,\infty]$, a point $\overline{x}$ with $f(\overline{x})$ finite and any $v\in\mathbb{R}^p$, the subderivative function $df(\overline{x}):\mathbb{R}^p\rightarrow[-\infty,\infty]$ is defined by
	\[
	df(\overline{x})(w):=\liminf\limits_{\tau \downarrow 0 \atop w^{\prime} \rightarrow w} \frac{f\left(\overline{x}+\tau w^{\prime}\right)-f(\overline{x})}{\tau},
	\]
	while the second subderivative of $f$ at $\overline{x}$ for $v$ and $w$ is defined by
	\[
	d^2f(\overline{x}|v)(w):={\displaystyle\liminf_{\tau\downarrow 0\atop w'\to w}}
	\frac{f(\overline{x}+\tau w')-f(\overline{x})-\tau\langle v,w'\rangle}{\frac{1}{2}\tau^2}.
	\]
\end{definition}

%-------------------------------------------------------------------------------
\begin{definition}\label{Gsubdiff-def}
	(\cite[Definition 8.45 $\&$ 8.3]{Roc70})
	Consider a function $f\!:\mathbb{R}^p\to[-\infty,+\infty]$ and a point
	$\overline{x}$ with $f(\overline{x})$ finite. The proximal subdifferential of $f$ at $\overline{x}$,
	denoted by $\widetilde{\partial}\!f(\overline{x})$, is defined as
	$
	\widetilde{\partial}\!f(\overline{x}):=\bigg\{v\in\mathbb{R}^p\ \big|\
	\liminf_{x\ne x'\to \overline{x}}
	\frac{f(x')-f(\overline{x})-\langle v,x'-\overline{x}\rangle}{\|x'-\overline{x}\|^2}>-\infty\bigg\};
	$
	the regular subdifferential of $f$ at $\overline{x}$, denoted by $\widehat{\partial}\!f(\overline{x})$, is defined as
	\[
	\widehat{\partial}\!f(\overline{x}):=\bigg\{v\in\mathbb{R}^p\ \big|\
	\liminf_{\overline{x}\ne x'\to \overline{x}}\frac{f(x')-f(\overline{x})-\langle v,x'-\overline{x}\rangle}{\|x'-\overline{x}\|}\ge 0\bigg\};
	\]
	and the (limiting) subdifferential of $f$ at $\overline{x}$, denoted by $\partial\!f(\overline{x})$, is defined as
	$
	\partial\!f(\overline{x}):=\Big\{v\in\mathbb{R}^p\ |\  \exists\,x^k\to \overline{x}\ {\rm with}\ f(x^k)\to f(\overline{x})\ {\rm and}\
	v^k\in\widehat{\partial}\!f(x^k)\ {\rm with}\ v^k\to v\Big\}.
	$
\end{definition}
%-------------------------------------------------------------------------------
\begin{remark}\label{remark-Gsubdiff}
	{\bf(a)} At each $\overline{x}$ with $f(\overline{x})$ finite, the sets $\widehat{\partial}f(\overline{x})$ and
	${\partial}\!f(\overline{x})$ are closed, $\widehat{\partial}\!f(\overline{x})$ and $\widetilde{\partial}f(\overline{x})$ are convex, but $\widetilde{\partial}f(\overline{x})$ is generally not closed(see the example in~\cite[Figure 6-12(a)]{RW98}), and
	$\widetilde{\partial}f(\overline{x})\subseteq\widehat{\partial}\!f(\overline{x})\subseteq\partial\!f(x)$.
	These inclusions may be strict if $f$ is nonconvex. When $f$ is convex,
	they reduce to the subdifferential of $f$ at $\overline{x}$ in the sense of \cite{Roc70}.
	
	\noindent
	{\bf(b)} The point $\overline{x}$ at which $0\in\partial\!f(\overline{x})$
	(respectively, $0\in\widetilde{\partial}\!f(\overline{x})$ and
	$0\in\widehat{\partial}\!f(\overline{x})$) is called a limiting
	(respectively, proximal and regular) critical point of $f$. By Definition \ref{Gsubdiff-def}, obviously, a local minimum of $f$
	is a proximal critical point of $f$ and then a regular and limiting critical point.
	
	\noindent
	{\bf(c)} Recall that a function $f\!:\mathbb{R}^p\to[-\infty,+\infty]$
	is said to be semiconvex if there exists a constant $\gamma>0$ such that
	$x\mapsto f(x)+\frac{\gamma}{2}\|x\|^2$ is convex, and the smallest $\gamma$
	of all such $\gamma$ is called the semiconvex modulus of $f$.
	For such a function,
	\(
	\widetilde{\partial}\!f(x)=\widehat{\partial}\!f(x)=\partial\!f(x)
	\)
	at all $x$ with $f(x)$ finite.
\end{remark}

The~KL property of extended real-valued functions and the~KL property of exponents~$1/2$ play a crucial role in the proof of the global convergence and local convergence rate of the iterative sequence generated by algorithms for solving nonconvex and nonsmooth optimization problems. This type of error bound property is explained as follows.

\begin{definition}\label{kl-property}
	(Kurdyka-{\L}ojasiewicz property)~Let $f:\mathcal{X}\rightarrow(-\infty,+\infty]$ be a proper function. The function $f$ is said to have the Kurdyka-{\L}ojasiewicz (KL) property at $\overline{x}\in{\rm dom}\partial f$ if there exist $\eta\in(0,+\infty]$, a continuous concave function $\varphi:[0,\eta)\rightarrow\mathbb{R}_+$ satisfying
	
	(i) $\varphi(0)=0$ and $\varphi$ is a continuously differentiable on $(0,\eta)$;
	
	(ii) for all $s\in(0,\eta)$, $\varphi'(s)>0$, and a neighborhood $\mathcal{U}$ of $\overline{x}$ such that $\varphi'(f(x)-f(\overline{x})){\rm dist}(0,\partial f(x))\geq1$ for all $x\in\mathcal{U}\cap[f(\overline{x})<f(x)<f(\overline{x})+\eta]$.
	
	In particular, when the function $\varphi$ can be selected as $\varphi(s)=c\sqrt{s}$, where $c>0$ is a constant, then the function $f$ is said to have the~KL property of exponential $\frac{1}{2}$ at the point $\overline{x}$. If $f$ has the KL property of exponent $\frac{1}{2}$ at every point in the set~${\rm dom}\partial f$, then $f$ is called a~KL function of exponent $\frac{1}{2}$.
\end{definition}

\begin{remark}\label{KL-remark}
	{\bf(a)} According to~\cite[Lemma 2.1]{Attouch10}, the proper lower semicontinuous functions have the~KL property of exponents~$\frac{1}{2}$ at noncritical points. Therefore, in order to prove that they are KL functions with an exponent of $\frac{1}{2}$, it is only necessary to prove that they possess this property at any critical point. Interested readers can refer to the calculation of the KL exponents\cite{LiPong18,Yu21,WuPanBi21}.
	
	\medskip
	\noindent
	{\bf(b)} According to \cite[Section~4]{Attouch10}, many types of functions are KL functions, such as semi-algebraic functions, globally subanalytic functions, and functions that can be defined in the real o-minimum structure.
\end{remark}

\subsection{Optimization properties of a family of composite functions}

Let $\mathscr{L}$ be a family of functions composed of closed proper convex functions $\phi\!:\mathbb{R}\to(-\infty,\infty]$ that satisfy the following properties:
\begin{equation}\label{phi-assump}
	{\rm int}({\rm dom}\,\phi)\supseteq[0,1],\
	t^*\!:=\mathop{\arg\min}_{0\le t\le 1}\phi(t),\ \phi(t^*)=0
	\ \ {\rm and}\ \ \phi(1)=1.
\end{equation}
For every $\phi\in\!\mathscr{L}$, let $\psi\!:\mathbb{R}\to(-\infty,\infty]$ be a closed proper convex function induced by the function $\phi$ as follows
\begin{equation}\label{psi-fun}
	\psi(t):=\!\left\{\!\begin{array}{cl}
		\phi(t) &{\rm if}\ t\in [0,1];\\
		\infty &{\rm if}\ t\notin [0,1].
	\end{array}\right.
\end{equation}
Let $\mathscr{L}_1$ be a family of strictly convex functions $\phi\in\!\mathscr{L}$ defined on $[0,1]$, and let $\mathscr{L}_2$ be composed of the functions $\phi\in\!\mathscr{L}_1$ satisfying $\phi(0)=0$ and there exist $0<\gamma_1<\gamma_2$ such that $\phi_{-}'(0)\ge\gamma_1$ and $\phi_{+}'(1)\le\gamma_2$. According to the \eqref{phi-assump}, it can be verified that $\mathscr{L}_2\subset\{\phi\in\mathscr{L}_1\ |\ t^*=0\}\subset\mathscr{L}_1$.

The following lemma characterizes the properties of the conjugate function $\psi^*$ of $\psi$ induced by functions in the function family $\mathscr{L}$. Its proof can be easily obtained using the definition of the function family $\mathscr{L}$, \cite[Theorem 26.3]{Roc70}, and \cite[Proposition 12.60]{RW98}, and will not be repeated here.
%------------------------------------------------------------------------------
\begin{lemma}\label{psi-star}
	Let $\psi$ be the function induced by $\phi\in\!\mathscr{L}$ according to \eqref{psi-fun}. Then $\psi^*$ is a finite nondecreasing convex function on~$\mathbb{R}$. If $\phi\in\mathscr{L}_1$, then $\psi^*$ is a continuously differentiable function on $\mathbb{R}$; In addition, if $\phi$ is also strongly convex on $[0,1]$, then $(\psi^*)'$ is Lipschitz continuous on $\mathbb{R}$.
\end{lemma}

\begin{lemma}\label{subdiff-vphirho}
	Let~$\psi$ be the function induced by~$\phi\in\!\mathscr{L}$ according to~\eqref{psi-fun}. For given~$\rho>0$, define
	\[
	\varphi_{\rho}(t):=t-\!\rho^{-1}\psi^*(\rho t)\quad\ \forall t\in\mathbb{R}.
	\]
	Then, for arbitrary~$\overline{t}\in\mathbb{R}$, $\partial\varphi_{\rho}(\overline{t})=1-\partial_{B}\psi^*(\rho\overline{t})$.
\end{lemma}
\begin{proof}
	By Lemma \ref{psi-star}, the function $\varphi_{\rho}$ is locally Lipschitz continuous, combining \cite[Corollary 9.21]{RW98}, the conclusion is obtained.
\end{proof}
%------------------------------------------------------------------------------
\begin{lemma}\label{smoothness}
	Let $\psi$ be the function induced by $\phi\in\!\mathscr{L}$ according to \eqref{psi-fun} and $F\!:\mathbb{R}^n\to\mathbb{R}^{l}$ be a continuously differential mapping. Suppose $\phi$ is strictly convex on $[0,1]$ and $\phi(0)=0$ and there exists $\gamma>0$ such that $\phi_{-}'(0)\ge\gamma$. Then, for arbitrarily given $\rho>0$ and $p\in(1,\infty)$, the function $\Gamma_{\rho}(x):=\rho^{-1}\psi^*(\rho\|F(x)\|_p)$ is continuously differentiable on $\mathbb{R}^n$. Moreover, 
	$\nabla\Gamma_{\!\rho}(x)=(\psi^*)'(\rho\|F(x)\|_p)\frac{\nabla F(x)[{\rm sign}(F(x))\circ|F(x)|^{p-1}]}{\|F(x)\|_p^{p-1}}$ if $F(x)\ne 0$ and $\nabla\Gamma_{\!\rho}(x)=0$ if $F(x)=0$.
	
%	\[
%	\nabla\Gamma_{\!\rho}(x)
%	=\left\{\begin{array}{cl}
%		(\psi^*)'(\rho\|F(x)\|_p)\frac{\nabla F(x)[{\rm sign}(F(x))\circ|F(x)|^{p-1}]}{\|F(x)\|_p^{p-1}}, &{\rm if} \ F(x)\ne 0;\\
%		0, &{\rm if} \ F(x)=0.  	
%	\end{array}\right.
%	\]
\end{lemma}
\begin{proof}
	According to the assumptions of $\phi$, it can be easily verified $\psi^*(\omega)=0$ for all $\omega\in(-\infty,\gamma]$. So $(\psi^*)'(\omega)=0$ holds for all $\omega\in(-\infty,\gamma)$. According to Lemma \ref{psi-star}, the function $\psi^*$ is continuously differentiable on $\mathbb{R}$. Fix any $x\in\mathbb{R}^n$. If $F(x)\ne 0$, since the function $x'\mapsto\|F(x')\|_p$ is continuously differentiable at $x$, so $\Gamma_{\!\rho}$ is continuously differentiable at $x$. Thus, it is only necessary to prove the case where $F(x)=0$.
	Since $F$ is differentiable at $x$, so $F(x+\Delta x)=F'(x)\Delta x+o(\|\Delta x\|)$ holds for any $\Delta x\to 0$, this means that for sufficiently small $\Delta x$, we obtain $\rho\|F(x+\Delta x)\|_p=\rho\|F'(x)\Delta x+o(\|\Delta x\|)\|_p\in[0,\gamma]$, and then $\psi^*(\rho\|F(x+\Delta x)\|_p)=0$. Thus, for arbitrary~$\Delta x\to 0$, it holds that $\Gamma_{\!\rho}(x\!+\!\Delta x)=\rho^{-1}\psi^*(\rho\|F(x\!+\!\Delta x)\|_p)=0=\Gamma_{\!\rho}(x)$, this indicates that $\Gamma_{\!\rho}$ is differentiable at $x$ and its derivative is $0$. Combining $(\psi^*)'(0)=0$ immediately obtains the conclusion.
\end{proof}

In the following, some common examples $\phi\in\!\mathscr{L}$ are given. Although these examples have appeared in reference \cite{Zhang-Pan-Bi-Sun23}, they are included here for the convenience of subsequent analysis and citation.

\begin{example}\label{example2.5.1}
	For every~$t\in\mathbb{R}$, define~$\phi(t):=t$. Obviously, $\phi\in\mathscr{L}$ and $t^*=0$. For arbitrary $\omega\in\mathbb{R}$,
	\[
	\psi^*(\omega)=\left\{\begin{array}{cl}
		0   & {\rm if}\ \omega\leq1,\\
		\omega-1 & {\rm if}\ \omega>1
	\end{array}\right.
	\]
	{\rm and}
	\[
	\varphi_{\rho}(\omega)=\left\{\begin{array}{cl}
		\omega   & {\rm if}\ \omega\leq1/\rho,\\
		1/\rho & {\rm if}\ \omega>1/\rho.
	\end{array}\right.
   \]
	It's easy to verify that $\phi\notin\mathscr{L}_1$, then it is acquired that $\psi^*$ is not continuously differentiable on $\mathbb{R}$, its nondifferentiable point is $\omega=1$.
\end{example}
\begin{example}\label{example2.5.2}
	For every $t\in\mathbb{R}$, define $\phi(t):=\frac{a-1}{a+1}t^2+\frac{2}{a+1}t\ (a>1)$. Obviously, $\phi\in\!\mathscr{L}$ and $t^*=0$. For arbitrary $\omega\in\mathbb{R}$, after simple calculation, it is obtained that
	\begin{equation}\label{part-scad}
		\psi^*(\omega)=\left\{\begin{array}{cl}
			0 & {\rm if}\ \omega\leq \frac{2}{a+1},\\
			\frac{((a+1)\omega-2)^2}{4(a^2-1)} & {\rm if}\ \frac{2}{a+1}<\omega\le\frac{2a}{a+1},\\
			\omega-1 & {\rm if}\ \omega>\frac{2a}{a+1}
		\end{array}\right.
	\end{equation}
	and
	\begin{equation}\label{scad}
		\varphi_{\rho}(\omega)=\left\{\begin{array}{cl}
			\omega & {\rm if}\ \omega\leq \frac{2}{\rho(a+1)},\\
			\omega-\frac{((a+1)\rho\omega-2)^2}{4\rho(a^2-1)} & {\rm if}\ \frac{2}{\rho(a+1)}<\omega\le\frac{2a}{\rho(a+1)},\\
			1/\rho & {\rm if}\ \omega>\frac{2a}{\rho(a+1)}.
		\end{array}\right.
	\end{equation}
	According to the Lemma \ref{psi-star}, $\varphi_{\rho}$ is continuously differentiable on $\mathbb{R}$ and $\varphi_{\rho}'$ is Lipschitz continuous. In addition, it can be easily verified that $\phi\in\mathscr{L}_2$ holds, and then the function $\Gamma_{\rho}$ of corresponding Lemma \ref{smoothness} is continuously differentiable.
\end{example}
\begin{example}\label{example2.5.3}
	For every $t\in\mathbb{R}$, define $\phi(t):=\frac{a^2}{4}t^2+\frac{2a-a^2}{2}t+\frac{(a-2)^2}{4}\ (a>2)$. Obviously, $\phi\in\mathscr{L}$ and $t^*=\frac{a-2}{a}$. After simple calculation, it is follows that
	\begin{equation}\label{part-MCP}
		\psi^*(\omega):=\!\left\{\!\begin{array}{cl}
			-\frac{(a-2)^2}{4} &{\rm if}\ \omega\leq a-\frac{a^2}{2},\\
			\frac{1}{a^2}(\frac{a(a-2)}{2}+\omega)^2-\frac{(a-2)^2}{4} &{\rm if}\ a-\frac{a^2}{2}<\omega\leq a,\\
			\omega-1 &{\rm if}\ \omega>a
		\end{array}\right.
	\end{equation}
	and
	\begin{equation}\label{MCP}
		\varphi_{\rho}(\omega):=\!\left\{\!\begin{array}{cl}
			\omega+\frac{(a-2)^2}{4\rho},{\rm if}\ \omega\le \frac{2a-a^2}{2\rho},\\
			\omega+\frac{(a-2)^2}{4\rho}\!-\!\frac{(\frac{a(a-2)}{2}+\rho\omega)^2}{a^2\rho},{\rm if}\ \frac{2a-a^2}{2\rho}<\omega\le \frac{a}{\rho} a,\\
			1/\rho,{\rm if}\ \omega>a/\rho.
		\end{array}\right.
	\end{equation}
	According to Lemma \ref{psi-star}, $\varphi_{\rho}$ is continuously differentiable on $\mathbb{R}$ and~$\varphi_{\rho}'$ is Lipschitz continuous. Obviously, it yields $\phi\notin\mathscr{L}_2$. It can be verified that the function $\omega\mapsto\varphi_{\rho}(|\omega|)$ is nondifferentiable at $\omega=0$. Therefore, the corresponding function $\Gamma_{\rho}$ of Lemma \ref{smoothness} is nondifferentiable at those points $x$ with $F(x)=0$.
\end{example}

\section{Proximal~MM method for equivalent DC surrogate problems}\label{section3}
Take any function $\phi\in\!\mathscr{L}_1$. By \cite[Theorem 3.3]{Pan-Liang-Liu2023}, if every corresponding multivalued function $\Upsilon_{\!s}$ is calm at the origin for all $x\in \!\Upsilon_ {\! S}(0)$, then for each $\rho\ge\overline{\rho}$, the following optimization problem
\begin{equation}\label{DC-Sprob1} 
	 \min_{x\in\mathcal{X}}\Theta_{\rho,\nu,\mu}(x)\!:=\vartheta(Ax-b)+\frac{\mu}{2}\|x\|^2+\!\rho\nu \sum_{i=1}^m\varphi_{\rho}(\|g_{\!_{J_i}}\!(x)\|)%+\mathbb{I}_{\mathcal{X}}(x)
\end{equation}
is the equivalent surrogate problem of group zero-norm regularized problem \eqref{prob1}, where $\varphi_{\rho}$ is defined in Lemma \ref{subdiff-vphirho}, and for each $i\in[m]$, since $x\mapsto\varphi_{\rho}(\|g_{\!_{J_i}}\!(x)\|)$ is a DC function, then \eqref{DC-Sprob1} with corresponding $\rho\ge\overline{\rho}$ is an equivalent DC surrogate problem of \eqref{prob1}. In this section, the authors propose a proximal MM method for solving DC problem \eqref{DC-Sprob1} to obtain the satisfying stationary point or local optimal solution of the group zero-norm regularized problem. For convenience, for arbitrary $x\in\mathbb{R}^p$, the authors define
\begin{equation}\label{wrho}
	w_{\rho}(x):=(w_{1,\rho}(x),\ldots,w_{m,\rho}(x))^{\top},
\end{equation}
 where $w_{i,\rho}(x):=(\psi^*)'(\rho\|g_{J_i}(x)\|),\ \ \forall i\in[m].$

\subsection{Proximal MM method for DC surrogate problem~\eqref{DC-Sprob1}}\label{sec3.1}

Fix arbitrary $x',x\in\mathbb{R}^p$. By the convexity of the function $\psi^*$ and mapping $G$ defined in \eqref{thetax}, it is obtained that
\begin{equation}\label{eq2}
	\sum_{i=1}^m\psi^*(\rho\|g_{\!_{J_i}}\!(x)\|)\geq\sum\limits_{i=1}^m\psi^*(\rho \|g_{\!_{J_i}}\!(x')\|)+\langle w_{\rho}(x'),\rho G(x)-\rho G(x')\rangle.
\end{equation}
By combining $\varphi_{\rho}$ and the expression of $\Theta_{\rho,\nu,\mu}$, it is acquired that
\begin{align*}
	&\Theta_{\rho,\nu,\mu}(x) \le\Xi_{\rho,\nu,\mu}(x,x')\\
	&\qquad \qquad :=\vartheta(Ax-b)+\frac{\mu}{2}\|x\|^2+\rho\nu\sum_{i=1}^m\|g_{J_i}(x)\|\\
	&\qquad \qquad -\rho\nu\langle w_{\rho}(x'),{G}(x)\rangle+\mathbb{I}_\mathcal{X}(x)+R_{\rho,\nu,\mu}(x'),
\end{align*}
where $R_{\rho,\nu,\mu}(x')\!=\rho\nu\langle w_{\rho}(x'),{G}(x')\rangle-\nu\sum_{i=1}^m\psi^*(\rho\|g_{J_i}(x')\|)$. 
Based on $\Xi_{\rho,\nu,\mu}(x',x')=\Theta_{\rho,\nu,\mu}(x')$, it can be seen that $\Xi_{\rho,\nu,\mu}(\cdot,x')$ is a majoriation of $\Theta_{\rho,\nu,\mu}(\cdot)$ at $x'$. This majorization is better than one induced by the convexity of each function $\Gamma_{i,\rho}(z):=\rho^{-1}\psi^*(\rho\|g_{\!_{J_i}}\!(z)\|)$. Indeed, according to \cite[Theorem 10.49]{RW98} and the smoothness of $g$, it is not difficult to obtain
\[
\partial\Gamma_{i,\rho}(x')=
\left\{\begin{array}{cl}
	\Big\{\frac{\nabla g_{J_{i}}(x')g_{J_i}(x')}{\|g_{J_i}(x')\|}(\psi^*)'(\rho\|g_{J_i}(x')\|)\Big\},g_{J_i}(x')\ne 0;\\
	\nabla g_{J_{i}}(x')\big[\bigcup_{\omega\in\partial\psi^*(0)}\partial(\omega\|\cdot\|)(0)\big],g_{J_i}(x')=0.
\end{array}\right.
\]
By using this expression and reviewing $g(x)=Bx$, for each $v^i\in\partial\Gamma_{i,\rho}(x')$, it can be verified that
\begin{equation}\label{ineq41-major}
	\langle v^i,x\rangle\le(\psi^*)'(\rho\|g_{J_i}(x')\|)\|g_{J_i}(x)\|=\langle w_{i,\rho}(x'),G_i(x)\rangle.
\end{equation}
According to the convexity of $\Gamma_{i,\rho}$, one can observe that $\Gamma_{i,\rho}(x)\ge\Gamma_{i,\rho}(x')+\langle v^i,x-x'\rangle$. Combining the expression of $\Theta_{\rho,\nu,\mu}$ yields that
\begin{align*}
	&\Theta_{\rho,\nu,\mu}(x)
	\le\widetilde{\Xi}_{\rho,\nu,\mu}(x,x')\\
	&\qquad \qquad:=\vartheta(Ax-b)+\frac{\mu}{2}\|x\|^2+\rho\nu\sum_{i=1}^m\|g_{J_i}(x)\|\\
	&\qquad \qquad-\rho\nu\sum_{i=1}^m\langle v^{i},x\rangle+\mathbb{I}_\mathcal{X}(x)+\widetilde{R}_{\rho,\nu,\mu}(x')
\end{align*}
where $\widetilde{R}_{\rho,\nu,\mu}(x')\!=\rho\nu\sum_{i=1}^m\langle v^{i},x'\rangle-\nu\sum_{i=1}^m\psi^*(\rho\|g_{J_i}(x')\|)$.
Since $\widetilde{\Xi}_{\rho,\nu,\mu}(x',x')
=\Theta_{\rho,\nu,\mu}(x')$, so~$\widetilde{\Xi}_{\rho,\nu,\mu}(\cdot,x')$ is also the majorization of $\Theta_{\rho,\nu,\mu}(\cdot)$ at $x'$. By \eqref{ineq41-major}, it follows that
\[
\Xi_{\rho,\nu,\mu}(x,x')+R_{\rho,\nu,\mu}(x')\le \widetilde{\Xi}_{\rho,\nu,\mu}(x,x')+\widetilde{R}_{\rho,\nu,\mu}(x').
\]
This indicates that the facts claimed above are valid. In view of this, the following uses the majorization $\Xi_{\rho,\nu,\mu}(\cdot,x^k)$ of $\Theta_{\rho,\nu,\mu}$ at current iterative $x^k$ to design an inexact proximal MM method for solving problem~\eqref{DC-Sprob1}, its basic idea is to seek the inexact minimum $x^{k+1}$ generated by proximal majorization function $x\mapsto\Xi_{\rho,\nu,\mu}(x,x^k)+\frac{1}{2}\|x\!-\!x^k\|^2_{\gamma_{1,k}I+\gamma_{2,k}A^{\top}A}$ of $\Theta_{\rho,\nu,\mu}$ at $x^k$, these inexact minimum points $x^{k+1}$ can be used to generate an iterative sequence $\{x^k\}_{k\in\mathbb{N}}$. The detailed iteration steps are described as follows.

\medskip
\setlength{\fboxrule}{0.8pt}
\noindent
%\fbox{
%	\parbox{0.96\textwidth}
	{
		%-------------------------------------------------------------------------------------------
		\begin{algorithm}\label{PMM}({\bf Inexact proximal MM method for~\eqref{DC-Sprob1}})
			\begin{description}
				\item[Initialization:] Choose $\widetilde{\lambda}>0,\widetilde{\gamma}_{1,0}>0,\widetilde{\gamma}_{2,0}>0$,
				seek an approximate optimal solution of the following problem
				\begin{align}\label{x0sub}					&x^{0}\approx\mathop{\arg\min}_{x\in \mathcal{X}}\bigg\{\vartheta(Ax-b)+\widetilde{\lambda}\langle e,G(x)\rangle	+\frac{\widetilde{\gamma}_{1,0}}{2}\|x\|^2\nonumber\\
				&\quad \quad \quad \quad\quad \quad+\frac{\widetilde{\gamma}_{2,0}}{2}\|Ax\|^2\bigg\}.
				\end{align}
				\hspace*{0.5cm} Choose $\phi\in\mathscr{L}_1,\,\mu>0,\,\nu>0,\rho\ge 1,\varrho\in(0,1],\underline{\gamma_1}>0,\underline{\gamma_2}>0,0<\gamma_{1,0}\le\widetilde{\gamma}_{1,0}$, $0<\gamma_{2,0}\le\widetilde{\gamma}_{2,0}$. Set $\lambda:=\rho\nu$ and $Q:=\underline{\gamma_1}I\!+\!\underline{\gamma_2}A^{\top}A$.
				
				\item[{\rm{\bf For}}] $k=1,2,\ldots$ {\rm{\bf do}}
				\begin{itemize}
					\item[{\rm S.1}] Compute $w^k=w_{\rho}(x^{k})$ and denote by $v^k=e-w^k$;
					
					\item[{\rm S.2}] Take error vector $\delta^k\in\mathbb{R}^p$ satisfying $\|\delta^k\|\le\frac{\|Q^{1/2}(x^k-x^{k-1})\|}{\sqrt{2}\|Q^{-1/2}\|}$.
					
					\item[{\rm S.3}] Set $Q_k=\gamma_{1,k}I\!+\!\gamma_{2,k}A^{\top}A$ and compute an approximate optimal solution $x^{k+1}$ of the strongly convex problem:
					\begin{align}\label{subprobk}
						&\min_{x\in \mathcal{X}}
						\Big\{\vartheta(Ax-b)+\frac{\mu}{2}\|x\|^2+\lambda\langle v^k,G(x)\rangle\nonumber\\
						&\quad \quad \quad +\frac{1}{2}\|x\!-\!x^k\|^2_{Q_k}-\langle\delta^k,x-x^k\rangle\Big\}.
					\end{align}
					
					\item[{\rm S.4}] Update $\gamma_{1,k}$ and $\gamma_{2,k}$ by $\gamma_{1,k+1}=\max(\underline{\gamma_1},\varrho\gamma_{1,k})$ and $\gamma_{2,k+1}=\max(\underline{\gamma_2},\varrho\gamma_{2,k})$.
				\end{itemize}
				\item[{\rm{\bf end~for}}]
			\end{description}
		\end{algorithm}
	}
%}

%\bigskip

%---------------------------------------------------------------------------------------
\begin{remark}\label{remark-PPM}
	{\bf(a)}~In the initialization step of the Algorithm \ref{PMM}, $\ell_{2,1}$-norm regularized minimization problem \eqref{x0sub} is solved approximately for providing satisfactory initial points. As shown in the following Proposition~\ref{initial-point}, such an initial point is good in a statistical sense.
	
	\noindent
	{\bf(b)}~The error vector $\delta^k$ in problem \eqref{subprobk} implies that $x^{k+1}$ is actually an inexact optimal solution of the following strongly convex programming
	\begin{equation}\label{Esubprobk}
		\min_{x\in \mathcal{X}}
		\Big\{\vartheta(Ax-b)+\frac{\mu}{2}\|x\|^2+\lambda\langle v^k,G(x)\rangle+\frac{1}{2}\|x\!-\!x^k\|^2_{Q_k}\Big\},
	\end{equation}
	where inexactness is reflected in
	\begin{align}\label{inclusion-xk}
		&\delta^k\in A^{\top}\partial\vartheta(Ax^{k+1}\!-b)+\mu x^{k+1}\!+\lambda\partial\langle v^k,G(\cdot)\rangle(x^{k+1})\nonumber\\
		&\quad+Q_k(x^{k+1}\!-\!x^k)+\mathcal{N}_{\mathcal{X}}(x^{k+1}).
	\end{align}
	Note that the error vector $\delta^k$ only depends on the information of the past two iterate points $x^{k-1}$ and $x^k$. Therefore, such inexact solutions without the need to be given beforehand can be automatically generated by solving the problem \eqref{Esubprobk}. Based on the definition of $w_{\rho}(\cdot)$ in \eqref{wrho}, $w_i^k\in[0,1]$ of step S.1 in Algorithm \ref{PMM} usually has an explicit expression. For example, when the function $\phi$ is taken from Example \ref{example2.5.2}, it is derived that
	\begin{equation}\label{wk-equa1}	w_i^k=\min\bigg\{1,\max\Big(0,\frac{(a+1)\rho\|g_{J_i}(x^k)\|-2}{2(a-1)}\Big)\bigg\}\quad\forall i\in[m];
	\end{equation}
	When $\phi$ is taken from Example \ref{example2.5.3}, it is attained that
	\begin{equation}\label{wk-equa2}		w_i^k=\min\bigg\{1,\max\Big(0,\frac{a-2}{a}+\frac{2\rho}{a^2}\|g_{J_i}(x^k)\|\Big)\bigg\}\quad\forall i\in[m].
	\end{equation}
	In this way, the main work in the for end loop of Algorithm \ref{PMM} is to find the approximate optimal solution of the subproblem \eqref{Esubprobk}.
	
	\noindent
	{\bf(c)}~The proximal term $\frac{1}{2}\|x\!-\!x^k\|^2_ {Q_k}$ in question \eqref{subprobk} has two functions: One ensures that the subproblems \eqref{subprobk} or \eqref{Esubprobk} are easy to solve; The other is to ensure that the iterative sequence generated by the algorithm has global convergence, as detailed in the convergence analysis in section \ref{subsec4.1.1}.
\end{remark}

Before conducting convergence analysis on the Algorithm \ref{PMM}, let's summarize the properties of functions $\Theta_{\rho,\nu,\mu}$ commonly used in the subsequent analysis.
%-----------------------------------------------------------------------------------------------
\begin{proposition}\label{pro-Thetarho}
	Pick $\phi\in\mathscr{L}_1$. Fix arbitrary $\nu>0$, $\mu>0$ and $\rho>0$. So, the following conclusions hold true.
	\begin{itemize}
		\item[(i)]~For each $i\in[m]$, $\Gamma_{i,\rho}(\cdot):=\rho^{-1}\psi^*(\rho\|g_{\!_{J_i}}\!(\cdot)\|)$ is finite convex function on $\mathbb{R}^p$ and
		\[
		\partial\Gamma_{i,\rho}(x)=
		\left\{\begin{array}{cl}
			\Big\{\frac{\nabla g_{J_{i}}(x)g_{J_i}(x)}{\|g_{J_i}(x)\|}(\psi^*)'(\rho\|g_{J_i}(x)\|)\Big\},
			{\rm if}\ g_{J_i}(x)\ne 0;\\
			\nabla g_{J_{i}}(x)\partial[(\psi^*)'(0)\|\cdot\|](0),{\rm if}\ g_{J_i}(x)=0
		\end{array}\right.\quad\forall x\in\mathbb{R}^p;
		\]
		If $\phi\in\mathscr{L}_2$, then $\Gamma_{i,\rho}$ is continuously differentiable and convex on $\mathbb{R}^p$, and for any $x\in\mathbb{R}^p$, it is derived that
		\[
		\nabla\Gamma_{i,\rho}(x)=
		\left\{\begin{array}{cl}
			\frac{\nabla g_{J_{i}}(x)g_{J_i}(x)}{\|g_{J_i}(x)\|}(\psi^*)'(\rho\|g_{J_i}(x)\|),{\rm if}\ g_{J_i}(x)\ne 0;\\
			0,{\rm if}\ g_{J_i}(x)=0.
		\end{array}\right.
		\]
		
		\item[(ii)]~The function $\Theta_{\rho,\nu,\mu}$ is a lower-bounded, coercive DC function on $\mathbb{R}^p$.
		
		\item[(iii)]~The limit subdifferental of function $\Theta_{\rho,\nu,\mu}$ at any point $x\in\mathbb{R}^p$ satisfies
		\begin{align*}
			\partial\Theta_{\rho,\nu,\mu}(x)
			&\subset A^{\top}\partial\vartheta(Ax\!-\!b)+\mu x+\rho\nu\sum_{i=1}^m\nabla g_{J_i}(x)S_i(x)\\
			&\quad-\rho\nu\sum_{i=1}^m(\psi^*)'(\rho\|g_{J_i}(x)\|)\nabla g_{J_{i}}(x)\Lambda_i(x)\big],
		\end{align*}
		where $S_i\!:\mathbb{R}^p\rightrightarrows\mathbb{R}^{|J_i|}$ and $\Lambda_i\!:\mathbb{R}^p\rightrightarrows\mathbb{R}^{|J_i|}$ for each $i\in[m]$ are set-valued mappings in the following form
		\begin{align*}
			S_i(x)&=\left\{\begin{array}{cl}
				\!\{\frac{g_{J_i}(x)}{\|g_{J_i}(x)\|}\} &{\rm if}\ g_{J_i}(x)\ne 0;\\
				\mathbb{B}_{\mathbb{R}^{|J_i|}} &{\rm if}\ g_{J_i}(x)= 0,
			\end{array}\right.\\
			\Lambda_i(x)&=\left\{\begin{array}{cl}
				\{\frac{g_{J_i}(x)}{\|g_{J_i}(x)\|}\} &{\rm if}\ g_{J_i}(x)\ne 0;\\
				\!\{v\in\mathbb{R}^{|J_i|}\ |\ \|v\|=1\} &{\rm if}\ g_{J_i}(x)= 0.
			\end{array}\right.
		\end{align*}
		In particular, if $\phi\in\mathscr{L}_2$, for any $x\in\mathbb{R}^p$, it holds that
		\begin{align*}
		&\partial\Theta_{\rho,\nu,\mu}(x)=A^{\top}\partial\vartheta(Ax\!-\!b)+\mu x+\rho\nu\sum_{i=1}^m\big[-\nabla\Gamma_{i,\rho}(x)\\
		&+\nabla g_{J_i}(x)S_i(x)\big].
	   \end{align*}
	\end{itemize}		
\end{proposition}
\begin{proof}
	{\bf(i)} For $z\in\mathbb{R}^l$, $h_\rho(z):=\rho^{-1}\psi^*(\rho\|z\|)$ is defined. By utilizing \cite[Theorem 10.49 \& Propersition 9.24 (b)]{RW98}, one can obtain
	\[
	\partial h_{\rho}(\overline{z})=
	\left\{\begin{array}{cl}
		\frac{\overline{z}}{\|\overline{z}\|}\partial\psi^*(\rho\|\overline{z}\|)
		&{\rm if}\ \overline{z}\ne 0;\\
		\bigcup_{\omega\in\partial\psi^*(0)}\partial(\omega\|\cdot\|)(0) &{\rm if}\ \overline{z}=0.
	\end{array}\right.
	\]
	Since $\phi\in\mathscr{L}_1$, by Lemma \ref{psi-star}, $\psi^*$ is a continuously differentiable convex function. Combining the above equation yields the conclusion.
	
	\noindent
	{\bf(ii)} Note that $\Theta_{\rho,\nu,\mu}(x)=\vartheta(Ax-b)+\frac{\mu}{2}\|x\|^2+\rho\nu\sum_{i=1}^m\big[\|g_{J_i}(x)\|-\Gamma_{i,\rho}(x)\big]$. Combining the conclusion of (i) and the lower-boundedness of $\vartheta$, it can be seen that the result holds.
	
	\noindent
	{\bf(iii)} Note that $\partial(-\Gamma_{i,\rho})(x)=-\partial_B\Gamma_{i,\rho}(x)\subset-(\psi^*)'(\rho\|g_{J_i}(x)\|)\nabla g_{J_{i}}(x)\Lambda_i$. By utilizing the finite-valued convexity of function $x\mapsto\vartheta(Ax-b)+\frac{\mu}{2}\|x\|^2+\rho\nu\sum_{i=1}^m\|g_{J_i}(x)\|$, the first part of the conclusion is obtained. If $\phi\in\mathscr{L}_2$, according to the Lemma \ref{smoothness}, the function $\Gamma_{i,\rho}$ is continuously differentiable for each $i\in[m]$. Thus, the required equation holds.
\end{proof}
%------------------------------------------------------------------------------------------------
\subsection{Convergence analysis of proximal~MM~method}\label{subsec4.1.1}

In order to analyze the convergence of Algorithm \ref{PMM}, for each given $\rho>1$, $\nu>0$ and $\mu>0$, this paper define the following potential function
\[
\Psi_{\rho,\nu,\mu}(x,y):=\Theta_{\rho,\nu,\mu}(x)+\frac{1}{4}\|x\!-\!y\|^2_{Q}\quad\ \forall x,y\in\mathbb{R}^p.
\]
The following lemma explains the properties of iterative sequence $\{x^k\}_{k\in\mathbb{N}}$ generated by potential function $\Psi_{\rho,\nu,\mu}$ in the corresponding Algorithm~\ref{PMM}.

%----------------------------------------------------------------------------------------------
\begin{lemma}\label{Psi-property}
	Let $\{x^k\}_{k\in\mathbb{N}}$ be generated by Algorithm \ref{PMM}, then $\{x^k\}_{k\in\mathbb{N}}\subset \mathcal{X}$, and for each $k\in\mathbb{N}$, it holds that
	\begin{itemize}
		\item[(i)] $\Psi_{\rho,\nu,\mu}(x^k,x^{k-1})\geq\Psi_{\rho,\nu,\mu}(x^{k+1},x^k)+\frac{1}{4}\|x^{k+1}-x^k\|^2_{Q}$;
		
		\item[(ii)] If $\phi\in\mathscr{L}_2$, then there exists $\zeta^k\in\partial\Psi_{\rho,\nu,\mu}(x^k,x^{k-1})$ such that $\|\zeta^k\|\leq c_1\|x^k-x^{k-1}\|+c_2\|x^{k-1}-x^{k-2}\|$, where $c_1>0$ and $c_2>0$ are constants independent of $k$.
	\end{itemize}
\end{lemma}
\begin{proof}
	According to \eqref{inclusion-xk} of Remark \ref{remark-PPM} (b), it can be inferred that $\{x^k\}_{k\in\mathbb{N}}\subset \mathcal{X}$. Fix any $k\in\mathbb{N}$.
	
	\noindent
	{\bf(i)} Substituting $x=x^{k+1}$ and $x'=x^k$ into the equation \eqref{eq2} yields
	\begin{align*}
	&\lambda\rho^{-1}\sum\limits_{i=1}^m\psi^*(\rho\|g_{J_i}(x^{k+1})\|)-\lambda\langle w^k,G(x^{k+1})\rangle\\
	&\geq\lambda\rho^{-1}\sum\limits_{i=1}^m\psi^*(\rho\|g_{J_i}(x^k)\|)-\lambda\langle w^k,G(x^k)\rangle.
   \end{align*}
	Based on the fact that $x^{k+1}$ is the unique optimal solution of the strongly convex optimization problem \eqref{subprobk} and the feasibility of $x^{k}$, it can be obtained that
	\begin{align*}
		&\vartheta(Ax^k-b)+\frac{1}{2}\mu\|x^k\|^2+\lambda\langle e-w^k,G(x^k)\rangle\\
		&\geq \vartheta(Ax^{k+1}-b)+\frac{\mu\|x^{k+1}\|^2}{2}+\lambda\langle e-w^k,G(x^{k+1})\rangle\\
		&\quad+\|x^{k+1}-x^k\|^2_{Q_k}-\langle\delta^k,x^{k+1}-x^k\rangle.
	\end{align*}
	Adding the above two inequalities and combining the definition of $\Theta_{\rho, \nu, \mu}$ can obtain
	\begin{align*}
		&\Theta_{\rho,\nu,\mu}(x^k)\\	&\ge\Theta_{\rho,\nu,\mu}(x^{k+1})+\|x^{k+1}-x^k\|^2_{Q_k}+\langle\delta^k,x^k-x^{k+1}\rangle\\	&\ge\Theta_{\rho,\nu,\mu}(x^{k+1})+\|x^{k+1}-x^k\|^2_{Q}+\langle\delta^k,x^k-x^{k+1}\rangle\\	&\ge\Theta_{\rho,\nu,\mu}(x^{k+1})+\|x^{k+1}-x^k\|^2_{Q}-\frac{1}{2}\|Q^{1/2}(x^{k+1}-x^k)\|^2\\
		&\quad-\frac{1}{2}\|Q^{-1/2}\delta^k\|^2\\	&\geq\Theta_{\rho,\nu,\mu}(x^{k+1})+\frac{1}{2}\|x^{k+1}-x^k\|^2_{Q}-\frac{1}{4}\|x^k-x^{k-1}\|^2_{Q},
	\end{align*}
	where the second inequality is due to the semi-positive definiteness of $Q_k-Q$, the third inequality utilizes Cauchy-Schwarz inequality and the last inequality utilizes the upper bound of $\|\delta_k\|$ in Algorithm \ref{PMM}. Combining the definition of $\Psi_{\rho,\nu,\mu}$, one can observe the result holds.
	
	\noindent
	{\bf(ii)}~Since $\phi\in\mathscr{L}_2$, according to the definition of $w_{\rho}$ in \eqref{wrho} and $(\psi^*)'(0)=0$, for each $i\notin{\rm gs}(g(x^k))$, we get $w_i^k=0$. Thus, combining the definition of $x^k$ and the optimality conditions of problem \eqref{subprobk} yields
	\begin{align}\label{temp-inclusion41}
		0&\in A^{\top}\partial\vartheta(Ax^k-b)+\mu x^k+\lambda\sum_{i\in{\rm gs}(g(x^k))}v_i^k\nabla g_{J_i}(x^k)\Lambda_i(x^k)\nonumber\\
		&+\mathcal{N}_\mathcal{X}(x^k)-\lambda\sum_{i\notin{\rm gs}(g(x^k))}\nabla g_{J_{i}}(x^k)\mathbb{B}_{\mathbb{R}^{|J_i|}}-\delta^{k-1}\nonumber\\
		&+Q_{k-1}(x^k-x^{k-1}).
	\end{align}
	In addition, according to Proposition \ref{pro-Thetarho} (iii), it can be inferred that
	\begin{align}\label{temp-equa41}
		&\partial\Theta_{\!\rho,\nu,\mu}(x^k)\nonumber\\
		&=A^{\top}\partial\vartheta(Ax^k-b)+\mu x^k+\lambda\sum_{i=1}^m\big[\nabla g_{J_i}(x^k)S_i(x^k)\nonumber\\
		&\quad\quad-\nabla\Gamma_{i,\rho}(x^k)\big]+\mathcal{N}_\mathcal{X}(x^k),\nonumber\\
		&=A^{\top}\partial\vartheta(Ax^k-b)+\mu x^k+\mathcal{N}_\mathcal{X}(x^k)\nonumber\\
		&\quad+\lambda\!\sum_{i\in{\rm gs}(g(x^k))}v_i^k\nabla g_{J_i}(x^k)S_i(x^k)\nonumber\\
		&\quad-\lambda\!\sum_{i\notin{\rm gs}(g(x^k))}\nabla g_{J_{i}}(x^k)\mathbb{B}_{\mathbb{R}^{|J_i|}},
	\end{align}
	where the second inequality is due to Proposition \ref{pro-Thetarho} (i), $v^k=e-w_{\rho}(x^k)$ and the definition of $w_{i,\rho}$. From \eqref{temp-inclusion41} and \eqref{temp-equa41}, it is obtained that
	$u^k:=\delta^{k-1}+Q_{k-1}(x^{k-1}-x^k)\in\partial\Theta_{\rho,\nu,\mu}(x^k)$. Then we have
	\begin{align*}
	&\zeta^k:=\Big(u^k+\frac{1}{2}Q(x^k-x^{k-1});\frac{1}{2}Q(x^{k-1}-x^k)\Big)\\
	&\quad\in\partial\Psi_{\rho,\nu,\mu}(x^k,x^{k-1}).
   \end{align*}
	It is noted that
	\begin{align*}
		\|\zeta^k\|&\leq\|\delta^{k-1}\|+\|Q_{k-1}\|\|x^k-x^{k-1}\|+\|Q\|\|x^k-x^{k-1}\|\\ &\leq\frac{\|Q^{1/2}\|\|x^{k-1}-x^{k-2}\|}{\sqrt{2}\|Q^{-1/2}\|}+(\|Q_0\|+\|Q\|)\|x^k-x^{k-1}\|,
	\end{align*}
	where the last inequality is due to $\|\delta^{k-1}\|\le\frac{\|Q^{1/2}\|\|x^{k-1}-x^{k-2}\|}{\sqrt{2}\|Q^{-1/2}\|}$ and $\|Q_{k-1}\|\le\|Q_0\|$. Pick $c_1=\|Q_0\|+\|Q\|$ and $c_2=\frac{\|Q^{1/2}\|}{\sqrt{2}\|Q^{-1/2}\|}$, the proof is completed.
\end{proof}
Lemma \ref{Psi-property} (i) implies that iterative sequence $\{\Psi_{\rho,\nu,\mu}(x^k,x^{k-1})\}_{k\in\mathbb{N}}$ is nonincreasing. It's noted from Proposition \ref{pro-Thetarho} (ii) that the function $\Psi_{\rho,\nu,\mu}$ is lower-bounded. Thus, the sequence $\{\Psi_{\rho,\nu,\mu}(x^k,x^{k-1})\}_{k\in\mathbb{N}}$ is convergent.
The following lemma mainly prove that every cluster point of the sequence $\{x^k\}_{k\in\mathbb{N}}$ generated by Algorithm \ref{PMM} is the critical point of function $\Theta_{\!\rho,\nu,\mu}$.
%-----------------------------------------------------------------------------------------------
\begin{lemma}\label{cluster-point-pmm}
	Denote the set composed of cluster points of sequence $\{x^k\}_{k\in\mathbb{N}}$ by $C(x^0)$. Then, we have the following results:
	\begin{itemize}
		\item[(i)]~$C(x^0)\subset \mathcal{X}$ is a nonempty compact set and $C(x^0)\subseteq{\rm crit}\Theta_{\!\rho,\nu,\mu}$;
		
		\item[(ii)]~$D:=\{(x,x)\ |x\in C(x^0)\}\subset{\rm crit}\Theta_{\!\rho,\nu,\mu}$ and $\lim_{k\rightarrow\infty}{\rm dist}((x^k,x^{k-1}),D)=0$;
		
		\item[(iii)]~The function $\Psi_{\rho,\nu,\mu}$  remains constant on set $D$. 	
	\end{itemize}
\end{lemma}
\begin{proof}
	{\bf(i)} According to Proposition \ref{pro-Thetarho} (ii), $\Theta_{\!\rho,\nu,\mu}$ is lower-bounded coercive. By the definition of $\Psi_{\rho,\nu,\mu}$, $\Psi_{\rho,\nu,\mu}$ is also lower-bounded coercive. Combining Lemma \ref{Psi-property} (i) can infer that $\{(x^k,x^{k-1})\}_{k\in\mathbb{N}}$ is bounded, then $C(x^0)$ is a nonempty compact set. Combining $\{x^k\}_{k\in\mathbb{N}}\subset \mathcal{X}$, it can be seen that $C(x^0)\subset \mathcal{X}$. Pick any $\overline{x}\in C(x^0)$, then there exists a subsequence $\{x^{k_l}\}_{l\in\mathbb{N}}$ such that $\lim_{l\rightarrow\infty}x^{k_l}=\overline{x}$. Combining Lemma \ref{Psi-property} (i) and the convergence of $\{\Psi_{\rho,\nu,\mu}(x^k,x^{k-1})\}_{k\in\mathbb{N}}$ yields $\lim_{l\rightarrow\infty}\|x^{k_l}\!-\!x^{k_l-1}\|=0$ and then $\lim_{l\rightarrow\infty}x^{k_l-1}=\overline{x}$. In addition, due to the proof of Lemma \ref{Psi-property}, it follows
	\[
	\delta^{k_{l}-1}+Q_{k_{l}-1}(x^{k_l-1}-x^{k_l})\in\partial\Theta_{\rho,\nu,\mu}(x^{k_l}).
	\]
	Since $\|\delta^{k_{l}-1}\|\le\frac{\|Q^{1/2}\|\|x^{k_l-1}-x^{k_l-2}\|}{\sqrt{2}\|Q^{-1/2}\|}$ and $\lim_{l\rightarrow\infty}\|x^{k_l}-x^{k_l-1}\|=0$, we have $\lim_{l\rightarrow\infty}\delta^{k_{l}-1}=0$. By using the above inclusion and outer semi-continuity of subdifferential mapping
	$\partial\Theta_{\rho,\nu,\mu}$, it is derived that $0\in\partial\Theta_{\rho,\nu,\mu}(\overline{x})$. According to the arbitrariness of $\overline {x}\in C(x^0)$, the required inclusion holds.
	
	\noindent
	{\bf(ii)} From (i) and the definition of the set~$D$, the conclusion holds.
	
	\noindent
	{\bf(iii)} Denote $\upsilon^*=\lim_{k\to\infty}\Psi_{\rho,\nu,\mu}(x^k,x^{k-1})$. Pick any $(\overline{x},\overline{x})\in D$, we first demonstrate below $\upsilon^*=\Psi_{\rho,\nu,\mu}(\overline{x},\overline{x})=\Theta_{\!\rho,\nu,\mu}(\overline{x})$, and then the function $\Psi_{\rho,\nu,\mu}$ keeps unchanged on $D$. It can be inferred from $\overline{x}\in C(x^0)$ that there exists a subsequence $\{x^{k_l}\}_{l\in\mathbb{N}}$ such that $\lim_{l\rightarrow\infty}x^{k_l}=\overline{x}$. It can be obtained obviously that $\lim_{l\rightarrow\infty}\Psi_{\!\rho,\nu,\mu}(x^{k_l},x^{k_l-1})=\upsilon^*$. By the proof of (ii), it yields $\lim_{l\rightarrow\infty}\|x^{k_l}-x^{k_l-1}\|=0$. Due to the lower semicontinuity of $\Theta_{\!\rho,\nu,\mu}$ and the definition of~$\Psi_{\!\rho,\nu,\mu}$, one can obtain $v^*=\liminf_{l\rightarrow\infty}\Psi_{\!\rho,\nu,\mu}(x^{k_l})\geq\Theta_{\!\rho,\nu,\mu}(\overline{x})$. Thus, it suffices to demonstrate $v^*\le\Theta_{\!\rho,\nu,\mu}(\overline{x})$. Since $x^{k_l}$ is the optimal solution to the problem \eqref{subprobk} and~$\overline{x}$ is a feasible solution of~\eqref{subprobk}, it holds that
	\begin{align*}
		&\vartheta(Ax^{k_l}\!-\!b)+\frac{\mu}{2}\|x^{k_l}\|^2+\lambda\langle v^{k_l-1},G(x^{k_l})\rangle+\\
		&\frac{1}{2}\|x^{k_l}\!-\!x^{k_l-1}\|^2_{Q_{k_l}-1}-\langle\delta^{k_l-1},x^{k_l}\!-\!x^{k_l-1}\rangle\\
		&\le\vartheta(A\overline{x}\!-\!b)+\frac{\mu}{2}\|\overline{x}\|^2+\lambda\langle v^{k_l-1},G(\overline{x})\rangle\\
		&\quad+\frac{1}{2}\|\overline{x}-x^{k_l-1}\|^2_{Q_{k_l-1}}-\langle\delta^{k_l-1},\overline{x}-x^{k_l-1}\rangle.
	\end{align*}
	Combining $\lambda=\rho\nu$ and the definition of the function $\Theta_{\!\rho,\nu,\mu}$, the above inequality can be equivalently expressed as
	\begin{align*}
		&\Theta_{\!\rho,\nu,\mu}(x^{k_l})-\lambda\langle w_{\rho}(x^{k_l-1}),G(x^{k_l})\rangle\\
		&+\lambda\rho^{-1}\sum_{i=1}^m\psi^*(\rho\|g_{\!_{J_i}}\!(x^{k_l})\|)+\frac{1}{2}\|x^{k_l}-x^{k_l-1}\|^2_{Q_{k_l}-1}\\
		&\leq\Theta_{\!\rho,\nu,\mu}(\overline{x})-\lambda\langle w_{\rho}(x^{k_l-1}),G(\overline{x})\rangle+\lambda\rho^{-1}\sum_{i=1}^m\psi^*(\rho\|g_{\!_{J_i}}\!(\overline{x})\|)\\
		&\quad\quad+\frac{1}{2}\|\overline{x}-x^{k_l-1}\|^2_{Q_{k_l-1}}+\langle\delta^{k_l-1},x^{k_l}-\overline{x}\rangle.
	\end{align*}
	Taking the limit as $l\rightarrow\infty$ on both sides of this inequality, using the continuity of functions $w_{\rho}(\cdot)$ and $\sum_{i=1}^m\rho^{-1}\psi^*(\rho\|g_{\!_{J_i}}\!(\cdot)\|)$ and combining the boundedness of $\{Q_k\}_{k\in\mathbb{N}}$ and $\{\delta^k\}_{k\in\mathbb{N}}$ yield $\upsilon^*=\limsup_{l\rightarrow\infty}\Psi_{\!\rho,\nu,\mu}(x^{k_l})=\limsup_{l\rightarrow\infty}\Theta_{\!\rho,\nu,\mu}(x^{k_l})\leq\Theta_{\!\rho,\nu,\mu}(\overline{x})$. As a consequence, it is acquired that $\upsilon^*=\Theta_{\!\rho,\nu,\mu}(\overline{x})=\Psi_{\!\rho,\nu,\mu}(\overline{x})$ and the proof is completed.
\end{proof}

When $\Theta_{\!\rho,\nu,\mu}$ is definable in an o-minimal structure over $(\mathbb{R},+,\cdot)$, then it yields from \cite[Section~4]{Attouch10} that $\Psi_{\!\rho,\nu,\mu}$ is a KL function. In addition, if $\Theta_{\!\rho,\nu,\mu}$ has KL property of exponent $\frac{1}{2}$ at $\overline{x}\in{\rm dom}\partial\Theta_{\!\rho,\nu,\mu}$, then we obtain from \cite[Theorem~3.6]{LiPong18} that the function $\Psi_{\rho,\nu,\mu}$ has also KL property of exponent $\frac{1}{2}$ at $(\overline{x},\overline{x})$. In this way, according to Lemma \ref{Psi-property} and Lemma \ref{cluster-point-pmm}, the following convergence conclusions for Algorithm \ref{PMM} can be obtained by using an analysis method similar to \cite[Theorem~1]{Bolte14} and \cite[Theorem~2]{Attouch-Bolte09}.
%--------------------------------------------------------------------------------------------------
\begin{theorem}
	Let $\{x^k\}_{k\in\mathbb{N}}$ be the sequence generated by Algorithm \ref{PMM}. Then, we have the following results:
	\begin{itemize}
		
		\item[(i)] If $\Theta_{\!\rho,\nu,\mu}$ is definable in an o-minimal structure over some real field, then it holds that $\sum_{k=1}^{\infty}\|x^{k+1}-x^k\|<\infty$, and then the sequence $\{x^k\}_{k\in\mathbb{N}}$ is convergent and its limit is a critical point of function $\Theta_{\!\rho,\nu,\mu}$.
		
		\item[(ii)] If $\Theta_{\rho,\nu,\mu}$ has KL property of exponent $\frac{1}{2}$ at $\overline{x}=\lim_{k\to\infty}x^k$, then the sequence $\{x^k\}_{k\in\mathbb{N}}$ is locally convergent at R-linear rate, that is to say, there exist $\overline{\nu}>0$ and $q\in(0,1)$ such that $\|x^k-\overline{x}\|\leq\overline{\nu} q^k$ holds.	
	\end{itemize}
\end{theorem}
\begin{remark}
	
	{\bf(a)} By the expression for $\Theta_{\! \rho, \nu, \mu}$, if the functions $\vartheta$ and $\psi$ are definable functions in an o-minimal structure over the same real field, and the set $\mathcal{X}$ is also definable set in o-minimum structure over the real field, then the function $\Theta_{\!\rho,\nu,\mu}$ must be definable function in the o-minimum structure over the real field. Thus, $\Psi_{\! \rho,\nu,\mu}$ is a KL function.
	
	\noindent
	{\bf(b)} According to \cite[Theorem 3.4]{Pan-Liang-Liu2023}, under certain assumptions, the critical point of the function $\Theta_ {\! \rho,\nu,\mu}$ is not only a weak stationary point of problem \eqref{DC-Sprob1} and it is also the locally optimal solution of problem \eqref{prob1}.
	
\end{remark}

\section{Statistical error bound for proximal MM method}\label{section4}
In this section, the statistical error bounds of the iterative sequence $\{x^k\}_{k\in\mathbb{N}}$ and its cluster point $\overline{x}$ generated by the Algorithm \ref{PMM} to the true solution $x^*$ are established in the case of identity mapping $g$, respectively. We suppose that the vector $b$ is obtained through the observation model \eqref{observation model}, and denote $\overline{r}=|\overline{S}|$, where $\overline{S}:={\rm gs}(x^*)=\{i:x^*_{J_i}\neq0\}$ represents the group support set of $x^*$. Since this paper focuses on the case where the sample size $n$ is less than the feature dimension $p$, namely $n<p$, the sample covariance matrix $\frac{1}{n}A^{\top}A$ is generally not positive definite, but it may be positive definite on a subset of $\mathbb{R}^p$. \cite[Section 3]{Pan-Liang-Liu2023} proves that the cluster point $\overline{x}$ of the iterative sequence $\{x^k\}_{k\in\mathbb{N}}$ generated by the Algorithm \ref{PMM} is a weakly stationary point of problem \eqref{DC-Sprob1}, and it has been proven that under certain assumptions of restricted eigenvalue conditions, this stationary point is a weakly stationary point of problem \eqref{prob1}, and it is also the local optimal solution of problem \eqref{prob1}. The subsequent analysis will be conducted under the Assumption \ref{ass0}, where for any given positive integer $r$ and constant $c>2$, the authors define
\begin{align*}
&\mathcal{C}(\overline{S},r):=\Big\{x\in\mathbb{R}^p\ |\ \exists  S\supset\overline{S}~{\rm and} ~|S|\leq r~{\rm such~that}\\
&\quad\quad\quad~\sum_{i\in S^c}\|x_{J_i}\|\leq\frac{2(c+1)}{c-2}\sum_{i\in S}\|x_{J_i}\|\Big\}.
\end{align*}
\begin{assumption}\label{ass0}
	There exists a constant $\kappa>0$ such that $\frac{1}{2n}\|Ax\|^2\geq\kappa\|x\|^2$ holds for all $x\in\mathcal{C}(\overline{S},1.5\overline{r})$.
\end{assumption}

%-----------------------------------------------------------------------------------------------
\subsection{The case where the function~$\vartheta$ is strongly square convex}\label{sec4.2.1}
This subsection will consider the case that $\vartheta$ is a finite-valued convex function and its square is a strongly convex function with a modulus of $\overline{\rho}$. Firstly, the authors review the definition of strongly convex functions and then provide the definition of squared strongly convex functions.

\begin{definition}\label{strongly-convex1}
	Suppose $f:\mathbb{R}^p\to\overline{\mathbb{R}}$, if there exists a constant $\sigma>0$ such that $x\mapsto f(x)-\frac{\sigma}{2}\|x\|^2$ is a convex function, then $f$ is called a strongly convex function with a modulus of $\sigma$, where $\sigma$ is called strongly convex parameter of $f$.
\end{definition}

\begin{definition}\label{strongly-convex2}
	Suppose $f:\mathbb{R}^p\to\overline{\mathbb{R}}$, if the function $f^2$ is a strongly convex function, then $f$ is called a squared strongly convex function.
\end{definition}

\begin{assumption}\label{ass1}
	{\bf(i)} $\vartheta\!:\mathbb{R}^n\to\mathbb{R}$ is finite-valued convex and globally Lipschitz continuous on the set~$\mathcal{Z}\!:=\{Az-b\,|\,z\in \mathcal{X}\}$ with Lipschitz modulus~$L_{\vartheta}$, it simultaneously satisfies
	\begin{align}\label{ineq-ass1}
		&\vartheta(0)\le 0,\ \vartheta(u+v)\le\vartheta(u)+\vartheta(v)\nonumber\\
		&{\rm and}\ \vartheta(-u)=\vartheta(u)\quad\ \forall u,v\in\mathbb{R}^n.
	\end{align}
	
	\noindent
	{\bf(ii)} $\vartheta^2$ is a strongly convex function with modulus~$\overline{\rho}$.
\end{assumption}
\begin{remark}\label{remark-ass1}
	{\bf(a)} According to \eqref{ineq-ass1}, it can be easily verified that $\vartheta$ is nonnegative and $\vartheta(0)=0$.
	
	\noindent	
	{\bf(b)} Since $\vartheta^2$ is a finite-valued and strongly convex function, so $\partial\vartheta^2(z)\ne\emptyset$ holds for any $z\in\mathbb{R}^p$. By \cite[Theorem 10.49]{RW98}, we have $\partial[\vartheta^2(z)]=2\vartheta(z)\partial\vartheta(z)$ for any $z\in\mathbb{R}^p$. Combining $\vartheta(0)=0$ yields that the subdifferential of $\vartheta^2$ at $z=0$ is a single-valued set $\{0\}$, then we know the function $\vartheta^2$ is differentiable at $z=0$.
	
	\noindent
	{\bf(c)} According to the Assumptions \ref{ass1} (ii) and $\vartheta(0)=0$, it holds that $\vartheta^2(z)\ge\frac{1}{2}\overline{\rho}\|z\|^2>0$ for any $z\ne 0$.
\end{remark}
%-----------------------------------------------------------------------------------------------
\begin{lemma}\label{lemma-vtheta}
	If $\vartheta\!:\mathbb{R}^n\to\mathbb{R}$ satisfies Assumption \ref{ass1}, then for arbitrary $z\in \mathcal{Z}\backslash\{0\}$ and $\xi\in\partial\vartheta^2(z)$, it holds that
	\[
	\|A^{\top}\xi\|/\vartheta(z)\le 2L_{\vartheta}\|A\|.
	\]
\end{lemma}
\begin{proof}
	Fix any $z\in \mathcal{Z}\backslash\{0\}$ and $\xi\in\partial\vartheta^2(z)$. According to Remark \ref{remark-ass1} (b), there exists $v\in\partial\vartheta(z)$ such that $\xi=2\vartheta(z)v$, then $A^{\top}\xi=2\vartheta(z)A^{\top}v$. Since $\vartheta$ is finite-valued convex and also globally Lipschitz continuous with modulus $L_{\vartheta}$ on $\mathcal{X}$, so due to \cite[Theorem 9.13]{RW98}, $\|\zeta\|\le L_{\vartheta}$ holds for any $\zeta\in\partial\vartheta(z)$ and then $\|A^{\top}\xi\|\le 2\vartheta(z)\|A^{\top}v\|\le2\vartheta(z)L_{\vartheta}\|A\|$ holds. According to Remark \ref{remark-ass1} (iii), it yields $\vartheta(z)>0$. The proof is completed.
\end{proof}

For convenience, for all $k\in\mathbb{N}$, denote $S^{k}\!:=\{i\in[m]\ |\ x^{k}_{J_i}\ne 0\},\,z^k\!:=Ax^k-b$, $\Delta x^k=x^k-x^*$ and $\xi^k:=Q_{k-1}(x^{k-1}-x^k)+\delta^{k-1}\!-\!\mu x^*$.

In order to establish the statistical error bound from iterative sequence $\{x^k\}_{k\in\mathbb{N}}$ generated by Algorithm \ref{PMM} and its cluster point $\overline{x}$ to the true solution $x^*$, this article gives a lemma, which demonstrates $\sum_{i\in (S^{k-1})^c}\|\Delta x^k_{J_i}\|$ can be bounded by $\sum_{i\in S^{k-1}}\|\Delta x^k_{J_i}\|$.
%------------------------------------------------------------------------------------------
\begin{lemma}\label{compatity condition}
	Suppose $\vartheta\!:\mathbb{R}^n\to\mathbb{R}$ satisfies the Assumption \ref{ass1} and $S^{k-1}\supset\overline{S}$ and
	$\max_{i\in(S^{k-1})^c}w_i^{k-1}\le\frac{1}{2}$ hold for some $k\in\mathbb{N}$. If $\lambda\ge c(\|{G}(\xi^k)\|_\infty+2L_{\vartheta}\|A\|)$, where $c>2$ is a constant, then it holds that
	\[
	\sum_{i\in (S^{k-1})^c}\|\Delta x^k_{J_i}\|\leq\frac{2(c+1)}{c-2}\sum_{i\in S^{k-1}}\|\Delta x^k_{J_i}\|.
	\]
\end{lemma}
\begin{proof}
	Since $x^k$ is a unique optimal solution to the strongly convex subproblem \eqref{subprobk} and $x^*\in \mathcal{X}$ is a feasible solution of this problem, so it follows that
	\begin{align*}
		&\vartheta(Ax^*\!-b)+\frac{\mu}{2}\|x^*\|^2+\lambda\langle v^{k-1},{G}(x^*)\rangle\\
		&+\frac{1}{2}\|x^*-x^{k-1}\|^2_{Q_{k-1}}-\langle\delta^{k-1},x^*-x^{k-1}\rangle\\
		&\geq\vartheta(Ax^k-b)+\frac{\mu}{2}\|x^k\|^2+\lambda\langle v^{k-1},{G}(x^k)\rangle\\
		&\quad+\frac{1}{2}\|x^k-x^{k-1}\|^2_{Q_{k-1}}-\langle\delta^{k-1},x^k-x^{k-1}\rangle\\
		&\quad+\frac{1}{2}\langle x^*-x^k,(\mu I+Q_{k-1})(x^*-x^k)\rangle.
	\end{align*}
	Adjusting the above inequalities appropriately and utilizing $\xi^k=Q_{k-1}(x^{k-1}-x^k)+\delta^{k-1}\!-\!\mu x^*$ yield that
	\begin{align}\label{ineq41-ebound}
		&\vartheta(Ax^k-b)-\vartheta(Ax^*-b)\nonumber\\
		&\leq\frac{\mu}{2}\|x^*\|^2-\frac{\mu}{2}\|x^k\|^2+\lambda\langle   v^{k-1},{G}(x^*)-{G}(x^k)\rangle\nonumber\\	&\quad+\frac{1}{2}\|x^*-x^{k-1}\|^2_{Q_{k-1}}-\frac{1}{2}\|x^k-x^{k-1}\|^2_{Q_{k-1}}\nonumber\\
		&\quad-\frac{1}{2}\langle x^*-x^k,(\mu I+Q_{k-1})(x^*-x^k)\rangle+\langle\delta^{k-1},x^k-x^*\rangle\nonumber\\
		&=-\mu\|\Delta x^k\|^2+\lambda\langle v^{k-1},{G}(x^*)-{G}(x^k)\rangle\nonumber\\
		&\quad\quad+\langle\xi^k,x^k-x^*\rangle,
	\end{align}
	where equality is due to $\frac{1}{2}[\|x^*\|^2-\|x^k\|^2]=-\frac{1}{2}\|x^k-x^*\|^2-\langle x^k-x^*,x^*\rangle$. According to the Assumption \ref{ass1} (ii) and Remark \ref{remark-ass1} (b), it holds that $\partial\vartheta(Ax^*\!-b)\ne\emptyset$. Pick any $\zeta\in\partial\vartheta^2(Ax^*-b)$, then it is obtained that
	\[
	\vartheta^2(Ax^k-b)-\vartheta^2(Ax^*-b)\ge\langle A^{\top}\zeta, x^k-x^*\rangle
	+\frac{1}{2}\overline{\rho}\|A(x^k-x^*)\|^2.
	\]
	By combining $z^k\!=Ax^k-b$, $Ax^*-b=\varpi\ne 0$ and $\vartheta(\varpi)>0$, it is obtained that
	\begin{align}\label{ineq42-ebound}
		&\vartheta(Ax^k-b)-\vartheta(Ax^*-b)\nonumber\\	&=\frac{\vartheta^2(Ax^k-b)-\vartheta^2(Ax^*-b)}{\vartheta(Ax^k-b)+\vartheta(Ax^*-b)}\nonumber\\
		&\ge\frac{\langle A^{\top}\zeta,x^k-x^*\rangle+0.5\overline{\rho}\|A(x^k-x^*)\|^2}{\vartheta(z^k)+\vartheta(\varpi)}\nonumber\\
		&\ge \frac{-\langle G(A^{\top}\zeta),G(\Delta x^k)\rangle+0.5\overline{\rho}\|A(x^k-x^*)\|^2}{\vartheta(z^k)+\vartheta(\varpi)}\nonumber\\
		&\ge\frac{-\langle G(A^{\top}\zeta),G(\Delta x^k)\rangle}{\vartheta(\varpi)}
		+\frac{0.5\overline{\rho}\|A(x^k-x^*)\|^2}{\vartheta(z^k)+\vartheta(\varpi)},
	\end{align}
	where inequality is due to the nonnegativity of function $\vartheta$. Combining \eqref{ineq42-ebound} and \eqref{ineq41-ebound} yields
	\begin{align}\label{ineq-errbound}
		&\mu\|\Delta x^k\|^2+\frac{0.5\overline{\rho}\|A\Delta x^k\|^2}{\vartheta(z^k)+\vartheta(\varpi)}\nonumber\\
		&\le\frac{1}{\vartheta(\varpi)}\langle G(A^{\top}\zeta),G(\Delta x^k)\rangle+\lambda\langle v^{k-1},{G}(x^*)-G(x^k)\rangle\nonumber\\
		&\quad+\langle\xi^k,\Delta x^k\rangle\nonumber\\
		&\le\frac{1}{\vartheta(\varpi)}\langle{G}(A^{\top}\zeta),{G}(\Delta x^k)\rangle+\sum_{i=1}^m\langle\xi^k_{J_i},\Delta x^k_{J_i}\rangle\nonumber\\
		&\quad+\lambda\sum_{i=1}^mv_i^{k-1}(\|x^*_{J_i}\|-\|x^k_{J_i}\|)\nonumber\\
		&\le\frac{\langle{G}(A^{\top}\zeta),{G}(\Delta x^k)\rangle}{\vartheta(\varpi)}+\langle{G}(\xi^k),{G}(\Delta x^k)\rangle\nonumber\\
		&\quad+\lambda\sum_{i\in\overline{S}}v_i^{k-1}(\|x^*_{J_i}\|-\|x^k_{J_i}\|)-\lambda\sum_{i\in\overline{S}^c}v_i^{k-1}\|x^k_{J_i}\|\\
		&\leq\frac{\langle{G}(A^{\top}\zeta),{G}(\Delta x^k)\rangle}{\vartheta(\varpi)}+\lambda\sum_{i\in\overline{S}}v_i^{k-1}\|\Delta x^k_{J_i}\|\nonumber\\
		&\quad-\lambda\sum_{i\in (S^{k-1})^c}v_i^{k-1}\|\Delta x^k_{J_i}\|+\langle{G}(\xi^k),{G}(\Delta x^k)\rangle,\nonumber
	\end{align}
	where the third inequality is due to the fact that $x^*_{J_i}=0$ holds for all $i\in\overline{S}^c$, the forth inequality is obtained based on $S^{k-1}\supset\overline{S}$ and $\|x^*_{J_i}\|-\|x^k_{J_i}\|\le\|x^*_{J_i}-x^k_{J_i}\|=\|\Delta x^k_{J_i}\|$. It can be inferred from $\max_{i\in(S^{k-1})^c}w_i^{k-1}\le\frac{1}{2}$ that $\min_{i\in(S^{k-1})^c}v_i^{k-1}\ge\frac{1}{2}$. At this point, utilizing $\min_{i\in(S^{k-1})^c}v_i^{k-1}\ge\frac{1}{2}$ yields
	\begin{align*}
		0&\le\mu\|\Delta x^k\|^2+\frac{0.5\overline{\rho}\|A\Delta x^k\|^2}{\vartheta(z^k)+\vartheta(\varpi)}\\
		&\leq\frac{1}{\vartheta(\varpi)}\bigg[\sum_{i\in S^{k-1}}\|{G}(A^{\top}\zeta)\|_\infty\|\Delta x^k_{J_i}\|\\
		&\quad+\|{G}(A^{\top}\zeta)\|_\infty\sum_{i\in(S^{k-1})^c}\|\Delta x^k_{J_i}\|\bigg]\\
		&\quad+\lambda\sum_{i\in\overline{S}}\|\Delta x^k_{J_i}\|-\frac{\lambda}{2}\!\sum_{i\in (S^{k-1})^c}\|\Delta x^k_{J_i}\|\\
		&\quad+\|{G}(\xi^k)\|_\infty\bigg[\sum_{i\in S^{k-1}}\!\|\Delta x^k_{J_i}\|+\!\sum_{i\in(S^{k-1})^c}\!\|\Delta x^k_{J_i}\|\bigg].
	\end{align*}
	Reviewing $\zeta\in\partial\vartheta^2(Ax^*\!-b)$ and $\varpi\in \mathcal{Z}\backslash\{0\}$ and combining Lemma \ref{lemma-vtheta}, it holds that  $\|{G}(A^{\top}\zeta)\|_\infty/\vartheta(\varpi)\le 2L_{\vartheta}\|A\|$. Thus, one can observe that
	\begin{align*}	&\big[{\lambda}/{2}-2L_{\vartheta}\|A\|-\|{G}(\xi^k)\|_\infty\big]\sum_{i\in(S^{k-1})^c}\|\Delta x^k_{J_i}\|\\
		&\leq\frac{1}{\vartheta(\varpi)}\sum_{i\in S^{k-1}}\|{G}(A^{\top}\zeta)\|_\infty\|\Delta x^k_{J_i}\|+\lambda\sum_{i\in\overline{S}}\|\Delta x^k_{J_i}\|\\
		&\quad+\|{G}(\xi^k)\|_\infty\sum_{i\in S^{k-1}}\|\Delta x^k_{J_i}\|\\
		&\leq \big(\|{G}(\xi^k)\|_\infty+2L_{\vartheta}\|A\|+\lambda\big)\sum_{i\in S^{k-1}}\|\Delta x^k_{J_i}\|,
	\end{align*}
	which combining $\lambda\ge c(\|{G}(\xi^k)\|_\infty+2L_{\vartheta}\|A\|)$ can derive the conclusion.
\end{proof}
\begin{remark}\label{compatity}
	In Lemma \ref{compatity condition}, the specific value of the constant $c>2$ is related to the function $\phi$ defined in \eqref{phi-assump}, as shown in Example \ref{example2.5.2} and Example \ref{example2.5.3}. According to the proof in \cite[Proposition 4.1]{Zhang-Pan-Bi-Sun23}, there exists $t_0\in[0,1)$ such that~$\frac {1}{1-t^*}\in\partial\phi(t_0)$, and according to literature~\cite[Lemma 1]{LiuBiPan18}, such~$t_0$ exists. When $\phi$ is taken from Example \ref{example2.5.2}, $t^*=0$, then it yields from $1=\frac{2(a-1)t_0}{a+1}+\frac{2}{a+1}$ that $t_0=\frac{1}{2}$ is obtained. In addition, we get $c=5$ due to $\frac{2}{1-t_0}=\frac{2(c+1)}{c-2}$. Similarly, if $\phi$ is taken from Example \ref{example2.5.3}, then $t_0=\frac{a-1}{a}$ and $c=\frac{2a+1}{a-1}$, where $a>2$. At this point, it is necessary to select the appropriate $c$ to minimize the statistical error bound.
\end{remark}
The following theorem characterizes the statistical error bound from each iterate point $x^k$ generated by Algorithm \ref{PMM} to the true solution $x^*$ of problem \eqref{prob1}.
%--------------------------------------------------------------------------------------------
\begin{theorem}\label{error bound}
	Suppose $A$ satisfies the Assumption \ref{ass0} on the set $\mathcal{C}(\overline{S},1.5\overline{r})$ and for some $k\in\mathbb{N}$, $S^{k-1}\supset\overline{S}$, $|S^{k-1}|\le 1.5\overline{r}$ and $\max_{i\in(S^{k-1})^c}w_i^{k-1}\le\frac{1}{2}$ hold. Then, when $\lambda\in[c(\|{G}(\xi^k)\|_\infty+2L_{\vartheta}\|A\|),\frac{2\mu \vartheta(\varpi)+p\kappa\overline{\rho}-\frac{3c}{c-2}L_{\vartheta}\|A\|_{2,\infty}(L_{\vartheta}\|A\|+\|{G}(\xi^k)\|_\infty)|S^{k-1}|}{\frac{3c}{c-2}L_{\vartheta}\|A\|_{2,\infty}\|v_{\overline{S}}^{k-1}\|_{\infty}|S^{k-1}|})$, where $c>2$ is a constant, it holds that $\|\Delta x^k\|\leq\frac{d_1}{d_2}$,
	%\begin{align*}
		%\|\Delta x^k\|\leq\frac{d_1}{d_2},
	%\end{align*}
where $d_1:=2\vartheta(\varpi)(L_{\vartheta}\|A\|+\|{G}(\xi^k)\|_\infty+\lambda\|v^{k-1}_{\overline{S}}\|_\infty)\sqrt{|S^{k-1}|}$ and $d_2:=2\mu \vartheta(\varpi)+n\kappa\overline{\rho}-\frac{3c}{c-2}L_{\vartheta}\|A\|_{2,\infty}(L_{\vartheta}\|A\|+\|{G}(\xi^k)\|_\infty+\lambda\|v^{k-1}_{\overline{S}}\|_\infty)|S^{k-1}|.$
\end{theorem}
\begin{proof}
	According to \eqref{ineq-errbound} of Lemma \ref{compatity condition} and $S^{k-1}\supset\overline{S}$, it can be inferred that
	\begin{align}\label{ineq43-ebound}
		&\mu\|\Delta x^k\|^2+\frac{0.5\overline{\rho}\|A\Delta x^k\|^2}{\vartheta(z^k)+\vartheta(\varpi)}\nonumber\\
		&\leq\frac{\langle{G}(A^{\top}\zeta),{G}(\Delta x^k)\rangle}{\vartheta(\varpi)}+\langle{G}(\xi^k),{G}(\Delta x^k)\rangle\nonumber\\
		&\quad+\lambda\sum_{i\in\overline{S}}v_i^{k-1}\|\Delta x^k_{J_i}\|-\lambda\sum_{i\notin S^{k-1}}v_i^{k-1}\|\Delta x^k_{J_i}\|\nonumber\\
		&\leq\frac{\|{G}(A^{\top}\zeta)\|_\infty}{\vartheta(\varpi)}\Big(\sum_{i\in S^{k-1}}\|\Delta x^k_{J_i}\|+\sum_{i\notin S^{k-1}}\|\Delta x^k_{J_i}\|\Big)\nonumber\\
		&\quad+\sum_{i\in S^{k-1}}\|{G}(\xi^k)\|_\infty\|\Delta x^k_{J_i}\|\nonumber\\
		&\quad+\|{G}(\xi^k)\|_\infty\sum_{i\notin S^{k-1}}\|\Delta x^k_{J_i}\|+\lambda\sum_{i\in\overline{S}}v_i^{k-1}\|\Delta x^k_{J_i}\|\nonumber\\
		&\quad-\lambda\sum_{i\notin S^{k-1}}v_i^{k-1}\|\Delta x^k_{J_i}\|\nonumber\\	&\leq\frac{\|{G}(A^{\top}\zeta)\|_\infty}{\vartheta(\varpi)}\sum_{i\in S^{k-1}}\|\Delta x^k_{J_i}\|+\|{G}(\xi^k)\|_\infty\sum_{i\in S^{k-1}}\|\Delta x^k_{J_i}\|\nonumber\\
		&\quad+\lambda\|v^{k-1}_{\overline{S}}\|_\infty\sum_{i\in\overline{S}}\|\Delta x^k_{J_i}\|,
	\end{align}
	where the last inequality is from $\lambda\geq c(\|{G}(\xi^k)\|_\infty+2L_{\vartheta}\|A\|)$. According to \eqref{ineq-ass1}, it follows that
	$$
	\vartheta(z^k)+\vartheta(\varpi)=\vartheta(A\Delta x^k-\varpi)+\vartheta(\varpi)\leq \vartheta(A\Delta x^k)+2\vartheta(\varpi).
	$$
	Therefore,~$\frac{0.5\overline{\rho}\|A\Delta x^k\|^2}{\vartheta(z^k)+\vartheta(\varpi)}\leq\frac{0.5\overline{\rho}\|A\Delta x^k\|^2}{\vartheta(A\Delta x^k)+2\vartheta(\varpi)}:=\Upsilon^k$. Combining \eqref{ineq43-ebound} yields
	\begin{align}\label{dq1}
		&\mu\|\Delta x^k\|^2+\frac{0.5\overline{\rho}\|A\Delta x^k\|^2}{\vartheta(A\Delta x^k)+2\vartheta(\varpi)}\nonumber\\
		&\leq\frac{\|{G}(A^{\top}\zeta)\|_\infty}{\vartheta(\varpi)}\sum_{i\in S^{k-1}}\|\Delta x^k_{J_i}\|+\|{G}(\xi^k)\|_\infty\sum_{i\in S^{k-1}}\|\Delta x^k_{J_i}\|\nonumber\\
		&\quad+\lambda\|v^{k-1}_{\overline{S}}\|_\infty\sum_{i\in\overline{S}}\|\Delta x^k_{J_i}\|.
	\end{align}
	From $\frac{\|A\Delta x^k\|^2}{2n}\geq\kappa\|\Delta x^k\|^2$, it holds that $\frac{n\kappa\overline{\rho}\|\Delta x^k\|^2}{\vartheta(A\Delta x^k)+2\vartheta(\varpi)}\leq{\Upsilon}^k.$ By using the Lipschitz continuity of $\vartheta$ on the set $\mathcal{Z}$ in the Assumption \ref{ass1} and combining $\vartheta(0)=0$ and $\sum_{i\in (S^{k-1})^c}\|\Delta x^k_{J_i}\|\leq\frac{2(c+1)}{c-2}\sum_{i\in S^{k-1}}\|\Delta x^k_{J_i}\|$ of Lemma \ref{compatity condition}, it holds that $\vartheta(A\Delta x^k)=\vartheta(A\Delta x^k)-\vartheta(0)=\vartheta(Ax^k-b-(Ax^*-b))-\vartheta(0)\leq L_{\vartheta}\|A\Delta x^k\|\le L_{\vartheta}\|A\|_{2,\infty}\|\Delta x^k\|\leq L_{\vartheta}\|A\|_{2,\infty}\|{G}(\Delta x^k)\|\leq L_{\vartheta}\|A\|_{2,\infty}\|{G}(\Delta x^k)\|_1\leq L_{\vartheta}\frac{3c}{c-2}\|A\|_{2,\infty}\sqrt{|S^{k-1}|}\|[{G}(\Delta x^k)]_{S^{k-1}}\|.$
	
	Multiplying $\vartheta(A\Delta x^k)+2\vartheta(\varpi)$ on both sides of inequality \eqref{dq1} yields
	\begin{align*}
		&(\mu(\vartheta(A\Delta x^k)+2\vartheta(\varpi))+n\kappa\overline{\rho})\|\Delta x^k\|^2\\
		&\leq \vartheta(A\Delta x^k)\Big(\frac{\|{G}(A^{\top}\zeta)\|_\infty}{\vartheta(\varpi)}\sum_{i\in S^{k-1}}\|\Delta x^k_{J_i}\|\\
		&\quad+\|{G}(\xi^k)\|_\infty\sum_{i\in S^{k-1}}\|\Delta x^k_{J_i}\|+\lambda\|v^{k-1}_{\overline{S}}\|_\infty\sum_{i\in\overline{S}}\|\Delta x^k_{J_i}\|\Big)\\
		&\quad\left. \right.+2\vartheta(\varpi)\Big(\frac{\|{G}(A^{\top}\zeta)\|_\infty}{\vartheta(\varpi)}\sum_{i\in S^{k-1}}\|\Delta x^k_{J_i}\|\\
		&\quad+\|{G}(\xi^k)\|_\infty\sum_{i\in S^{k-1}}\|\Delta x^k_{J_i}\|+\lambda\|v^{k-1}_{\overline{S}}\|_\infty\sum_{i\in\overline{S}}\|\Delta x^k_{J_i}\|\Big)\\
		&\leq\vartheta(A\Delta x^k)\Big(\frac{\|{G}(A^{\top}\zeta)\|_\infty}{\vartheta(\varpi)}+\|{G}(\xi^k)\|_\infty\\
		&\quad+\lambda\|v^{k-1}_{\overline{S}}\|_\infty\Big)\|[{G}(\Delta x^k)]_{S^{k-1}}\|_1\\
		&\quad\left. \right.+2\vartheta(\varpi)\Big(\frac{\|{G}(A^{\top}\zeta)\|_\infty}{\vartheta(\varpi)}+\|{G}(\xi^k)\|_\infty\\
		&\quad+\lambda\sqrt{\sum_{i\in\overline{S}}(v_i^{k-1})^2}\Big)\|[{G}(\Delta x^k)]_{S^{k-1}}\|_1\\
		&=L_{\vartheta}\frac{3c}{c-2}\|A\|_{2,\infty}\Big(\frac{\|{G}(A^{\top}\zeta)\|_\infty}{\vartheta(\varpi)}+\|{G}(\xi^k)\|_\infty\\
		&\quad+\lambda\|v^{k-1}_{\overline{S}}\|_\infty\Big)|S^{k-1}|\|\Delta x^k\|^2\\	&\quad\left.\right.+2\vartheta(\varpi)\Big(\frac{\|{G}(A^{\top}\zeta)\|_\infty}{\vartheta(\varpi)}+\|{G}(\xi^k)\|_\infty\\
		&\quad+\lambda\|v^{k-1}_{\overline{S}}\|_\infty\Big)\sqrt{|S^{k-1}|}\|\Delta x^k\|.
	\end{align*}
	So, by combining $\mu\vartheta(A\Delta x^k)\|\Delta x^k\|^2\geq0$, it is derived that
	\begin{align}\label{dqw}
		&(2\mu \vartheta(\varpi)+n\kappa\overline{\rho})\|\Delta x^k\|^2\nonumber\\
		&\leq \frac{3c}{c-2}L_{\vartheta}\|A\|_{2,\infty}\Big(\frac{\|{G}(A^{\top}\zeta)\|_\infty}{\vartheta(\varpi)}+\|{G}(\xi^k)\|_\infty\nonumber\\
		&\quad\left.\right.+\lambda\|v^{k-1}_{\overline{S}}\|_\infty\Big)|S^{k-1}|\|\Delta x^k\|^2\nonumber\\	&\quad\left.\right.+2\vartheta(\varpi)\Big(\frac{\|{G}(A^{\top}\zeta)\|_\infty}{\vartheta(\varpi)}+\|{G}(\xi^k)\|_\infty+\nonumber\\
		&\quad\left.\right.\lambda\|v^{k-1}_{\overline{S}}\|_\infty\Big)\sqrt{|S^{k-1}|}\|\Delta x^k\|.
	\end{align}
	Utilizing $\lambda<\frac{2\mu \vartheta(\varpi)+n\kappa\overline{\rho}-\frac{3c}{c-2}L_{\vartheta}\|A\|_{2,\infty}(L_{\vartheta}\|A\|+\|{G}(\xi^k)\|_\infty)|S^{k-1}|}{\frac{3c}{c-2}L_{\vartheta}\|A\|_{2,\infty}\|v_{\overline{S}}^{k-1}\|_\infty|S^{k-1}|}$ yields the conclusion.
\end{proof}
\begin{remark}\label{remark-error-bound}
	From the proof process of Theorem \ref{error bound} above, it can be inferred from the inequality \eqref{dqw} that the set on $\lambda$ in this lemma is nonempty when the sample size $n$ satisfies $n>\frac{1}{\kappa\overline{\rho}}(\frac{3c}{c-2}L_{\vartheta}\|A\|_{2,\infty}|S^{k-1}|(L_{\vartheta}\|A\|+\|{G}(\xi^k)\|_{\infty}+\lambda)-2\mu\vartheta(\varpi))$.
	From the conclusion of Theorem \ref{error bound}, when the design matrix $A$ satisfies Assumption \ref{ass0} on the set $\mathcal{C}(\overline{S},1.5\overline{r})$, it can be seen that $\frac{2\mu \vartheta(\varpi)+n\kappa\overline{\rho}-\frac{3c}{c-2}L_{\vartheta}\|A\|_{2,\infty}(L_{\vartheta}\|A\|+\|{G}(\xi^k)\|_\infty)|S^{k-1}|}{\frac{3c}{c-2}L_{\vartheta}\|A\|_{2,\infty}\|v_{\overline{S}}^{k-1}\|_{\infty}|S^{k-1}|}$ increases as $n$ increases, which then leads to an increase in the interval of $\lambda$. Additionally, the distance between the iterates $x^k$ generated by the Algorithm \ref{PMM} and the true solution of the problem will be smaller. When the parameter $\lambda$ is fixed, the error bound of the iterative sequence $\{x^k\}_{k\in\mathbb{N}}$ to the true solution decreases as the group sparsity rate increases.
\end{remark}

In the Algorithm \ref{PMM}, this article first finds an initial point $x^0$, and then calculates $w_i^k\in(\psi^*)'(\rho\|x^k_{J_i}\|)$ for $i=1,\ldots,m$ and $k\in\mathbb{N}$, and obtains $x^{k+1}$ by approximately solving a strongly convex optimization problem with a proximal term. The iterates $x^{k+1}$ cannot be guaranteed to fall into the cone $\mathcal{C}(\overline {S}, 1.5\overline{r})$ defined in this section, which poses difficulties in proving statistical error bounds. To overcome this difficulty, this section presents the iterative sequence $\{x^k\}_{k\in\mathbb{N}}$ generated by Algorithm \ref{PMM}. Assuming the cluster point $\overline{x}$ of $\{x^k\}_{k\in\mathbb{N}}$ is satisfied with that $\|\overline{x}_{J_i}\|$ is small enough in the case of $i\notin\overline{S}$, which in this way, ensures that each iterate $x^k$ can fall into the cone $\mathcal{C}(\overline{S},1.5\overline{r})$ after sufficiently large iterations.  The following Theorem \ref{error bound1} is proposed based on the above idea. Therefore, we will establish the statistical error bound between the cluster point $\overline {x}$ of the iterative sequence $\{x^k\}_{k\in\mathbb{N}}$ generated by the Algorithm \ref{PMM} and the true solution $x^*$ of problem \eqref{prob1}.

\begin{theorem}\label{error bound1}
	Suppose $A$ satisfies the Assumption \ref{ass0} on the set $\mathcal{C}(\overline{S},1.5\overline{r})$ and the noise vector $\varpi$ is nonzero. If $n>\frac{1}{\kappa\overline{\rho}}(\frac{4.5\overline{r}c}{c-2}L_{\vartheta}\|A\|_{2,\infty}(L_{\vartheta}\|A\|+\|{G}(\mu x^*)\|_{\infty}+\lambda)-2\mu\vartheta(\varpi))$ and $\lambda$ satisfies
	\begin{align*}
	&c(\|{G}(\mu x^*)\|_\infty+2L_{\vartheta}\|A\|)\le\lambda<\\
	&\frac{2\mu \vartheta(\varpi)+n\kappa\overline{\rho}-\frac{4.5\overline{r}c}{c-2}L_{\vartheta}\|A\|_{2,\infty}(L_{\vartheta}\|A\|+\|{G}(\mu x^*)\|_\infty)}{\frac{4.5\overline{r}c}{c-2}L_{\vartheta}\|A\|_{2,\infty}},
   \end{align*}
	where $c>2$ is a constant. When the cluster point $\overline{x}$ of the iterative sequence $\{x^k\}_{k\in\mathbb{N}}$ generated by Algorithm \ref{PMM} satisfies $\|\overline{x}_{J_i}\|\leq\frac{2a}{\rho(a+1)}$ for~$i\in\overline{S}^c$, it holds that
	\begin{align*}	
		&\|\overline{x}-x^*\|\leq\\
		&\frac{2\sqrt{1.5\overline{r}}\vartheta(\varpi)(L_{\vartheta}\|A\|+\|{G}(\mu x^*)\|_\infty+\lambda)}{2\mu \vartheta(\varpi)+n\kappa\overline{\rho}-\frac{4.5\overline{r}c}{c-2}L_{\vartheta}\|A\|_{2,\infty}(L_{\vartheta}\|A\|+\|{G}(\mu x^*)\|_\infty+\lambda)}.
	\end{align*}
\end{theorem}
\begin{proof}
	For sufficiently large $k$, the authors first demonstrate $S^{k-1}\equiv\overline{S}$. Since $x_k\rightarrow\overline{x}$ holds as $k\rightarrow\infty$ and we have $\gamma_{i,k}\in[\underline{\gamma}_i,\gamma_{i,0}]$ and $\|\delta^{k-1}\|\leq\frac{\|Q^{1/2}(x^{k-1}-x^{k-2})\|}{\sqrt{2}\|Q^{-1/2}\|}$ for $i=1,2$. According to the definition of $\xi^k$, it follows $\xi^k\rightarrow-\mu x^*$ as $k\rightarrow\infty$. Therefore, as $k\rightarrow\infty$, it yields $\|\xi^k_{J_i}\|\rightarrow\|\mu x^*_{J_i}\|$ for $i=1,2,\ldots,m$. Then ${G}(\xi^k)\rightarrow{G}(\mu x^*)$ holds as $k\rightarrow\infty$. So there exists an integer $\widehat{k}>0$ such that the inequality $\frac{1}{2}\|G(\mu x^*)\|_\infty\leq\|{G}(\xi^k)\|_\infty\leq\frac{3}{2}\|{G}(\mu x^*)\|_\infty$ holds for all $k\geq\widehat{k}$. As $k\rightarrow\infty$, according to $x_k\rightarrow\overline{x}$, for arbitrarily small constant $\epsilon>0$ (we can take $\epsilon=\frac{1}{\rho(a+1)}$), there exists a positive integer $\widetilde{k}$ such that for all $i=1,2,\ldots,m$, it holds that $\|x^k_{J_i}\|-\|\overline{x}_{J_i}\|\leq\|x^k_{J_i}-\overline{x}_{J_i}\|\leq\frac{1}{\rho(a+1)}$ for all $k\geq\widetilde{k}$. According to the assumption on~$\overline{x}$, namely $\|\overline{x}_{J_i}\|\leq\frac{a}{\rho(a+1)}$ holds for~$i\notin\overline{S}$. It is inferred that $\|x^k_{J_i}\|\leq\|\overline{x}_{J_i}\|+\frac{1}{\rho(a+1)}\leq\frac{a}{\rho(a+1)}+\frac{1}{\rho(a+1)}=\frac{1}{\rho}$ holds for $i\notin\overline{S}$. Take $\overline{k}=\max\{\widetilde{k},\widehat{k}\}$, then it is obtained that $S^{k-1}\equiv\overline{S}$ for all $k\geq\overline{k}+1$. As $k\geq\overline{k}+1$, combining $\max_{i\in\overline{S}^c}w^{k-1}_i\leq\frac{1}{2}$ follows $\|v^{k-1}_{\overline{S}}\|_\infty\leq1$. Therefore, by selecting the $\lambda$ that satisfies Theorem \ref{error bound} and combining it with Theorem \ref{error bound}, it can be obtained that
	\begin{align*}	
	&\|\overline{x}-x^*\|\leq\\
	&\frac{2\sqrt{1.5\overline{r}}\vartheta(\varpi)(L_{\vartheta}\|A\|+\|{G}(\mu x^*)\|_\infty+\lambda)}{2\mu \vartheta(\varpi)+n\kappa\overline{\rho}-\frac{4.5\overline{r}c}{c-2}L_{\vartheta}\|A\|_{2,\infty}(L_{\vartheta}\|A\|+\|{G}(\mu x^*)\|_\infty+\lambda)}.
    \end{align*}
	The proof is finished.
\end{proof}

From Theorem \ref{error bound1}, it can be seen that when the design matrix $A$ satisfies the Assumption \ref{ass0} on the set $\mathcal{C}(\overline{S},1.5\overline{r})$ and the sample size $n$ satisfies $n>\frac{1}{\kappa\overline{\rho}}(\frac{4.5\overline{r}c}{c-2}L_{\vartheta}\|A\|_{2,\infty}(L_{\vartheta}\|A\|+\|{G}(\mu x^*)\|_{\infty}+\lambda)-2\mu\vartheta(\varpi))$, the value range of $\lambda$ is nonempty and the interval of $\lambda$ increases as $n$ increases. In addition, the distance from the cluster point $\overline{x}$ of the iterative sequence $\{x^k\}_{k\in\mathbb{N}}$ generated by Algorithm~\ref{PMM} to the true solution $x^*$ of problem \eqref{prob1} will be smaller.

\subsection{The case where the loss function is a piecewise linear function}
Suppose the function $\vartheta(z):=\frac{1}{n}\sum^n_{i=1}\theta(z_i)$, where $z=Ax-b$. This article makes the following assumptions about the $\theta$ function.
\begin{assumption}\label{theta}
	Suppose $\theta:\mathbb{R}\rightarrow\mathbb{R}^+$ is a convex function and satisfies $\theta(0)=0$ and also $\theta$ is a squared strongly convex function with modulus $\overline{\rho}>0$. According to \cite[theorem~23.4]{Roc76}, there exists $\widetilde{\rho}>0$ such that $|\eta|\leq\widetilde{\rho}$ holds for any $t\in\mathbb{R}$ and any $\eta\in\partial\theta(t)$.
\end{assumption}
For convenience, denote $S^{k-1}:=\{i:x^{k-1}_{J_i}\neq0\}$, $v^k=e-w^k$, $z^k=Ax^k-b$, $\varpi=b-Ax^*$, $\Delta x^k=x^k-x^*$ and $\xi^k:=Q_{k-1}(x^{k-1}-x^k)+\delta^{k-1}-\mu x^*$.
Set $\overline{J}:=\{i\in\{1,\ldots,n\}|\varepsilon_i\neq0\}$ and $\mathcal{J}_k:=\{i\notin\overline{J}|z_i^k\neq0\}$. This paper first gives an important lemma as follows in this subsection.
\begin{lemma}\label{compatity condition2}
	For some $k\geq1$, if there exists $S^{k-1}\supseteq\overline{S}$ such that $\max_{i\in(S^{k-1})^c}w_i^{k-1}\le\frac{1}{2}$. Then for $\lambda\geq c(\|{G}(\xi^k)\|_\infty+2n^{-1}\widetilde{\rho}|\!\| A_{\overline{J}\cdot}|\!\|_1)$, where $c>2$ is a constant, it holds that
	$$\sum_{i\in (S^{k-1})^c}\|\Delta x^k_{J_i}\|\leq\frac{2(c+1)}{c-2}\sum_{i\in S^{k-1}}\|\Delta x^k_{J_i}\|.$$
\end{lemma}
\begin{proof}
	Since the objective function of problem \eqref{subprobk} is a strongly convex function, a unique solution is ensured. According to the definition of~$x^k$ in \eqref{subprobk}, it yields
	\begin{align*}
		&\vartheta(Ax^*-b)+\frac{\mu}{2}\|x^*\|^2+\lambda\langle v^{k-1},{G}(x^*)\rangle+\frac{1}{2}\|x^*-x^{k-1}\|^2_{Q_{k-1}}\\
		&\quad-\langle\delta^{k-1},x^*-x^{k-1}\rangle\\
		&\geq\vartheta(Ax^k-b)+\frac{\mu}{2}\|x^k\|^2+\lambda\langle v^{k-1},{G}(x^k)\rangle+\frac{1}{2}\|x^k-x^{k-1}\|^2_{Q_{k-1}}\\
		&\quad-\langle\delta^{k-1},x^k-x^{k-1}\rangle+\frac{1}{2}\langle x^*-x^k,(\mu I+Q_{k-1})(x^*-x^k)\rangle.
	\end{align*}
	After transposition and due to $\frac{1}{2}[\|x^*\|^2-\|x^k\|^2]=-\frac{1}{2}\|x^k-x^*\|^2-\langle x^k-x^*,x^*\rangle$, one can observe that
	\begin{align*}
		&\vartheta(Ax^k-b)-\vartheta(Ax^*-b)\\
		&\leq\frac{\mu}{2}\|x^*\|^2-\frac{\mu}{2}\|x^k\|^2+\lambda\langle v^{k-1},{G}(x^*)-{G}(x^k)\rangle\\
		&\quad+\frac{1}{2}\|x^*-x^{k-1}\|^2_{Q_{k-1}}-\frac{1}{2}\|x^k-x^{k-1}\|^2_{Q_{k-1}}\\
		&\quad-\frac{1}{2}\langle x^*-x^k,(\mu I+Q_{k-1})(x^*-x^k)\rangle+\langle\delta^{k-1},x^k-x^*\rangle\\
		&=-\mu\|\Delta x^k\|^2+\lambda\langle v^{k-1},{G}(x^*)-{G}(x^k)\rangle+\langle\xi^k,x^k-x^*\rangle.
	\end{align*}
	By the assumptions on $\vartheta$, it is obtained that
	\begin{align*} 
		&\vartheta(Ax^k-b)-\vartheta(Ax^*-b)\\
		&=\frac{1}{n}\sum_{i=1}^n(\theta(z_i^k)-\theta(-\varpi_i))\\
		&=\frac{1}{n}\sum_{i\in\overline{J}}\frac{\theta^2(z_i^k)-\theta^2(-\varpi_i)}{\theta(z_i^k)
			+\theta(-\varpi_i)}+\frac{1}{n}\sum_{i\in\mathcal{J}_k}\frac{\theta^2(z_i^k)-\theta^2(-\varepsilon_i)}{\theta(z_i^k)+\theta(-\varpi_i)}\\
		&\geq\frac{1}{n}\sum_{i\in\overline{J}}\frac{\theta^2(z_i^k)-\theta^2(-\varpi_i)}{\theta(z_i^k)
			+\theta(-\varpi_i)}+\frac{1}{n}\sum_{i\in\mathcal{J}_k}\frac{\theta^2(z_i^k)-\theta^2(-\varpi_i)}{\widetilde{\rho}\|z^k\|_\infty},
	\end{align*}
	where the above inequality is from Assumption \ref{theta}. For $i=1,\ldots,n$, it follows $\theta(0)-\theta(z_i^k)\geq\overline{\eta}_i(0-z_i^k)$, where $\overline{\eta}_i\in\partial\theta(z_i^k)$, and then we get $\theta(z_i^k)\leq\widetilde{\rho}\|z^k\|_\infty$.
	By Assumption \ref{theta}, $\theta^2$ is strongly convex with modulus $\overline{\rho}$, so for each $i\in\{1,\ldots,n\}$, there exists $\widetilde{\eta}_i\in\partial(\theta^2)(-\varpi_i)$ such that
	\begin{align}\label{c1}	\theta^2(z_i^k)-\theta^2(-\varpi_i)\geq\widetilde{\eta}_i(z_i^k+\varpi_i)+0.5\overline{\rho}(z_i^k+\varpi_i)^2.
	\end{align}
	It yields from Assumption \ref{theta} that $\theta$ is a convex function and satisfies $\theta(0)=0$. Due to \cite[Theorem 10.49]{RW98}, it holds that $\partial(\theta^2)(-\varpi_i)=2D^*\theta(-\varpi_i)(\theta(-\varpi_i))$, where $D^*\theta(-\varpi_i):\mathbb{R}\rightrightarrows\mathbb{R}$
	is the coderivative of $\theta$ at $-\varpi_i$, which follows that $D^*\theta(-\varpi_i)(\theta(-\varpi_i))=\partial(\theta(-\varpi_i)\theta)(-\varpi_i)=\theta(-\varpi_i)\partial\theta(-\varpi_i)$ by combining \cite[Proposition 9.24(b)]{RW98}. Therefore, for $i\in\mathcal{J}_k$, it follows from $\varpi_i=0$ that $\theta(-\varpi_i)=0$. Then it is inferred that $\theta^2(z_i^k)-\theta^2(-\varpi_i)\geq0.5\overline{\rho}(z_i^k+\varpi_i)^2$ and
	\begin{align}\label{c2} \frac{1}{n}\sum_{i\in\mathcal{J}_k}\frac{\theta^2(z_i^k)-\theta^2(-\varpi_i)}{\widetilde{\rho}\|z^k\|_\infty}\geq\frac{0.5\overline{\rho}}{n\widetilde{\rho}}\sum_{i\in\mathcal{J}_k}\frac{(z_i^k+\varpi_i)^2}{\|z^k\|_\infty}.
	\end{align}
	For $i\in\overline{J}$, one can infer from \eqref{c1} that
	\begin{align}\label{c3}
		&\frac{1}{n}\sum_{i\in\overline{J}}\frac{\theta^2(z_i^k)-\theta^2(-\varpi_i)}{\theta(z_i^k)+\theta(-\varpi_i)}\nonumber\\
		&\geq\frac{1}{n}\sum_{i\in\overline{J}}\widetilde{\gamma}_i^k(z_i^k+\varpi_i)+\frac{0.5\overline{\rho}}{n\widetilde{\rho}}\sum_{i\in\overline{J}}\frac{(z_i^k+\varpi_i)^2}{\|z^k\|_\infty+\|\varpi\|_\infty},
	\end{align}
	where $\widetilde{\gamma}_i^k:=\frac{\widetilde{\eta}_i}{\theta(z_i^k)+\theta(-\varpi_i)}$. Note that $|\widetilde{\gamma}_i^k|\leq|\frac{\widetilde{\eta}_i}{\theta(-\varpi_i)}|\leq2\widetilde{\rho}$ holds for $i\in\overline{J}$ by Assumption \ref{theta} and $\partial(\theta^2)(-\varpi_i)=2\theta(-\varpi_i)\partial\theta(-\varpi_i)$. Therefore, combining \eqref{c1}, \eqref{c2} and \eqref{c3} yields
	\begin{align*}
		&\vartheta(Ax^k-b)-\vartheta(Ax^*-b)\\
		&\geq\frac{1}{n}\sum_{i\in\overline{J}}\frac{\theta^2(z_i^k)-\theta^2(-\varpi_i)}{\theta(z_i^k)
			+\theta(-\varpi_i)}+\frac{1}{n}\sum_{i\in\mathcal{J}_k}\frac{\theta^2(z_i^k)-\theta^2(-\varpi_i)}{\widetilde{\rho}\|z^k\|_\infty}\\
		&\geq-2n^{-1}\widetilde{\rho}|\!\|A_{\overline{J}\cdot}|\!\|_1\|G(\Delta x^k)\|_1+\frac{0.5\overline{\rho}}{n\widetilde{\rho}}\frac{\|A\Delta x^k\|^2}{\|z^k\|_\infty+\|\varpi\|_\infty}.
	\end{align*}
	So it is obtained from the previous inequality that
	\begin{align}\label{ieq-errb}
		&\mu\|\Delta x^k\|^2+\frac{0.5\overline{\rho}}{n\widetilde{\rho}}\frac{\|A\Delta x^k\|^2}{\|z^k\|_\infty+\|\varpi\|_\infty}\nonumber\\
		&\leq2n^{-1}\widetilde{\rho}|\!\| A_{\overline{J}\cdot}|\!\|_1\|G(\Delta x^k)\|_1+\lambda\langle v^{k-1},{G}(x^*)-G(x^k)\rangle\nonumber\\
		&\quad+\langle\xi^k,\Delta x^k\rangle\nonumber\\
		&\leq2n^{-1}\widetilde{\rho}|\!\| A_{\overline{J}\cdot}|\!\|_1\|G(\Delta x^k)\|_1+\lambda\sum_{i=1}^mv_i^{k-1}(\|x^*_{J_i}\|-\|x^k_{J_i}\|)\nonumber\\
		&\quad+\sum_{i=1}^m\langle\xi^k_{J_i},\Delta x^k_{J_i}\rangle\nonumber\\
		&\leq2n^{-1}\widetilde{\rho}|\!\| A_{\overline{J}\cdot}|\!\|_1\|G(\Delta x^k)\|_1+\lambda\sum_{i\in\overline{S}}v_i^{k-1}(\|x^*_{J_i}\|-\|x^k_{J_i}\|)\nonumber\\
		&\quad-\lambda\sum_{i\in\overline{S}^c}v_i^{k-1}\|x^k_{J_i}\|+\langle{G}(\xi^k),{G}(\Delta x^k)\rangle\nonumber\\
		&\leq2n^{-1}\widetilde{\rho}|\!\| A_{\overline{J}\cdot}|\!\|_1\|G(\Delta x^k)\|_1+\lambda\sum_{i\in\overline{S}}v_i^{k-1}\|\Delta x^k_{J_i}\|\nonumber\\
		&\quad-\lambda\sum_{i\in (S^{k-1})^c}v_i^{k-1}\|\Delta x^k_{J_i}\|+\langle{G}(\xi^k),{G}(\Delta x^k)\rangle\\
		&\leq2n^{-1}\widetilde{\rho}|\!\|A_{\overline{J}\cdot}|\!\|_1(\sum_{i\in S^{k-1}}\|\Delta x^k_{J_i}\|+\sum_{i\in(S^{k-1})^c}\|\Delta x^k_{J_i}\|)+\nonumber\\
		&\quad\lambda\sum_{i\in\overline{S}}\|\Delta x^k_{J_i}\|-\frac{\lambda}{2}\sum_{i\in (S^{k-1})^c}\|\Delta x^k_{J_i}\|\nonumber\\
		&\quad+\|{G}(\xi^k)\|_\infty\sum_{i\in S^{k-1}}\|\Delta x^k_{J_i}\|+\|{G}(\xi^k)\|_\infty\sum_{i\in(S^{k-1})^c}\|\Delta x^k_{J_i}\|\nonumber\\
		&\leq(2n^{-1}\widetilde{\rho}|\!\| A_{\overline{J}\cdot}|\!\|_1+\lambda+\|{G}(\xi^k)\|_\infty)\sum_{i\in S^{k-1}}\|\Delta x^k_{J_i}\|\nonumber\\
		&\quad(2n^{-1}\widetilde{\rho}|\!\| A_{\overline{J}\cdot}|\!\|_1-0.5\lambda+\|{G}(\xi^k)\|_\infty)\sum_{i\in (S^{k-1})^c}\|\Delta x^k_{J_i}\|\nonumber,
	\end{align}
	where the third inequality is from the definition of ${G}(\cdot)$ and $x^*_{J_i}=0$ for all $i\in\overline{S}^c$, the second to last inequality is due to $\min\limits_{i\in(S^{k-1})^c}v_i^{k-1}\geq\frac{1}{2}$, the last inequality is based on $S^{k-1}\supseteq\overline{S}$. By $\mu\|\Delta x^k\|^2+\frac{0.5\overline{\rho}}{n\widetilde{\rho}}\frac{\|A\Delta x^k\|^2}{\|z^k\|_\infty+\|\varpi\|_\infty}\geq0$, it holds that
	\begin{align*}
		&-(2n^{-1}\widetilde{\rho}|\!\| A_{\overline{J}\cdot}|\!\|_1-0.5\lambda+\|{G}(\xi^k)\|_\infty)\sum_{i\in (S^{k-1})^c}\|\Delta x^k_{J_i}\|\\
		&\leq(2n^{-1}\widetilde{\rho}|\!\| A_{\overline{J}\cdot}|\!\|_1+\lambda+\|{G}(\xi^k)\|_\infty)\sum_{i\in S^{k-1}}\|\Delta x^k_{J_i}\|.
	\end{align*}
	As a consequence, combining $\lambda\geq c(\|{G}(\xi^k)\|_\infty+2n^{-1}\widetilde{\rho}|\!\| A_{\overline{J}\cdot}|\!\|_1)$ again can derive the conclusion.
\end{proof}

Lemma \ref{compatity condition2} demonstrates that the left term $\sum_{i\in (S^{k-1})^c}\|\Delta x^k_{J_i}\|$ of inequality can be bounded by its right term
$\frac{2(c+1)}{c-2}\sum_{i\in S^{k-1}}\|\Delta x^k_{J_i}\|$.
Similarly, the selection of a constant $c>2$ is closely related to the function $\phi$ defined in \eqref{phi-assump}, which can be found in Remark \ref{compatity}. Under certain assumptions, the following theorem characterizes the statistical error bound from the iterates $x^k$ to the true solution $x^*$ of the problem.

\begin{theorem}\label{error bound2}
	Suppose $A$ satisfies the Assumption~\ref{ass0} on the set $\mathcal{C}(\overline{S},1.5\overline{r})$, and there exists an index set $S^{k-1}$ satisfying $S^{k-1}\supseteq \overline{S}$ and~$|S^{k-1}|\leq1.5\overline{r}$ for $k\geq1$ such that $\max\limits_{i\in (S^{k-1})^c}w_i^{k-1}\leq\frac{1}{2}$ holds. If $\lambda$ satisfies the following inequality
	\begin{align*}
	&c(\|{G}(\xi^k)\|_\infty+2n^{-1}\widetilde{\rho}|\!\| A_{\overline{J}\cdot}|\!\|_1)\le\lambda<\\
	&\frac{2\mu \|\varpi\|_\infty+\frac{\kappa\overline{\rho}}{\widetilde{\rho}}-\frac{3c\|A\|_\infty}{c-2}(\|{G}(\xi^k)\|_\infty+2n^{-1}\widetilde{\rho}|\!\| A_{\overline{J}\cdot}|\!\|_1)|S^{k-1}|}{\frac{3c\|A\|_\infty}{c-2}\|v_{\overline{S}}^{k-1}\|_\infty|S^{k-1}|},
	\end{align*}
	where $c>2$ is a constant, then it holds that
	\begin{align*}
		\|\Delta x^k\|\leq\frac{d_1}{d_2},
	\end{align*}
where $d_1:=2\|\varpi\|_\infty(\|{G}(\xi^k)\|_\infty+2n^{-1}\widetilde{\rho}|\!\| A_{\overline{J}\cdot}|\!\|_1+\lambda\|v^{k-1}_{\overline{S}}\|_\infty)\sqrt{|S^{k-1}|}$ and $d_2:=(2\mu\|\varpi\|_\infty+\frac{\overline{\rho}\kappa}{\widetilde{\rho}})-\frac{3c}{c-2}\|A\|_\infty(\|{G}(\xi^k)\|_\infty+2n^{-1}\widetilde{\rho}|\!\| A_{\overline{J}\cdot}|\!\|_1+\lambda\|v^{k-1}_{\overline{S}}\|_\infty)|S^{k-1}|$.
\end{theorem}
\begin{proof}
	According to \eqref{ieq-errb} in the proof of Lemma \ref{compatity condition2} and $S^{k-1}\supseteq\overline{S}$, we have the following inequalities
	\begin{align*}
		&\mu\|\Delta x^k\|^2+\frac{0.5\overline{\rho}}{n\widetilde{\rho}}\frac{\|A\Delta x^k\|^2}{\|z^k\|_\infty+\|\varpi\|_\infty}\\
		&\leq2n^{-1}\widetilde{\rho}|\!\| A_{\overline{J}\cdot}|\!\|_1\big(\sum_{i\in S^{k-1}}\|\Delta x^k_{J_i}\|+\sum_{i\in(S^{k-1})^c}\|\Delta x^k_{J_i}\|\big)\\
		&\quad+\lambda\sum_{i\in\overline{S}}v_i^{k-1}\|\Delta x^k_{J_i}\|-\frac{\lambda}{2}\sum_{i\in (S^{k-1})^c}\|\Delta x^k_{J_i}\|\\
		&\quad+\|{G}(\xi^k)\|_\infty(\sum_{i\in S^{k-1}}\|\Delta x^k_{J_i}\|+\sum_{i\in(S^{k-1})^c}\|\Delta x^k_{J_i}\|)\\
		&\leq\sum_{i\in S^{k-1}}\big(\frac{2\widetilde{\rho}|\!\| A_{\overline{J}\cdot}|\!\|_1}{n}+\|{G}(\xi^k)\|_\infty\big)\|\Delta x^k_{J_i}\|\\
		&\quad+(\frac{2\widetilde{\rho}|\!\| A_{\overline{J}\cdot}|\!\|_1}{n}-0.5\lambda+\|{G}(\xi^k)\|_\infty)\sum_{i\in (S^{k-1})^c}\|\Delta x^k_{J_i}\|\\
		&\quad+\sum_{i\in\overline{S}}\lambda v_i^{k-1}\|\Delta x^k_{J_i}\|\\
		&\leq\sum_{i\in S^{k-1}}(\frac{2\widetilde{\rho}|\!\| A_{\overline{J}\cdot}|\!\|_1}{n}+\|{G}(\xi^k)\|_\infty)\|\Delta x^k_{J_i}\|\sum_{i\in\overline{S}}\lambda v_i^{k-1}\|\Delta x^k_{J_i}\|\\
		&\leq(\|{G}(\xi^k)\|_\infty+\frac{2\widetilde{\rho}|\!\| A_{\overline{J}\cdot}|\!\|_1}{n}+\lambda\|v^{k-1}_{\overline{S}}\|_\infty)\|[G(\Delta x^k)]_{S^{k-1}}\|_1,
	\end{align*}
	where the third to last inequality is due to $\lambda\geq c(\|{G}(\xi^k)\|_\infty+2n^{-1}\widetilde{\rho}|\!\| A_{\overline{J}\cdot}|\!\|_1)$.
	Note that $\|z^k\|_\infty+\|\varpi\|_\infty=\|A\Delta x^k-\varpi\|_\infty+\|\varpi\|_\infty\leq\|A\Delta x^k\|_\infty+2\|\varpi\|_\infty$. Therefore, $\frac{0.5\overline{\rho}}{n\widetilde{\rho}}\frac{\|A\Delta x^k\|^2}{\|A\Delta x^k\|_\infty+2\|\varepsilon\|_\infty}\leq\frac{0.5\overline{\rho}\|A\Delta x^k\|^2}{\|z^k\|_\infty+\|\varpi\|_\infty}$ is obtained. Combining $\frac{\|A\Delta x^k\|^2}{2n}\geq\kappa\|\Delta x^k\|^2$ yields $\frac{\overline{\rho}\kappa}{\widetilde{\rho}}\frac{\|\Delta x^k\|^2}{\|A\Delta x^k\|_\infty+2\|\varpi\|_\infty}\leq\frac{0.5\overline{\rho}}{n\widetilde{\rho}}\frac{\|A\Delta x^k\|^2}{\|A\Delta x^k\|_\infty+2\|\varpi\|_\infty}.$
	Furthermore, by combining with the previous inequalities, it can be concluded that
	\begin{align}\label{c4}
		&\mu\|\Delta x^k\|^2+\frac{\overline{\rho}\kappa}{\widetilde{\rho}}\frac{\|\Delta x^k\|^2}{\|A\Delta x^k\|_\infty+2\|\varpi\|_\infty}\nonumber\\
		&\leq(\|{G}(\xi^k)\|_\infty+\frac{2\widetilde{\rho}|\!\| A_{\overline{J}\cdot}|\!\|_1}{n}+\lambda\|v^{k-1}_{\overline{S}}\|_\infty)\|[G(\Delta x^k)]_{S^{k-1}}\|_1.
	\end{align}
	By utilizing $\sum_{i\in (S^{k-1})^c}\|\Delta x^k_{J_i}\|\leq\frac{2(c+1)}{c-2}\sum_{i\in S^{k-1}}\|\Delta x^k_{J_i}\|$ in Lemma \ref{compatity condition2}, it follows that $\|A\Delta x^k\|_\infty\leq\|A\|_\infty\|{G}(\Delta x^k)\|_1\leq
	\frac{3c}{c-2}\|A\|_\infty\|[{G}(\Delta x^k)]_{S^{k-1}}\|_1.$
	Multiplying $\|A\Delta x^k\|_\infty+2\|\varepsilon\|_\infty$ simultaneously on both sides of inequality \eqref{c4} above yields
	\begin{align*}
		&(\mu(\|A\Delta x^k\|_\infty+2\|\varpi\|_\infty)+\frac{\overline{\rho}\kappa}{\widetilde{\rho}})\|\Delta x^k\|^2\\
		&\leq(\|A\Delta x^k\|_\infty+2\|\varpi\|_\infty)[(\|{G}(\xi^k)\|_\infty+2n^{-1}\widetilde{\rho}|\!\| A_{\overline{J}\cdot}|\!\|_1\\
		&\quad+\lambda\|v^{k-1}_{\overline{S}}\|_\infty)\|[{G}(\Delta x^k)]_{S^{k-1}}\|_1]\\
		&\leq\frac{3c}{c-2}\|A\|_\infty[(\|{G}(\xi^k)\|_\infty+2n^{-1}\widetilde{\rho}|\!\| A_{\overline{J}\cdot}|\!\|_1\\
		&\quad+\lambda\|v^{k-1}_{\overline{S}}\|_\infty)\|[{G}(\Delta x^k)]_{S^{k-1}}\|^2_1]\\
		&\quad+2\|\varepsilon\|_\infty[(\|{G}(\xi^k)\|_\infty+2n^{-1}\widetilde{\rho}|\!\| A_{\overline{J}\cdot}|\!\|_1\\
		&\quad+\lambda\|v^{k-1}_{\overline{S}}\|_\infty)\|[{G}(\Delta x^k)]_{S^{k-1}}\|_1]\\
		&\leq\frac{3c}{c-2}\|A\|_\infty[(\|{G}(\xi^k)\|_\infty+2n^{-1}\widetilde{\rho}|\!\| A_{\overline{J}\cdot}|\!\|_1\\
		&\quad+\lambda\|v^{k-1}_{\overline{S}}\|_\infty)|S^{k-1}|\|\Delta x^k\|^2]\\
		&\quad+2\|\varepsilon\|_\infty[(\|{G}(\xi^k)\|_\infty+2n^{-1}\widetilde{\rho}|\!\| A_{\overline{J}\cdot}|\!\|_1\\
		&\quad+\lambda\|v^{k-1}_{\overline{S}}\|_\infty)\sqrt{|S^{k-1}|}\|\Delta x^k\|].
	\end{align*}
	Therefore, combining $\mu\|A\Delta x^k\|_\infty\|\Delta x^k\|^2\geq0$ follows
	\begin{align}\label{dqw2}
		&[(2\mu\|\varpi\|_\infty+\frac{\overline{\rho}\kappa}{\widetilde{\rho}})-\frac{3c\|A\|_\infty}{c-2}[(\|{G}(\xi^k)\|_\infty
		\nonumber\\
		&+\frac{2\widetilde{\rho}|\!\| A_{\overline{J}\cdot}|\!\|_1}{n}+\lambda\|v^{k-1}_{\overline{S}}\|_\infty)|S^{k-1}|]\|\Delta x^k\|^2\nonumber\\
		&\leq2\|\varpi\|_\infty[(\|{G}(\xi^k)\|_\infty+2n^{-1}\widetilde{\rho}|\!\| A_{\overline{J}\cdot}|\!\|_1+\nonumber\\
		&\quad\lambda\|v^{k-1}_{\overline{S}}\|_\infty)\sqrt{|S^{k-1}|}\|\Delta x^k\|].
	\end{align}
	Since $\lambda<\frac{2\mu \|\varpi\|_\infty+\frac{\kappa\overline{\rho}}{\widetilde{\rho}}-\frac{3c\|A\|_\infty}{c-2}[(\|{G}(\xi^k)\|_\infty+2n^{-1}\widetilde{\rho}|\!\| A_{\overline{J}\cdot}|\!\|_1)|S^{k-1}|]}{\frac{3c\|A\|_\infty}{c-2}\|v_{\overline{S}}^{k-1}\|_\infty|S^{k-1}|}$, so the proof is completed.
\end{proof}

It follows from inequality \eqref{dqw2} of Theorem \ref{error bound2} that the interval of $\lambda$ is nonempty. And if $n$ satisfies inequality $n>\frac{\frac{6c|S^{k-1}|}{c-2}\widetilde{\rho}|\!\| A_{\overline{J}\cdot}|\!\|_1}{2\mu\|\varpi\|_\infty+\frac{\kappa\overline{\rho}}{\widetilde{\rho}}
	-\frac{3c|S^{k-1}|}{c-2}\|A\|_\infty(\|{G}(\mu x^*)\|_\infty+\lambda)}$, then $c(\|{G}(\xi^k)\|_\infty+2n^{-1}\widetilde{\rho}|\!\| A_{\overline{J}\cdot}|\!\|_1)$ decreases and $\frac{2\mu \|\varpi\|_\infty+\frac{\kappa\overline{\rho}}{\widetilde{\rho}}-\frac{3c}{c-2}\|A\|_\infty[(\|{G}(\xi^k)\|_\infty+2n^{-1}\widetilde{\rho}|\!\| A_{\overline{J}\cdot}|\!\|_1)|S^{k-1}|]}{\frac{3c}{c-2}\|A\|_\infty\|v_{\overline{S}}^{k-1}\|_\infty|S^{k-1}|}$ increases as $n$ increases. In addition, as the number of samples $n$ increases, the distance between the iterative points generated by the proximal {\rm MM} algorithm and the true solution of the problem will become smaller. When the parameter $\lambda$ is fixed, the error bound of the iterates to the true solution decreases as the group sparsity increases. Based on the first two conclusions of this subsection, the following theorem will prove the statistical error bound of the cluster point of the iterative sequence generated by the Algorithm \ref{PMM} to the true solution of problem \eqref{prob1}.

\begin{theorem}\label{error bound3}
	Suppose $A$ satisfies the Assumption~\ref{ass0} on the set $\mathcal{C}(\overline{S},1.5\overline{r})$ and the noise vector $\varpi$ is nonzero. If $n>\frac{\frac{9\overline{r}c}{c-2}\widetilde{\rho}|\!\| A_{\overline{J}\cdot}|\!\|_1}{2\mu\|\varpi\|_\infty+\frac{\kappa\overline{\rho}}{\widetilde{\rho}}
		-\frac{4.5\overline{r}c}{c-2}\|A\|_\infty(\|{G}(\mu x^*)\|_\infty+\lambda)}$,
	and $\lambda$ satisfies 
	\begin{align*}
	&c(\|{G}(\mu x^*)\|_\infty+2n^{-1}\widetilde{\rho}|\!\| A_{\overline{J}\cdot}|\!\|_1)\le\lambda<\\
	&\frac{2\mu \|\varpi\|_\infty+\frac{\kappa\overline{\rho}}{\widetilde{\rho}}-\frac{4.5\overline{r}c}{c-2}\|A\|_\infty(\|{G}(\mu x^*)\|_\infty+2n^{-1}\widetilde{\rho}|\!\| A_{\overline{J}\cdot}|\!\|_1)}{\frac{4.5\overline{r}c}{c-2}\|A\|_\infty},
	\end{align*}
	where $c>2$ is a constant. Then, when the cluster point $\overline{x}$ of iterative sequence $\{x^k\}_{k\in\mathbb{N}}$ generated by Algorithm \ref{PMM} satisfies $\|\overline{x}_{J_i}\|\leq\frac{2a}{\rho(a+1)}$ for $i\in\overline{S}^c$, it holds that $\|\overline{x}-x^*\|\leq\frac{2\|\varpi\|_\infty[(\|{G}(\mu x^*)\|_\infty+2n^{-1}\widetilde{\rho}|\!\| A_{\overline{J}\cdot}|\!\|_1+\lambda)\sqrt{1.5\overline{r}}}{(2\mu\|\varpi\|_\infty+\frac{\overline{\rho}\kappa}{\widetilde{\rho}})-\frac{4.5\overline{r}c}{c-2}\|A\|_\infty(\|{G}(\mu x^*)\|_\infty+2n^{-1}\widetilde{\rho}|\!\|A_{\overline{J}\cdot}|\!\|_1+\lambda)}.$
\end{theorem}
\begin{proof}
	For sufficiently large $k$, we first prove $S^{k-1}\equiv\overline{S}$. Since $x_k\rightarrow\overline{x}$ for $k\rightarrow\infty$, $\gamma_{i,k}\in[\underline{\gamma}_i,\gamma_{i,0}]$ for $i=1,2$ and $\|\delta^{k-1}\|\leq\frac{\|Q^{1/2}(x^k-x^{k-1})\|}{\sqrt{2}\|Q^{-1/2}\|}$. Therefore, according to the definition of $\xi^k$, $\xi^k\rightarrow-\mu x^*$ holds as $k\rightarrow\infty$. So for $i=1,2,\ldots,m$, $\|\xi^k_{J_i}\|\rightarrow\|\mu x^*_{J_i}\|$ holds as $k\rightarrow\infty$. Then ${G}(\xi^k)\rightarrow{G}(\mu x^*)$ holds $k\rightarrow\infty$. Therefore, there exists $\widehat{k}>0$ such that~$\frac{1}{2}\|G(\mu x^*)\|_\infty\leq\|{G}(\xi^k)\|_\infty\leq\frac{3}{2}\|{G}(\mu x^*)\|_\infty$ holds for all~$k\geq\widehat{k}$. As $k\rightarrow\infty$, according to $x_k\rightarrow\overline{x}$, then for arbitrarily small $\epsilon>0$ (we can take $\epsilon=\frac{1}{\rho(a+1)}$), there exists a positive integer $\widetilde{k}$ for all $k\geq\widetilde{k}$ such that $\|x^k_{J_i}\|-\|\overline{x}_{J_i}\|\leq\|x^k_{J_i}-\overline{x}_{J_i}\|\leq\frac{1}{\rho(a+1)}$ holds for each $i=1,2,\ldots,m$. By the assumption on $\overline{x}$, it is obtained that $\|\overline{x}_{J_i}\|\leq\frac{a}{\rho(a+1)}$ for $i\notin\overline{S}$, so it is further acquired that $\|x^k_{J_i}\|\leq\|\overline{x}_{J_i}\|+\frac{1}{\rho(a+1)}\leq\frac{a}{\rho(a+1)}+\frac{1}{\rho(a+1)}=\frac{1}{\rho}$ for $i\notin\overline{S}$. Pick $\overline{k}=\max\{\widetilde{k},\widehat{k}\}$, then it yields $S^{k-1}\equiv\overline{S}$ for all $k\geq\overline{k}+1$. When $k\geq\overline{k}+1$, it follows from $\max_{i\in\overline{S}^c}w^{k-1}_i\leq\frac{1}{2}$ that $\|v^{k-1}_{\overline{S}}\|_\infty\leq1$. So, by selecting $\lambda$ that satisfies Theorem \ref{error bound2} and combining it with Theorem \ref{error bound2}, one can obtain
$\|\overline{x}-x^*\|\leq\frac{2\|\varpi\|_\infty[(\|{G}(\mu x^*)\|_\infty+2n^{-1}\widetilde{\rho}|\!\| A_{\overline{J}\cdot}|\!\|_1+\lambda)\sqrt{1.5\overline{r}}}{(2\mu\|\varpi\|_\infty+\frac{\overline{\rho}\kappa}{\widetilde{\rho}})-\frac{4.5\overline{r}c}{c-2}\|A\|_\infty(\|{G}(\mu x^*)\|_\infty+2n^{-1}\widetilde{\rho}|\!\|A_{\overline{J}\cdot}|\!\|_1+\lambda)}.$
The proof is completed.
\end{proof}

From the above Theorem \ref{error bound3}, it can be observed that when the design matrix $A$ satisfies the Assumption \ref{ass0} on the set $\mathcal{C}(\overline {S}, 1.5\overline{r})$, and when the number of samples $n$ reaches a certain level, the value range of $\lambda$ is nonempty. Moreover, as the number of samples $n$ increases, the interval of $\lambda$ expands, and the distance from the cluster point of the iterative sequence generated by Algorithm \ref{PMM} to the true solution of the problem diminishes. The discussion regarding the constant $c$ can also refer to the previous section. The following proposition provides the statistical error bound from the $x^0$ of Algorithm \ref{PMM} to the true solution $x^*$ of the problem.

\begin{proposition}\label{initial-point}
	Let $\Theta(x):=\vartheta(Ax-b)+\widetilde{\lambda}\langle e,G(x)\rangle
	+\frac{\widetilde{\gamma}_{1,0}}{2}\|x\|^2+\frac{\widetilde{\gamma}_{2,0}}{2}\|Ax\|^2$. If $x^0$ is the approximate optimal solution to the problem \eqref{x0sub}, i.e., there exist $\xi^0\in\mathbb{R}^n$ and $\varepsilon\ge0$ satisfying $\|\xi^0\|\le\varepsilon$ such that $\xi^0\in\partial\Theta(x^0)$. In addition, if $\widetilde{\lambda}\geq2(2n^{-1}\widetilde{\rho}|\!\|A|\!\|_1+\widetilde{\gamma}_{1,0}\|G(x^*)\|_\infty+\widetilde{\gamma}_{2,0}\|G(A^{\top}Ax^*)\|_\infty+\varepsilon)$, then it holds that $\|x^0-x^*\|\le C\sqrt{\overline{r}}$.		
\end{proposition}
\begin{proof}
	Denote $\Delta x^0:=x^0-x^*$. Based on known assumptions, strong convexity of function $\Theta$ and $x^*\in \mathcal{X}$, it can be concluded that
	\begin{align*}
		&\vartheta(Ax^*-b)+\widetilde{\lambda}\langle e,G(x^*)\rangle
		+\frac{\widetilde{\gamma}_{1,0}}{2}\|x^*\|^2+\frac{\widetilde{\gamma}_{2,0}}{2}\|Ax^*\|^2\\
		&\ge\vartheta(Ax^0-b)+\widetilde{\lambda}\langle e,G(x^0)\rangle
		+\frac{\widetilde{\gamma}_{1,0}}{2}\|x^0\|^2+\frac{\widetilde{\gamma}_{2,0}}{2}\|Ax^0\|^2\\
		&+\langle\xi^0,x^*-x^0\rangle+\frac{1}{2}\langle x^*-x^0,(\widetilde{\gamma}_{1,0} I+\widetilde{\gamma}_{2,0}A^{\top}A)(x^*-x^0)\rangle.
	\end{align*}
	From $\vartheta(z)=\frac{1}{n}\sum^n_{i=1}\theta(z_i)$ and Assumption \ref{theta}, it follows
	\begin{align*}
		\vartheta(Ax^0-b)-\vartheta(Ax^*-b)
		&\geq-2n^{-1}\widetilde{\rho}\|G(A\Delta x^0)\|_1.
	\end{align*}
	According to $0.5\widetilde{\gamma}_{1,0}[\|x^0\|^2-\|x^*\|^2]=0.5\widetilde{\gamma}_{1,0}\|x^0-x^*\|^2+\widetilde{\gamma}_{1,0}\langle x^0-x^*,x^*\rangle$ and
	$0.5\widetilde{\gamma}_{2,0}[\|Ax^0\|^2
	-\|Ax^*\|^2]=0.5\widetilde{\gamma}_{2,0}\|A(x^0-x^*)\|^2+\widetilde{\gamma}_{2,0}\langle x^0-x^*,A^{\top}Ax^*\rangle$, it can be obtained that
	\begin{align*}
		&\widetilde{\gamma}_{1,0}\|\Delta x^0\|^2\leq \langle x^0-x^*,(\widetilde{\gamma}_{1,0} I+\widetilde{\gamma}_{2,0}A^{\top}A)(x^0-x^*)\rangle\\
		&\leq2p^{-1}\widetilde{\rho}\|G(A\Delta x^0)\|_1+\widetilde{\lambda}\sum_{i=1}^m(\|x^*_{J_i}\|-\|x^0_{J_i}\|)\\
		&\quad+\langle x^0-x^*,\xi^0-(\widetilde{\gamma}_{1,0} I+\widetilde{\gamma}_{2,0}A^{\top}A)x^*\rangle\\
		&\leq(2n^{-1}\widetilde{\rho}|\!\|A|\!\|_1+\widetilde{\lambda}+\widetilde{\gamma}_{1,0}\|G(x^*)\|_\infty+\widetilde{\gamma}_{2,0}\|G(A^{\top}Ax^*)\|_\infty\\
		&\quad+\varepsilon)\|[G(\Delta x^0)]_{\overline{S}}\|_1\\
		&\quad+(2n^{-1}\widetilde{\rho}|\!\|A|\!\|_1-\widetilde{\lambda}+\widetilde{\gamma}_{1,0}\|G(x^*)\|_\infty+\widetilde{\gamma}_{2,0}\|G(A^{\top}Ax^*)\|_\infty\\
		&\quad+\varepsilon)\|[G(\Delta x^0)]_{(\overline{S})^c}\|_1\\	&\leq(2n^{-1}\widetilde{\rho}|\!\|A|\!\|_1+\widetilde{\lambda}+\widetilde{\gamma}_{1,0}\|G(x^*)\|_\infty+\widetilde{\gamma}_{2,0}\|G(A^{\top}Ax^*)\|_\infty\\
		&\quad+\varepsilon)\|[G(\Delta x^0)]_{\overline{S}}\|_1\\ &\leq(2n^{-1}\widetilde{\rho}|\!\|A|\!\|_1+\widetilde{\lambda}+\widetilde{\gamma}_{1,0}\|G(x^*)\|_\infty+\widetilde{\gamma}_{2,0}\|G(A^{\top}Ax^*)\|_\infty\\
		&\quad+\varepsilon)\sqrt{1.5\overline{r}}\|\Delta x^0\|,
	\end{align*}
	where the second to last inequality is based on $\widetilde{\lambda}\geq2(2n^{-1}\widetilde{\rho}|\!\|A|\!\|_1+\widetilde{\gamma}_{1,0}\|G(x^*)\|_\infty+\widetilde{\gamma}_{2,0}\|G(A^{\top}Ax^*)\|_\infty+\varepsilon)$ and the nonnegativity of $\widetilde{\gamma}_{1,0}\|\Delta x^0\|^2$, this means that $\Delta x^0$ satisfies the inequality $\|[G(\Delta x^0)]_{(\overline{S})^c}\|_1\leq3\|[G(\Delta x^0)]_{\overline{S}}\|_1$. Then it yields that $\|\Delta x^0\|\leq\frac{(2n^{-1}\widetilde{\rho}|\!\|A|\!\|_1+\widetilde{\lambda}+\widetilde{\gamma}_{1,0}\|G(x^*)\|_\infty+\widetilde{\gamma}_{2,0}\|G(A^{\top}Ax^*)\|_\infty+\varepsilon)\sqrt{1.5\overline{r}}}{\widetilde{\gamma}_{1,0}}.$ 
\end{proof}

In accordance with Algorithm~\ref{PMM}, the generation of the iterates $x^k$ is contingent upon the selection of the initial point $x^0$. Consequently, in Proposition \ref{initial-point}, we stipulate the selected initial point should be within the corresponding cone $\mathcal{C}(\overline{S},1.5\overline {r})$. From the proof of the proposition, it can be seen obviously that the constant $C$ in the above proposition hinges on $\widetilde{\gamma}_{1,0}$ and $\widetilde{\gamma}_{2,0}$.

%%%%%%%%%%%%%%%%%%%%%%%%%%%%%%%%%%%%%%%%%%%%%%%%%%%%%%%
\section{Numerical experiment}\label{section5}

In this section, numerical testing is conducted on synthetic and real data to verify the effectiveness of the proximal~MM method for solving \eqref{DC-Sprob1}. The authors use MATLAB R2018b software to execute the algorithm, and all numerical tests are conducted on the same computer. The operating system of the computer is~64 bit~Windows 11, the processor is Intel(R)~Core (TM)~i5-1035G1, the frequency is~1.19GHz, and the memory is~8 GB.

\bigskip

The authors test the group zero-norm regularized composite optimization problem in problem \eqref{DC-Sprob1} by taking $\vartheta(Ax-b)=\|Ax-b\|_q/\sqrt{n}$, where $q=1,2$, $g(x)=x$, and $\mathcal{X}$ is the entire space, namely the following optimization problem
\begin{align}\label{sr-gopt}
	\mathop{\min}_{x\in\mathbb{R}^p}\Big\{\|Ax-b\|_q/\sqrt{n}+0.5\mu\|x\|^2+\lambda\langle e-w^k,{G}(x)\rangle\Big\}.
\end{align}
The authors use the inexact proximal~MM method (algorithm \ref{PMM}) to solve \eqref{sr-gopt} in section \ref{sec3.1}, where the termination criterion of algorithm \ref{PMM} is as follows in Remark \ref{PMM-stopcriterion}.
\begin{remark}\label{PMM-stopcriterion}
	For convenience, denote $f_1(z):=\|z\|_q/\sqrt{n}$, $h_k(x):=\textstyle{\sum_{i=1}^m}v_i^k\|x_{J_i}\|$, where $v_i^k=\lambda(1-w_i^k)$ for $i=1,2,\ldots,m$. Then the first-order optimality condition of \eqref{sr-gopt} has the following form:
	\[
	\xi\in\partial f_1(z),~\delta_k-A^{\top}\xi-\mu x\in\partial h_k(x),~Ax-z-b=0,
	\]
	where $\xi\in\mathbb{R}^p$ is the multiplier vector corresponding to~$Ax-z-b=0$. Therefore, the KKT residual ${\rm Err_{PMM}}^{k+1}$ of problem \eqref{sr-gopt} at $(x^{k+1},z^{k+1},\xi^{k+1})$ can be defined as
	$(\sqrt{r_1+r_2+r_3})/(1+\|b\|)$, where $r_1,~r_2$ and $r_3$ are respectively represented as
	$r_1=\|z^{k+1}-\mathcal{P}_{f_1}(z^{k+1}+\xi^{k+1})\|^2$, $r_2=\|x^{k+1}-\mathcal{P}_{h_k}(\delta_k+(1-\mu)x^{k+1}-A\xi^{k+1})\|^2$ and $r_3=\|Ax^{k+1}-z^{k+1}-b\|^2$. This paper uses ${\rm Err_{PMM}}^{k+1}\leq \epsilon_{\rm PMM}^{k+1}$ as the termination condition for Algorithm~\ref{PMM}.
\end{remark}

\subsection{Proximal dual semismooth Newton method for subproblem}
Subproblem \eqref{subprobk} can be equivalently written as
\begin{align}\label{eq-PMM-penalty-gz}
	&\min\limits_{x\in \mathbb{R}^p,z\in \mathbb{R}^n}\{f_1(z)+0.5\mu\|x\|^2+h_k(x)+\frac{\gamma_{1,k}}{2}\|x-x^k\|^2\nonumber\\
	&\quad\quad\quad\quad+\frac{\gamma_{2,k}}{2}\|z-z^k\|^2:~~Ax-z-b=0\}.
\end{align}
The multiplier variable $\xi$ is introduced for \eqref{eq-PMM-penalty-gz}, then the Lagrangian function of \eqref{eq-PMM-penalty-gz} is defined as
\begin{align*}
&\mathcal{L}(x,z;\xi):=f_1(z)+\frac{\mu}{2}\|x\|^2+h_k(x)+\frac{\gamma_{1,k}}{2}\|x-x^k\|^2\\
&\quad\quad\quad\quad\quad+\frac{\gamma_{2,k}}{2}\|z-z^k\|^2+\langle\xi,Ax-z-b\rangle.
\end{align*}
After calculation, it can be concluded that the dual problem of problem \eqref{eq-PMM-penalty-gz} is
\begin{align}\label{dual-eq-PMM-penalty-gz}
	&\mathop{\min}_{\xi\in\mathbb{R}^n}\Big\{\Psi_k(\xi):=\langle b,\xi\rangle+\frac{\gamma_{2,k}}{2}\|z^k+\gamma^{-1}_{2,k}\xi\|^2\nonumber\\
	&\quad\quad-e_{\gamma^{-1}_{2,k}}f_1(z^k+\gamma^{-1}_{2,k}\xi)-e_{r_k^{-1}}h_k(r_k^{-1}(\gamma_{1,k}x^k-A^{\top}\xi))\nonumber\\
	&\quad\quad+\frac{\|\gamma_{1,k}x^k-A^{\top}\xi\|^2}{2r_k}\Big\},
\end{align}
where $r_k:=\mu+\gamma_{1,k}$. Obviously, the strong duality of problem \eqref{eq-PMM-penalty-gz} and \eqref{dual-eq-PMM-penalty-gz} holds. Due to the smoothness and convexity of function $\Psi_k(\xi)$, solving the dual problem \eqref{dual-eq-PMM-penalty-gz} is equivalent to finding the roots of the following nonsmooth system:
\begin{align}\label{nonsmooth system}
	&\nabla\Psi_k(\xi)=b+z_k+\gamma^{-1}_{2,k}\xi-\Pi_\Gamma(z^k+\gamma^{-1}_{2,k}\xi)+\nonumber\\
	&A\Pi_\Lambda(r_k^{-1}(\gamma_{1,k}x^k-A^{\top}\xi))+r_k^{-1}A(A^{\top}\xi-\gamma_{1,k}x^k)=0,
\end{align}
where $\Gamma:=\{y:\|y\|\leq\gamma^{-1}_{2,k}\}$,~$\Lambda:=\Lambda_1\times\Lambda_2\times\cdots\times\Lambda_m$ and $\Lambda_i:=\{y_{J_i}:\|y_{J_i}\|\leq r_k^{-1}v_i^k\}$ holds for $i=1,2,\ldots,m$. For convenience, denote $u_k:=z^k+\gamma^{-1}_{2,k}\xi$ and $y^k:=r_k^{-1}(\gamma_{1,k}x^k-A^{\top}\xi)$. The generalized Clarke Jacobian $\partial\nabla\Psi_k$ of $\nabla\Psi_k$ satisfies
\begin{align}\label{Clarke Jacobian}
	&\partial(\nabla\Psi_k)(\xi)\subseteq\widehat{\partial}^2(\Psi_k)(\xi):=\gamma^{-1}_{2,k}(I-\partial\Pi_\Gamma(z^k+\gamma^{-1}_{2,k}\xi))\nonumber\\
	&\qquad\quad\left. \right.+r_k^{-1}A[I-\partial\Pi_\Lambda(r_k^{-1}(\gamma_{1,k}x^k-A^{\top}\xi))]A^{\top}
\end{align}
with
\begin{align}
	&\partial\Pi_\Gamma(u_k)\nonumber\\
	&=\begin{cases}\label{z-prox-operator}
		I &{\rm if}~\|u_k\|<\gamma^{-1}_{2,k},\\
		{\rm conv}(I,I-\gamma^{2}_{2,k}{u_k}^{\top}u_k) &{\rm if}~\|u_k\|=\gamma^{-1}_{2,k},\\
		\left\{\gamma^{-1}_{2,k}\left(\frac{1}{\|u_k\|}I-\frac{1}{\|u_k\|^3}{u_k}^{\top}u_k\right)\right\} &{\rm if}~\|u_k\|>\gamma^{-1}_{2,k}.
	\end{cases}
\end{align}
and
\begin{align}
	&\partial\Pi_{\Lambda_i}(y^k_{J_i})\nonumber\\
	&=\begin{cases}\label{z-prox-operator}
		I &{\rm if}~\|y^k_{J_i}\|<r_k^{-1}v_i^k,\\
		{\rm conv}(I,I-(r_k^{-1}v_i^k)^{-2}{y^k_{J_i}}^{\top}y^k_{J_i}) &{\rm if} ~\|y^k_{J_i}\|=r_k^{-1}v_i^k,\\
		\left\{r_k^{-1}v_i^k\left(\frac{1}{\|y^k_{J_i}\|}I-\frac{1}{\|y^k_{J_i}\|^3}{y^k_{J_i}}^{\top}y^k_{J_i}\right)\right\} &{\rm if}~\|y^k_{J_i}\|>r_k^{-1}v_i^k.
	\end{cases}
\end{align}
$\mathcal{U}_j(u)$ and~$\mathcal{V}_j(u)$ are respectively defined as
\begin{align}
	\mathcal{U}_j(u):=\{{\rm Diag}(v_1,v_2,\ldots,v_n)|v_i\in\partial\Pi_{\Gamma}(\gamma_{2,j}z^j+u)\}
\end{align}
and
\begin{align}
	&\mathcal{V}_j(u):=\{{\rm Diag}(W_{J_1},W_{J_2},\ldots,W_{J_m})|W_{J_i}\in\nonumber\\
	&\quad\quad\quad\quad\quad\partial\Pi_{\Lambda_i}[(\gamma_{1,j}x^j-A^{\top}u)_{J_i}]\}.
\end{align}
It is worth emphasizing that directly applying the semismooth Newton method to solve the system \eqref{nonsmooth system} will fail, because the generalized Hessian $\partial(\nabla\Psi_k)$ defined in the solving system \eqref{nonsmooth system} may be singular. At this point, this paper will use the proximal dual semismooth Newton method to solve the subproblem \eqref{dual-eq-PMM-penalty-gz} in inexact proximal point Algorithm \ref{PMM}. The specific iteration steps of this algorithm are as follows.

\setlength{\fboxrule}{0.8pt}
\noindent
%\fbox{
%	\parbox{0.96\textwidth}
	{
		%----------------------------------------------------------------------------------------------Algorithm
		\begin{algorithm}\label{PPA}({\bf Inexact proximal point algorithm for problem \eqref{dual-eq-PMM-penalty-gz}})
			\begin{description}
				\item[S.0] Fix $k\in\mathbb{N}$. Pick $\underline{\varrho}\geq0$, $\varrho_0>0$ and $\xi^j\in\mathbb{R}^p$. Set $j=0$.
				
				\item[{\rm{\bf while}}] termination conditions are not met {\rm{\bf do}}
				
				\item[S.1] $\xi^{j+1}$ is obtained by using the semismooth Newton method to solve the following problem
				\begin{equation}\label{PPA-DUAL}	\xi^{j+1}\approx\arg\min_{\xi\in\mathbb{R}^p}\Upsilon_{k,j}(\xi):=\Psi_k(\xi)+\frac{\varrho_j}{2}\|\xi-\xi^j\|^2.
				\end{equation}
				
				\item[S.2] Set $\varrho_{j+1}\downarrow\underline{\varrho}$ and $j\leftarrow j+1$, and return to step S.1.
				
				\item[{\rm{\bf end~while}}]
				
			\end{description}
		\end{algorithm}
	}
%}

%\bigskip

%\medskip
\begin{remark}\label{remark-PPA}
	According to \cite[Section~3]{Roc76}, for the problem \eqref{PPA-DUAL}, the algorithm \ref{PPA} adopts the following inexact discriminant criteria:
	\[
	\|\nabla\Upsilon_{k,j}(\xi^{j+1})\|\leq\alpha_j\varrho_j\|\xi^{j+1}-\xi^j\|~~\rm{and} ~~\sum_{j=0}^{\infty}\alpha_j<\infty.
	\]
	In this criteria, the results of the global convergence and linear convergence analysis of algorithm \ref{PPA} can be referenced in \cite{Roc76}.
\end{remark}
Note that solving the subproblem \eqref{PPA-DUAL} is equal to solving the roots of the following systems
\begin{equation}\label{PPA-DUAL-root}
	\nabla\Upsilon_{k,j}(\xi)=0.
\end{equation}
Since the mapping $\nabla\Upsilon_{k,j}:\mathbb{R}^p\rightarrow\mathbb{R}^p$ is~Lipschitz continuous, we define the generalized Hessian matrix of $\Upsilon_{k,j}$ at $\xi$ as $\partial^2\Upsilon_{k,j}(\xi):=\partial_C\nabla\Upsilon_{k,j}(\xi)$. By \cite[Theorem 2.2]{Hiriart84}, for all $d\in\mathbb{R}^p$, it holds that $\partial^2\Upsilon_{k,j}(\xi)d=\widehat{\partial}^2\Upsilon_{k,j}(\xi)d$, where
\begin{align*}
&\widehat{\partial}^2\Upsilon_{k,j}(\xi)=\varrho_jI+\gamma^{-1}_{2,k}\partial[\mathcal{P}_{\gamma^{-1}_{2,k}f}](z^k+\gamma^{-1}_{2,k}\xi)\\
&\qquad\qquad\quad+\gamma_{1,k}^{-1}A\partial[\mathcal{P}_{\gamma_{1,k}^{-1}h_k}](x_k-\gamma_{1,k}^{-1}A^{\top}\xi)A^{\top}.
\end{align*}
Since the matrices $U\in\partial[\mathcal{P}_{\gamma^{-1}_{2,k}f}](z^k+\gamma^{-1}_{2,k}\xi)$ and $V\in\partial[\mathcal{P}_{\gamma_{1,k}^{-1}h_k}](x_k-\gamma_{1,k}^{-1}A^{\top}\xi)$ are symmetric positive semidefinite matrices. For each $j\in\mathbb{N}$, combining $\varrho_j>0$ infers that $\varrho_jI+\gamma^{-1}_{2,k}U+\gamma_{1,k}^{-1}AVA^{\top}$ is positive definite, so all elements in $\widehat{\partial}^2\Upsilon_{k,j}(\xi)$ are nonsingular. Thus, the following semismooth Newton method is applied to problem \eqref{PPA-DUAL} to seek an inexact root $\xi^{j+1}$. The global and local convergence analysis of this algorithm can be found in reference \cite[Theorem 3.3-3.4]{ZhaoSunToh2010}. The detailed iterative steps of the semismooth Newton method for solving \eqref{PPA-DUAL-root} are given below.

%\bigskip

\setlength{\fboxrule}{0.8pt}
\noindent
%\fbox{
%	\parbox{0.96\textwidth}
	{
		%----------------------------------------------------------------------------------------------Algorithm
		\begin{algorithm} \label{SNCG}({\bf  Semismooth Newton method(SNCG) for system \eqref{PPA-DUAL-root}})
			\begin{description}
				\item[Initialization] Fix~$k,j\in\mathbb{N}$. Choose $\underline{\eta}>0$, $\beta\in(0,1)$, $\varsigma\in(0,1]$ and $0<c_1<c_2<1$. Let $\xi^0=\xi^j$ and set $l=0$.
				
				\item[{\rm{\bf while}}] termination conditions are not met {\rm{\bf do}}
				
				\item[S.1] Choose~$U^l\in\partial[\mathcal{P}_{\gamma^{-1}_{2,k}f}](z^k+\gamma^{-1}_{2,k}\xi^l)$ and $V^l\in\partial[\mathcal{P}_{\gamma_{1,k}^{-1}h_k}](x_k-\gamma_{1,k}^{-1}A^{\top}\xi^l)$. Set $W^l=\varrho_j I+\gamma^{-1}_{2,k}(I-U^l)+r_{1,k}^{-1}A[I-V^l]A^{\top}$. Solving the following linear system
				\begin{equation}\label{SNCG-dj}
					W^ld=-\nabla\Upsilon_{k,j}(\xi^l)
				\end{equation}
				yields~$d^l\in\mathbb{R}^p$ such that $\|W^ld+\nabla\Upsilon_{k,j}(\xi^l)\|\leq\min(\underline{\eta},\|\nabla\Upsilon_{k,j}(\xi^l)\|^{1+\varsigma})$.
				
				\item[S.2] Set $\alpha_l=\beta^{m_l}$, where $m_l$ is the smallest nonnegative integer $m$ that satisfies the following relationships
				\begin{align} \Upsilon_{k,j}(\xi^l+\beta^md^l)\leq\Upsilon_{k,j}(\xi^l)+c_1\beta^m\langle\nabla\Upsilon_{k,j}(\xi^l),d^l\rangle,
				\end{align}
				\begin{align}
					|\langle\nabla\Upsilon_{k,j}(\xi^l+\beta^md^l),d^l\rangle|\leq c_2|\langle\nabla\Upsilon_{k,j}(\xi^l),d^l\rangle|.
				\end{align}
				\item[S.3] Set~$\xi^{l+1}=\xi^l+\alpha_ld^l$ and $l\leftarrow l+1$, and return to step S.1.
				
				\item[{\rm{\bf end~while}}]
				
			\end{description}
		\end{algorithm}
	}
%}

%\bigskip

%\medskip
\begin{remark}
	When we are solving the linear system \eqref{SNCG-dj}, based on the size of the sample size $n$ and the number of features $p$ of the problem, if $(n,p)$ is of a small to medium scale, the direct method can be used to solve it. That is, in this paper, the Sherman Morrison Woodbury formula is used as follows:
	\begin{align*}
	&(I_p+A^{\top}A)^{-1}=I_p-A^{\top}(I_n+AA^{\top})^{-1}A,\\
	&(I_n+AA^{\top})^{-1}=I_n-A(I_p+A^{\top}A)^{-1}A^{\top},
	\end{align*}
	where if $n<p$, then Cholesky decomposition can be used to process $I_n+AA^{\top} $; If $n>p$, this article discusses $I_p+A^{\top}A$ for processing. Of course, in small-scale problems, this article can also use the conjugate gradient method to solve linear systems \eqref{SNCG-dj}. If $(n,p)$ is large-scale, the conjugate gradient method can be used.
	
	Next, the termination conditions of algorithm \ref{SNCG} in this article are considered. For $k\in \mathbb{N}\cup\{0\}$, let $\xi^{k,*}$ be the unique root of the system \eqref{nonsmooth system}. Set $x^{k,*}=\mathcal{P}_{r_k^{-1}h_k}(r_k^{-1}(\gamma_{1,k}x^k-A^{\top}\xi^{k,*}))$ and~$z^{k,*}=\mathcal{P}_{\gamma^{-1}_{2,k}f}(z^k+\gamma^{-1}_{2,k}\xi^{k,*})$. Then $Ax^{k,*}-z^{k,*}-b=0$, that is, ~$(x^{k,*},z^{k,*})$ is a feasible solution of the problem~\eqref{eq-PMM-penalty-gz}, and it is obtained that
	\begin{align*}
	&f_1(z^{k,*})+\frac{\mu}{2}\|x^{k,*}\|^2+h_k(x^{k,*})+\frac{\gamma_{1,k}}{2}\|x^{k,*}-x^k\|^2\\
	&+\frac{\gamma_{2,k}}{2}\|z^{k,*}-z^k\|^2+\Psi_k(\xi^{k,*})=\frac{\|x^k\|^2}{2}+\frac{\|z^k\|^2}{2}.
	\end{align*}
	Thus, the dual gap is ${\rm val_{gap}}:=\frac{\|x^k\|^2}{2}+\frac{\|z^k\|^2}{2}$, and then the termination rules of~\ref{SNCG} can be set as $\frac{|{\rm val_{gap}}|}{1+\|b\|}\leq\epsilon^k_{\rm SNCG}$ and $\frac{\|\Psi_k(\xi^{k,*})\|}{1+\|b\|}\leq\epsilon^k_{\rm SNCG}$.
\end{remark}

\begin{remark}\label{weighted-L2}
	This section provides the proximal mapping of function $h_k(x)$ and the Clarke Jacobian of this proximal mapping. For arbitrarily given $\mu\geq0$ and any $x\in\mathbb{R}^p$, set $h_k(x):=\sum_{i=1}^mv_i^k\|x_{J_i}\|+\frac{\mu}{2}\|x\|^2$. Then for any $z\in\mathbb{R}^p$, it holds $[\mathcal{P}_{\gamma^{-1}}h_k(z)]_{J_i}=\frac{\gamma}{\gamma+\mu}\max(1-v_i^k/(\gamma\|z_{J_i}\|))z_{J_i}$
	and its corresponding Clarke Jacobian is
	\begin{align}
		&\partial_C(\mathcal{P}_{\gamma^{-1}}h_k)(z)=\nonumber\\
		&\left\{{\rm diag}(\xi_{J_1},\ldots,\xi_{J_m})\mid\xi_{J_i}=\begin{cases}\label{weighted-l2}
			\frac{\gamma}{\gamma+\mu}(I-M),\ {\rm if}~|z_i|>\overline{c},\\
			\frac{\gamma\alpha}{\gamma+\mu}(I-M),\ {\rm if}~|z_i|=\overline{c},\\
			\{0\},\ {\rm if}~|z_i|<\overline{c}.
		\end{cases}\right\}
	\end{align}
	where $M=\frac{1}{\gamma}(\frac{1}{\|z_{J_i}\|}I-\frac{zz^{\top}}{\|z_{J_i}\|^3})$,~$\overline{c}=\gamma^{-1}v_i^k$ for $\alpha\in[0,1]$.
\end{remark}
In this subsection, the proximal alternating direction method of multiplier(pADMM) is also used to solve the problem~\eqref{sr-gopt}, which can be equivalently written as
\begin{align}\label{p-sim-opt}
	&\mathop{\min}_{x\in\mathbb{R}^p,z\in\mathbb{R}^n}\{\|z\|_q/\sqrt{n}+\frac{\mu}{2}\|x\|^2\nonumber\\
	&\qquad+\sum_{i=1}^m\lambda(1-w_i^k)\|x_{J_i}\|):Ax-z-b=0\}.
\end{align}
For a given parameter $\sigma>0$, the augmented Lagrange function of \eqref{p-sim-opt} is
\begin{align*}
&\mathcal{L}_\sigma(x,z;\xi)=\|z\|_q/\sqrt{n}+\frac{\mu}{2}\|x\|^2+\sum_{i=1}^m\lambda(1-w_i^k)\|x_{J_i}\|\\
&\qquad+\langle \xi,Ax-z-b\rangle+\frac{\sigma}{2}\|Ax-z-b\|^2.
\end{align*}
This paper will provide the detailed iterative steps of pADMM for solving problem \eqref{p-sim-opt}.

%\bigskip

\setlength{\fboxrule}{0.8pt}
\noindent
%\fbox{\parbox{0.96\textwidth}
	{
		%----------------------------------------------------------------------------------------------Algorithm
		\begin{algorithm} \label{spADMM-pri}(pADMM for solving~\eqref{p-sim-opt})
			\begin{description}
				\item[Initialization:] Choose parameters~$\sigma>0$,~$\gamma\geq\sigma\|A^{\top}A\|$ and $\tau\in(1,\frac{1+\sqrt{5}}{2})$, pick the initial point~$(x^0,z^0;\xi^0)\in\mathbb{R}^p\times\mathbb{R}^n\times\mathbb{R}^n$ with~$x^0=x^{k-1}$. Set~$j=0$.
				\item[{\rm{\bf while}}]termination conditions do not satisfies~{\rm{\bf do}}
				\item[S.1]  Computing the following convex optimization problem
				\begin{equation}\label{z-part}
					z^{j+1}=\mathop{\arg\min}_{z\in\mathbb{R}^n}
					\mathcal{L}_{\sigma}(z,x^{j};\xi^j).
				\end{equation}

				\item[S.2] Computing the following convex optimization problem
				\begin{equation}\label{x-part}
					x^{j+1}=\mathop{\arg\min}_{x\in\mathbb{R}^p}
					\mathcal{L}_{\sigma}(z^{j+1},x;\xi^j)+\frac{1}{2}\|x-x^j\|^2_{\gamma I-\sigma A^{\top}A}.
				\end{equation}
				\item[S.3] Updating multiplier by following equation
				\begin{equation}\label{xi-part}
					\xi^{j+1}:=\xi^j+\tau\sigma(Ax^{j+1}-z^{j+1}-b).
				\end{equation}
				
				\item[S.4] Set~$j\leftarrow j+1$ and return to step~{\rm{\bf S.1}}.
				\item[{\rm{\bf end~while}}]
			\end{description}
		\end{algorithm}
	}
%}

%\bigskip

\begin{remark}\label{t1}
	(i)~This paper first considers the subproblem \eqref{z-part} on $z$, i.e.,
	\begin{align*}		z^{j+1}=\mathop{\arg\min}_{z\in\mathbb{R}^n}\Big\{\frac{\|z\|}{\sqrt{n}}+\frac{\sigma}{2}\|z-u^j\|^2\Big\}.
	\end{align*}
	It's easy to obtain
	\begin{equation}
		z^{j+1}=\begin{cases}
			0 &{\rm if}~\|u^j\|\leq1/(\sigma\sqrt{n}),\\
			\frac{u^j(\|u^j\|-1/({\sigma\sqrt{n}}))}{\|u^j\|} &{\rm if} ~\|u^j\|>1/(\sigma\sqrt{n}),
		\end{cases}
	\end{equation}
	where $u^j=Ax^j-b+\sigma^{-1}\xi^j$.

	(ii)~Next this paper considers the subproblem \eqref{x-part} on $x$ in the following
	\begin{align*}		x^{j+1}=\mathop{\arg\min}_{x\in\mathbb{R}^p}\Big\{\frac{\gamma}{2}\|x-\widetilde{x}\|^2+h_k(x)\Big\},
	\end{align*}
	where $\widetilde{x}:=x^j-\gamma^{-1}A^{\top}(\sigma(A x^j-z^{j+1}-b)+\xi^j)]$, $h_k(x):=\sum_{i=1}^mv_i^k\|x_{J_i}\|+\frac{\mu}{2}\|x\|^2$ and $v^k:=\lambda(1-w^k)$ for $i=1,\ldots,m$, it yields from Remark \ref{weighted-L2} that $[\mathcal{P}_{\gamma^{-1}}h_k(g)]_{J_i}=\frac{\gamma}{\gamma+\mu}\max(1-\frac{v_i^k}{\gamma\|g_{J_i}\|},0)g_{J_i}$.
\end{remark}

\subsection{Numerical testing for~$\ell_1$ loss function}
The numerical test is divided into two parts. The first part is the synthetic data as follows. In this subsection, the synthetic data is first generated through the observation model \eqref{observation model}, group sparsity rate $s^*$ satisfies $s^*\ll p$. Each row of the input matrix $A\in\mathbb{R}^{n\times p}$ follows a multivariate normal distribution $N(0,\Sigma)$, $\Sigma$ is the corresponding covariance matrix, and there are several ways to generate the covariance matrix $\Sigma$: (1) $\Sigma=I$; (2) Autoregressive structure: $\Sigma=(0.5^{|i-j|})_{ij}$; (3) Autoregressive structure: $\Sigma=(0.8^{|i-j|})_{ij}$; (4) Composite symmetric structure: $\Sigma=(\alpha+(1-\alpha)\mathbb{I}_{i=j})_{i}$ with $\alpha=0.6$; (5) Composite symmetric structure: $\Sigma=(\alpha+(1-\alpha)\mathbb{I}_{i=j})_{i}$ with $\alpha=0.8$.
The non-zero elements of the noise vector $\varpi$ are generated as follows: (1) Normal distribution $N(0,100)$; (2) $t$-Distribution $\sqrt{2}\times t_4$ with degrees of freedom 4; (3) The Cauchy distribution with a density function of $d(u)=\frac {1} {\Pi(1+u^2)}$; (4) Mixed normal distribution~$N(0,\sigma^2)$, where~$\sigma\sim{\rm Unif} (1,5)$; (5) A Laplace distribution with a density function of $d(u)=0.5 {\rm exp}(-|u|)$. The authors randomly select the $\overline{r}$ groups from the $m$ groups, such as $\{m_1, m_2,\ldots, m_{\overline {r}}\}$, as the support for~$x^*$, $x^*_{J_j}=5{\rm randn}(|J_i|,1)-0.5$ and~$b=Ax^*+\varpi$. The data is grouped in the same way as in reference \cite{Zhang2020}.

Before testing, an example is used to test the impact of the noise vector $\varpi$ on the error between the solutions generated by the above two algorithms and the true solution. If the number of groups for $\varpi$ is $m$ and the group support is $\mathcal{I}$, then the group sparsity rate of $\varpi$ is $(m-|\mathcal{I}|)/m$.

\begin{example}
	Given~$(n,p)=(1000,5000)$,~the chosen covariance matrix is an autoregressive structure:~$\Sigma=(0.5^{|i-j|})_{ij}$, true solution is~$x^*_{J_j}=5 {\rm randn}(|J_i|,1)-0.5$.
\end{example}
The figure fig1-1 below shows the curves of the relative errors obtained by two algorithms under the same parameter $\lambda$ and different noise group sparsity. From the figure, it can be seen that the relative errors obtained by the two algorithms decrease with the increase of group sparsity, which is consistent with the conclusion of the theorem verified in this paper. In addition, the relative error obtained by the Algorithm \ref{PMM} is smaller than that obtained by pADMM. The figure fig1-2 shows the relative error curves obtained by the two algorithms at different $\lambda$. From the graph, it can be seen that if the group sparsity of the noise vector is properly controlled, there will be an interval of $\lambda$ between the two algorithms, so that within this interval, the relative errors obtained by the algorithms will tend to stabilize.
%\\

Next, the authors present 20 examples of different combinations of four covariance matrices and five types of noise generation methods, and the conclusions obtained under two algorithms. The conclusion provides the values of $\lambda$, the group sparsity of the solutions obtained by the two algorithms, the objective function value, the relative error, and the running time (excluding the time to obtain the initial point). From the table \ref{result-synthetic-l1}, it can be seen that the relative error and time of the solutions obtained by the PMM algorithm are smaller than those obtained by~pADMM. Please refer to the following table for details.

In the following, denote the number of nonzero of solution $\widehat{x}$ by ${\rm nnz}:=\sum_{i=1}^p\mathbb{I}\{|{x}_i|>10^{-8}\|x\|_\infty\}$, denote the number of groups by ${\rm ng}$, $x^{\rm out}$ represents the output of the algorithm, ${\rm L2err}:=\frac{\|x^{\rm out}-x^*\|}{\|x^*\|}$ is relative error $\ell_2$. ${\rm pobj}$ is the objective value of problem, $\eta_{\rm kkt}$ represents kkt residual, ${\rm Time(s)}$ represents time (seconds). Let the termination accuracy of PMM be $\epsilon_{\rm PMM}=10^{-5}$, the termination accuracy of SNCG be~$\epsilon_{\rm SNCG}=10^{-8}$ and the termination accuracy of pADMM be $\epsilon_{\rm pADMM}=10^{-5}$. In the synthetic data set, the following parameters are selected: $a=4$, $\mu=10^{-8}$, $\sigma=1$, $\tau=1.618$, $\rho=2$, $\varrho=1/1.4$ and $\lambda_0={\gamma}_1\max(10^{-6},0.05\max({\rm abs}(A^{\top}b)))$ as the value at which the initial point is solved. Also, the authors select $\lambda={\gamma}_2\lambda_0$, the specific values of ${\gamma}_1,{\gamma}_2$ correspond to the table, $\widetilde{\gamma}_{1,0}=10$,~$\widetilde{\gamma}_{2,0}=\widetilde{\gamma}_{1,0}$, $\underline{\gamma}_1=\underline{\gamma}_2=10^{-6}$, $\gamma_{1,0}=\gamma_{2,0}=\max(10^{-4},\widetilde{\gamma}_{1,0})$.

The second part of the test is aimed at the large-scale LIBSVM dataset~$(A,b)$ \cite{Chang2011} originated from the UCI database \cite{Lichman}. Similar to literature \cite{Li2018}, this paper also uses method in \cite{Huang2010} to expand the features of these data through polynomial basis functions. The last digit in these data set represents the order of the basis function, for example, the examples we tested like abalone7, bodyfat7, housing7, mpg7, pyrim5 and space ga9. In testing LIBSVM data, the authors select parameters $a=6$, $\mu=10^{-8}$, $\sigma=1.168$, $\varrho=1/1.4$, the spacific values of $\lambda=\overline{\gamma}\max(10^{-6},0.05\max({}(A^{\top}b)))$, $\overline{\gamma}$ correspond to the table, select $\rho=\max(1,6/\|x_0\|_\infty)$. In examples 1, 2, and 8, this paper chooses $\widetilde{\gamma}_{1,0}=0.001$ and picks $\widetilde{\gamma}_{1,0}=0.01$ in other examples. In addition, $\widetilde{\gamma}_{2,0}=0.01\widetilde{\gamma}_{1,0}$ and $\gamma_{1,0}=\gamma_{2,0}=\max(10^{-4},\widetilde{\gamma}_{1,0})$ are selected in all examples.

Table~\ref{result-lib-l1} compares the numerical results of the pADMM and the PMM in solving the group zero-norm regularized problem with $\ell_1$ loss. For each composite data set, we fit the model onto the training data set and use the validation data set to select the regularization parameter $\lambda$. As you can see, in most cases, PMM not only achieves better group sparsity than pADMM, but also it takes less time.

\subsection{Numerical testing for $\ell_2$ loss function}
Numerical testing is also divided into two main parts. The first part is synthetic data, which is generated by the observation model \eqref{observation model}. Each row of the input matrix~$A\in\mathbb{R}^{n\times p}$ follows a multivariate normal distribution~$N(0,\Sigma)$, where~$\Sigma$ is the corresponding covariance matrix. The true predicted variable~$x^*$ is an unknown sparse vector, and the group sparsity~$s^*$ satisfies $s^*\ll p$. In addition, the data matrix generation method is $A={\rm randn}(n,p)$, this paper randomly selects~$\overline{r}$ groups from $m$ groups, such as $\{m_1, m_2,\ldots, m_{\overline {r}}\}$ as the support for~$x^*$, $x^*_{J_j}=5 {\rm randn}(|J_i|,1)-0.5$, noise vector $\varpi={\rm randn}(n, 1)/\|{\rm randn}(n,1)\|$ and $b=Ax^*+{\rm randn}(n,1)/\|{\rm randn}(n,1)\|$. The data is grouped in the same way as in reference \cite{Zhang2020}.

Similar to the previous section, this section also uses an example to test the impact of the noise vector~$\varpi$ on the error between the solutions generated by the above two algorithms and the true solutions. If the number of groups for $\varpi$ is~$m$ and the group support is~$\mathcal{I}$, then the group sparsity of~$\varpi$ is~$(m-|\mathcal{I}|)/m$.

\begin{example}
	Set~$(n,p)=(ceil(4*4096/15),4*4096)$, the covariance matrix is taken as: $\Sigma=(0.5^{|i-j|})_{ij}$, the true solution is~$x^*_{J_j}=5 {\rm randn}(|J_i|,1)-0.5$.
\end{example}
The figure fig2-1 below shows the curves of the relative errors obtained by two algorithms under the same parameter~$\lambda$~and different noise group sparsity. From the figure, it can be seen that the relative errors obtained by the two algorithms decrease with the increase of group sparsity, which is consistent with the conclusion of the theorem verified in the article. In addition, the relative error obtained by the PMM~algorithm~is smaller than that obtained by pADMM. The figure fig2-2 on the right shows that if the group sparsity of the noise vector is properly controlled, there will be an interval of~$\lambda$, within which the relative errors obtained by both algorithms will tend to stabilize.
%\\
%\includegraphics[height=0.54\textwidth]{figs/fig1_1.jpg}
%\includegraphics[height=0.54\textwidth]{figs/fig1_2.jpg}
%\includegraphics[height=0.54\textwidth]{figs/fig2_1.jpg}
%\includegraphics[height=0.54\textwidth]{figs/fig2_2.jpg}

The real dataset tested in the second part is the same as in the previous section. Denote~the objective values by ${\rm pobj}$,~let $\eta_{\rm kkt}$ be~kkt residual. Termination accuracy of PMM is $\epsilon_{\rm PMM}=10^{-7}$,~the termination accuracy of SNCG is~$\epsilon_{\rm SNCG}=10^{-8}$,~the termination accuracy of pADMM is $\epsilon_{\rm pADMM}=10^{-5}$. In the synthesized dataset, this section selects the following parameters: $a=4$, $\mu=10^{-8}$,~$\sigma=1$,~$\tau=1.618$,~$\rho=2$,~$\varrho=1/1.4$, $\widetilde{\gamma}_{1,0}=0.01$,~$\widetilde{\gamma}_{2,0}=0.1\widetilde{\gamma}_{1,0}$,~$\gamma_{1,0}=\gamma_{2,0}=\max(10^{-4},\widetilde{\gamma}_{1,0})$,~$\underline{\gamma}_1=\underline{\gamma}_2=10^{-6}$, where the specific values of $\lambda=\overline{\gamma}\max(10^{-6},0.05\max({\rm abs}(A^{\top}b)))$ and $\overline{\gamma}$ are listed in the table. In testing the LIBSVM data, the corresponding parameters are selected: $a=6$, $\mu=10^{-8}$, $\sigma=1.168$, $\varrho=1/1.4$, $\lambda=\overline{\gamma}\max(10^{-6},0.05\max({\rm abs}(A^{\top}b)))$, and the specific values of $\overline{\gamma}$ and $\rho=\max(1,6/\|x_0\|_\infty)$ are listed in the table. In examples 1, 2, and 8, the authors choose $\widetilde{\gamma}_{1,0}=0.001$ and choose $\widetilde{\gamma}_{1,0}=0.01$ in other examples. In addition, $\widetilde{\gamma}_{2,0}=0.01\widetilde{\gamma}_{1,0}$ and $\gamma_{1,0}=\gamma_{2,0}=\max(10^{-4},\widetilde{\gamma}_{1,0})$ are selected in all examples.

This subsection compares the numerical results of the pADMM and the PMM in solving the group zero-norm regularized problem with the square root loss. In the PMM proposed in this paper, the authors use it as the initial point for the second stage, if $\eta_{kkt}<10^{-4}$, PMM stops. If $\eta_{kkt}<10^{-6}$ or when the algorithm reaches the preset maximum number of iterations (500~for PMM and 10000 for pADMM) or the preset maximum running time of 4 hours, the algorithm will terminate. For each composite dataset, fit the model onto the training dataset and use the validation dataset to select the regularization parameter $\lambda$. The test results of SCAD regularized PMM and pADMM on the synthetic dataset and real test set are shown in Table~\ref{result-synthetic-l2} and Table~\ref{result-lib-l2}, respectively. As you can see, in most cases, PMM not only obtains better objective function values than pADMM, but also achieves better group sparsity, although sometimes it takes more time.

\includegraphics[height=0.54\textwidth]{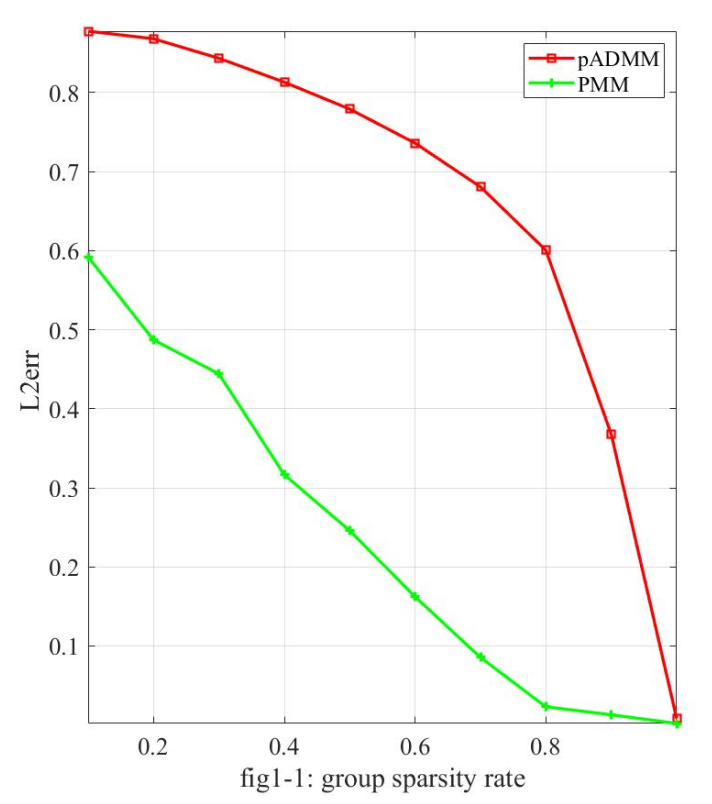}
\includegraphics[height=0.54\textwidth]{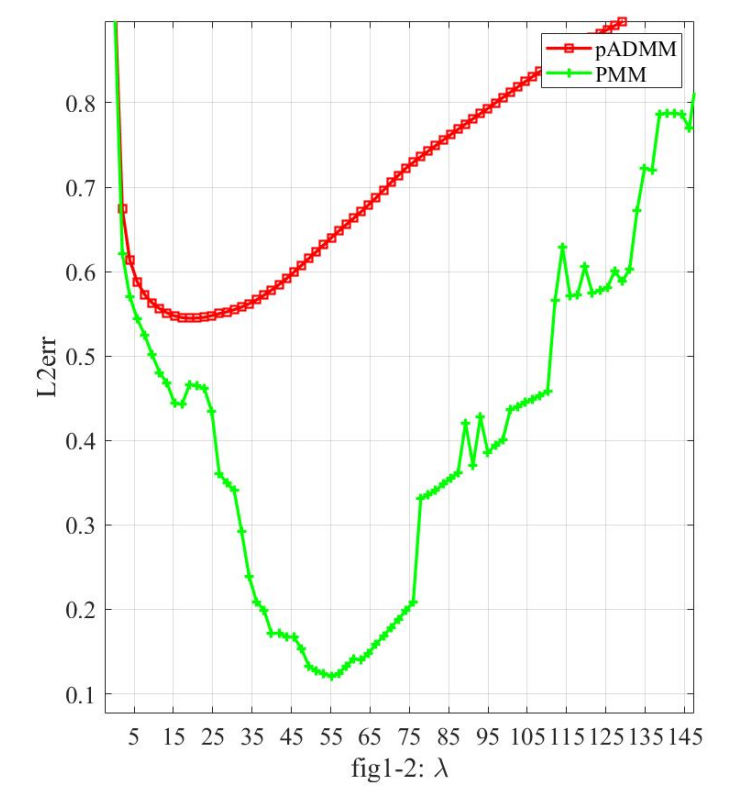}
\includegraphics[height=0.54\textwidth]{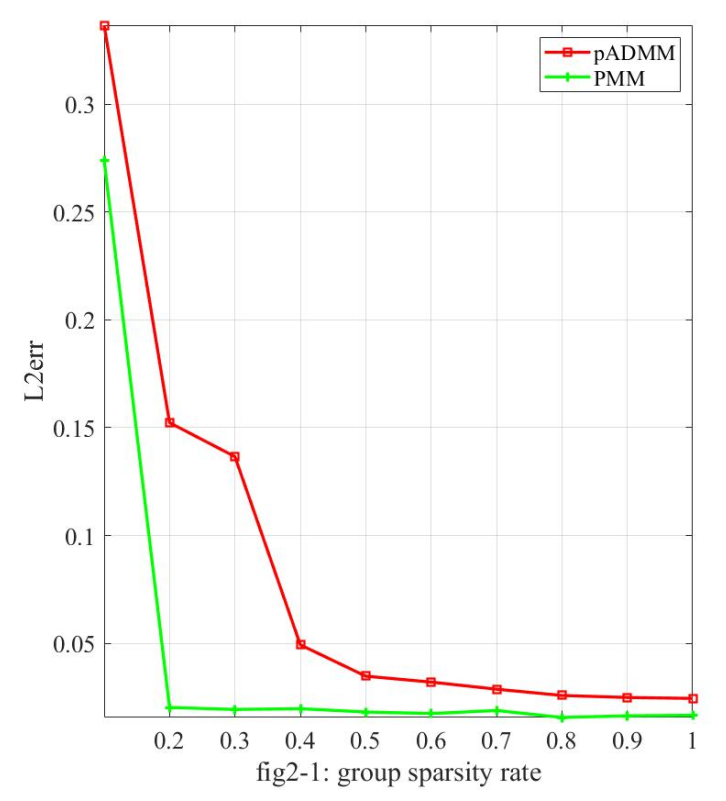}
\includegraphics[height=0.54\textwidth]{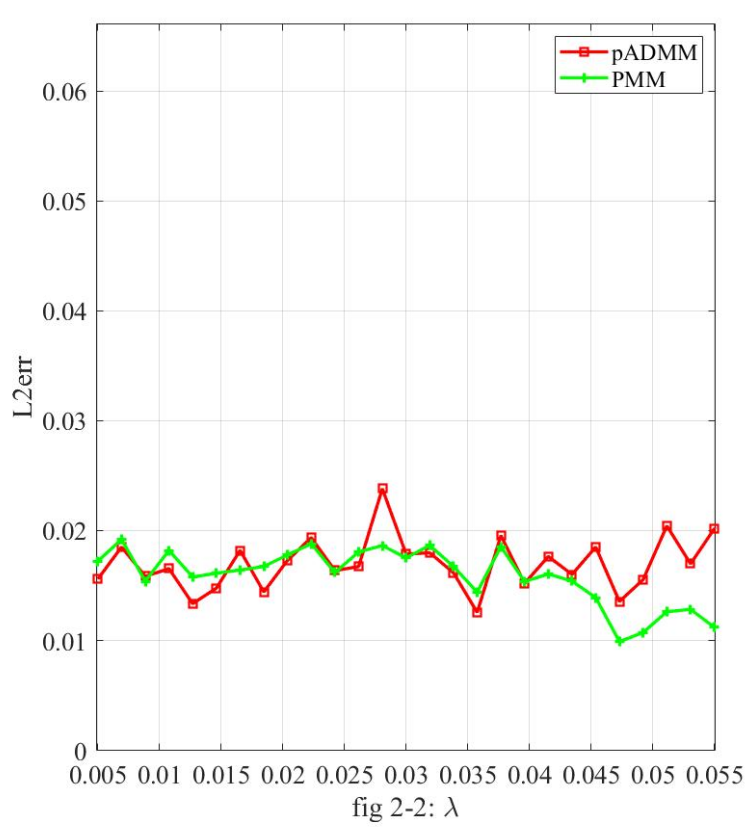}

\begin{sidewaystable}
	%\begin{table}[h]
	\caption{Performance of two algorithms on the synthetic data set}
	\label{result-synthetic-l1}
	\centering
	\scalebox{0.85}{
		\begin{tabular}{c|c|c|c|cc|cc|cc}
		%	\hline
			\hline
			\multirow{2}*{Probname}   & \multirow{2}*{ng}     & \multirow{2}*{${\gamma}_1$} & \multirow{2}*{${\gamma}_2$}
			& \multicolumn{2}{ c| }{L2err}   & \multicolumn{2}{ c| }{pobj}     &\multicolumn{2}{ c }{Time(s)}    \\
			&  & & & PMM&pADMM  &  PMM&pADMM  &   PMM&pADMM \\
			\hline
			\multirow{2}*{(1,1)}  & \multirow{2}*{1} &  \multirow{2}*{0.2} &  \multirow{2}*{2}
			&  \multirow{2}*{8.6032e-12}     & \multirow{2}*{5.9192e-04}     & \multirow{2}*{2.4773}
			&  \multirow{2}*{2.4774}  & \multirow{2}*{6.775} & \multirow{2}*{8.947}  \\
			\multirow{2}*{(1,2)}  & \multirow{2}*{1} & \multirow{2}*{0.1}  &  \multirow{2}*{2}
			&  \multirow{2}*{5.3186e-04}     & \multirow{2}*{3.7698e-04}     & \multirow{2}*{0.453}
			&  \multirow{2}*{0.453}  & \multirow{2}*{5.326} & \multirow{2}*{34.873} \\
			\multirow{2}*{(1,3)}  & \multirow{2}*{1} &  \multirow{2}*{0.2}  &  \multirow{2}*{2}
			&  \multirow{2}*{1.95e-07}     & \multirow{2}*{9.34e-04}     & \multirow{2}*{3.2621}
			&  \multirow{2}*{3.2621} & \multirow{2}*{1.514} & \multirow{2}*{8.791} \\
			\multirow{2}*{(1,4)}   & \multirow{2}*{1} & \multirow{2}*{0.2}  &  \multirow{2}*{2}
			&  \multirow{2}*{2.09e-2}     & \multirow{2}*{2.28e-02}     & \multirow{2}*{0.7445}
			&  \multirow{2}*{0.7446}   & \multirow{2}*{2.144}     & \multirow{2}*{8.966} \\
			\multirow{2}*{(1,5)} & \multirow{2}*{1} &  \multirow{2}*{0.15} &  \multirow{2}*{2}
			&  \multirow{2}*{6.9132e-09}     & \multirow{2}*{1.1e-03}     & \multirow{2}*{0.3253}
			&  \multirow{2}*{0.3254} & \multirow{2}*{1.395}  & \multirow{2}*{9.052} \\
			& &  &  & & & & & &  \\
			\hline
			\multirow{2}*{(2,1)} & \multirow{2}*{1} &  \multirow{2}*{0.2} &  \multirow{2}*{3}
			&  \multirow{2}*{4.2456e-12} & \multirow{2}*{2.11e-01} & \multirow{2}*{2.4773}
			&  \multirow{2}*{4.1884} & \multirow{2}*{1.629} & \multirow{2}*{4.1884} \\
			\multirow{2}*{(2,2)}   & \multirow{2}*{1} &  \multirow{2}*{0.15}  &  \multirow{2}*{2}
			&  \multirow{2}*{3.2397e-17}     & \multirow{2}*{7.9e-03}     & \multirow{2}*{0.4529}
			&  \multirow{2}*{0.4967}  & \multirow{2}*{0.836}& \multirow{2}*{0.754}  \\
			\multirow{2}*{(2,3)}   & \multirow{2}*{1} &  \multirow{2}*{0.2}   &  \multirow{2}*{3}
			&  \multirow{2}*{4.04e-18}     & \multirow{2}*{6.5e-03}     & \multirow{2}*{3.2621}
			&  \multirow{2}*{3.2996}  & \multirow{2}*{0.957}& \multirow{2}*{0.723}  \\
			\multirow{2}*{(2,4)}   & \multirow{2}*{1} &  \multirow{2}*{0.15}   &  \multirow{2}*{3}
			&  \multirow{2}*{1.3095e-16}     & \multirow{2}*{2.44e-02}     & \multirow{2}*{0.743}
			&  \multirow{2}*{0.8949}  & \multirow{2}*{6.563}& \multirow{2}*{34.742}  \\
			\multirow{2}*{(2,5)}   & \multirow{2}*{1} &  \multirow{2}*{0.15}   &  \multirow{2}*{3}
			&  \multirow{2}*{3.70e-17}     & \multirow{2}*{8.3e-03}     & \multirow{2}*{0.3253}
			&  \multirow{2}*{0.3756}  & \multirow{2}*{0.779}& \multirow{2}*{0.628}  \\
			& &  &  & & & & & &  \\
			\hline
			\multirow{2}*{(3,1)}   & \multirow{2}*{1} &  \multirow{2}*{0.15}  &  \multirow{2}*{3}
			&  \multirow{2}*{7.43e-12}     & \multirow{2}*{4.57e-02}     & \multirow{2}*{2.4773}
			&  \multirow{2}*{2.5938}  & \multirow{2}*{0.18}& \multirow{2}*{1.802}  \\
			\multirow{2}*{(3,2)}   & \multirow{2}*{1} &  \multirow{2}*{0.18}   &  \multirow{2}*{5}
			&  \multirow{2}*{3.87e-09}     & \multirow{2}*{5.86e-02}     & \multirow{2}*{0.4529}
			&  \multirow{2}*{0.6414}  & \multirow{2}*{1.68}& \multirow{2}*{1.81}  \\
			\multirow{2}*{(3,3)}   & \multirow{2}*{1} &  \multirow{2}*{0.1}   &  \multirow{2}*{3}
			&  \multirow{2}*{1.56e-10}     & \multirow{2}*{6.65e-02}     & \multirow{2}*{3.2621}
			&  \multirow{2}*{3.542}  & \multirow{2}*{1.239}& \multirow{2}*{1.82}  \\
			\multirow{2}*{(3,4)}   & \multirow{2}*{1} &  \multirow{2}*{0.1}   &  \multirow{2}*{3}
			&  \multirow{2}*{7.907e-15}     & \multirow{2}*{5.34e-02}     & \multirow{2}*{0.743}
			&  \multirow{2}*{0.8999}  & \multirow{2}*{0.269}& \multirow{2}*{1.829}  \\
			\multirow{2}*{(3,5)}   & \multirow{2}*{1} &  \multirow{2}*{0.1}   &  \multirow{2}*{3}
			&  \multirow{2}*{6.2996e-12}     & \multirow{2}*{1.1e-01}     & \multirow{2}*{0.3253}
			&  \multirow{2}*{0.9805}  & \multirow{2}*{1.614}& \multirow{2}*{1.884}  \\
			& &  &  & & & & & &   \\
			\hline
			\multirow{2}*{(4,1)}   & \multirow{2}*{1} &  \multirow{2}*{0.1}  &  \multirow{2}*{2}
			&  \multirow{2}*{7.4911e-05}     & \multirow{2}*{1.0e-01}     & \multirow{2}*{2.4778}
			&  \multirow{2}*{2.4773}  & \multirow{2}*{5.579}& \multirow{2}*{25.994}  \\
			\multirow{2}*{(4,2)}    & \multirow{2}*{1} &  \multirow{2}*{0.15}   &  \multirow{2}*{3}
			&  \multirow{2}*{3.57e-12}     & \multirow{2}*{1.62e-07}     & \multirow{2}*{0.4529}
			&  \multirow{2}*{0.4529}  & \multirow{2}*{0.064}& \multirow{2}*{36.229}  \\
			\multirow{2}*{(4,3)}    & \multirow{2}*{1} &  \multirow{2}*{0.15}   &  \multirow{2}*{3}
			&  \multirow{2}*{3.68e-12}     & \multirow{2}*{6.99e-09}     & \multirow{2}*{3.2621}
			&  \multirow{2}*{3.2621}  & \multirow{2}*{0.054}& \multirow{2}*{4.156} \\
			\multirow{2}*{(4,4)}    & \multirow{2}*{1} &  \multirow{2}*{0.15}   &  \multirow{2}*{3}
			&  \multirow{2}*{2.89e-12}     & \multirow{2}*{6.58e-09}     & \multirow{2}*{0.743}
			&  \multirow{2}*{0.743}  & \multirow{2}*{0.062}& \multirow{2}*{7.175}  \\
			\multirow{2}*{(4,5)}    & \multirow{2}*{1} &  \multirow{2}*{0.15}   &  \multirow{2}*{3}
			&  \multirow{2}*{8.85e-14}     & \multirow{2}*{4.83e-02}     & \multirow{2}*{0.3253}
			&  \multirow{2}*{0.5773}  & \multirow{2}*{1.271}& \multirow{2}*{39.527}  \\
			& &  &  & & & & & &  \\
			\hline
			\multirow{2}*{(5,1)}   & \multirow{2}*{1} &  \multirow{2}*{0.1}  &  \multirow{2}*{2}
			&  \multirow{2}*{9.40e-13}     & \multirow{2}*{2.0e-03}     & \multirow{2}*{2.4773}
			&  \multirow{2}*{2.4923}  & \multirow{2}*{0.144}& \multirow{2}*{1.831}  \\
			\multirow{2}*{(5,2)}    & \multirow{2}*{1} &  \multirow{2}*{0.15}   &  \multirow{2}*{3}
			&  \multirow{2}*{5.96e-12}     & \multirow{2}*{3.94e-02}     & \multirow{2}*{0.4529}
			&  \multirow{2}*{0.5951}  & \multirow{2}*{0.349}& \multirow{2}*{9.359}  \\
			\multirow{2}*{(5,3)}    & \multirow{2}*{1} &  \multirow{2}*{0.15}   &  \multirow{2}*{3}
			&  \multirow{2}*{9.56e-12}     & \multirow{2}*{6.62e-02}     & \multirow{2}*{3.2621}
			&  \multirow{2}*{3.5052}  & \multirow{2}*{0.519}& \multirow{2}*{9.246} \\
			\multirow{2}*{(5,4)}    & \multirow{2}*{1} &  \multirow{2}*{0.15}   &  \multirow{2}*{3}
			&  \multirow{2}*{8.37e-12}     & \multirow{2}*{3.73e-02}     & \multirow{2}*{0.743}
			&  \multirow{2}*{0.9093}  & \multirow{2}*{0.54}& \multirow{2}*{8.973}  \\
			\multirow{2}*{(5,5)}    & \multirow{2}*{1} &  \multirow{2}*{0.15}   &  \multirow{2}*{3}
			&  \multirow{2}*{2.34e-08}     & \multirow{2}*{2.2e-03}     & \multirow{2}*{0.3253}
			&  \multirow{2}*{0.3328}  & \multirow{2}*{0.157}& \multirow{2}*{9.229}  \\
			& &  &  & & & & & &  \\
			\hline
	\end{tabular}}
	%\end{table}
\end{sidewaystable}

\begin{sidewaystable}
%\begin{table}[h]
	\caption{Performance of two algorithms on the synthetic data set}
	\label{result-synthetic-l2}
	\centering
	\scalebox{0.82}{
		\begin{tabular}{c|c|c|cc|cc|cc|cc }
		%	\hline
			\hline
			\multirow{2}*{Probname}   & \multirow{2}*{nnz}     & \multirow{2}*{$\overline{\gamma}$}
			& \multicolumn{2}{ c| }{L2err}   & \multicolumn{2}{ c| }{$\eta_{\rm kkt}$}     &\multicolumn{2}{ c| }{pobj}            & \multicolumn{2}{ c }{Time(s)} \\
			ng&   & & PMM&pADMM  &  PMM&pADMM  &   PMM&pADMM   &  PMM&pADMM \\
			\hline
			(334,5000)  & \multirow{2}*{8} &  \multirow{2}*{3}
			&  \multirow{2}*{1.25e-02}     & \multirow{2}*{2.26e-02}     & \multirow{2}*{1.2e-08}
			&  \multirow{2}*{3.35e-04}  & \multirow{2}*{5.24e-04} & \multirow{2}*{6.35e-04} & \multirow{2}*{0.969} & \multirow{2}*{18.311} \\
			500  & &    & & & & & &  &  & \\
			\hline
			(667,10000)  & \multirow{2}*{8} & \multirow{2}*{3}
			&  \multirow{2}*{9.3e-03}     & \multirow{2}*{1.79e-02}     & \multirow{2}*{3.72e-08}
			&  \multirow{2}*{1.51e-04}  & \multirow{2}*{2.63e-04} & \multirow{2}*{3.256e-04} & \multirow{2}*{2.534} & \multirow{2}*{71.754} \\
			500 & &   & & & &  & & &  & \\
			\hline
			(1000,15000)  & \multirow{2}*{8} &  \multirow{2}*{3}
			&  \multirow{2}*{7.3e-03}     & \multirow{2}*{1.22e-02}     & \multirow{2}*{2.68e-08}
			&  \multirow{2}*{9.69e-05} & \multirow{2}*{1.73e-04} & \multirow{2}*{2.04e-04} & \multirow{2}*{4.846} & \multirow{2}*{157.918}\\
			500  & &    & & & & & &  &  & \\
			\hline
			(1334,20000) & \multirow{2}*{8} &  \multirow{2}*{3}
			&  \multirow{2}*{5.8e-03}     & \multirow{2}*{1.36e-02}     & \multirow{2}*{6.46e-09}
			&  \multirow{2}*{9.07e-05}   & \multirow{2}*{1.33e-04} & \multirow{2}*{1.89e-04} & \multirow{2}*{6.95} & \multirow{2}*{298.104}\\
			500 & &    & & & &  & & &  & \\
			\hline
			(1667,25000)   & \multirow{2}*{8} & \multirow{2}*{5}
			&  \multirow{2}*{5.5e-03}     & \multirow{2}*{1.06e-02}     & \multirow{2}*{2.47e-08}
			&  \multirow{2}*{6.64e-05}   & \multirow{2}*{1.04e-04}     & \multirow{2}*{1.35e-04} &\multirow{2}*{10.878} &  \multirow{2}*{413.387}\\
			500   & &   & & & &  & & &  & \\
			\hline
			(2000,30000)& \multirow{2}*{8} &  \multirow{2}*{5}
			&  \multirow{2}*{4.6e-03}     & \multirow{2}*{1.1e-02}     & \multirow{2}*{2.83e-08}
			&  \multirow{2}*{5.96e-05} & \multirow{2}*{8.86e-05}  & \multirow{2}*{1.25e-04} & \multirow{2}*{12.525} &\multirow{2}*{595.586}\\
			500  & &    & & & & & &  &  & \\
			\hline
		%	\hline
	\end{tabular}}
%\end{table}
\end{sidewaystable}

\begin{sidewaystable}
	\caption{Performance of two algorithms on the LIBSVM data set}
	\label{result-lib-l1}
	\centering
	\scalebox{0.85}{
		\begin{tabular}{c|c|c|c|cc|cc|cc}
			%\hline
			\hline
			\multirow{2}*{Probname}   & \multirow{2}*{ng}   & \multirow{2}*{${\gamma}_1$}   & \multirow{2}*{${\gamma}_2$}
			& \multicolumn{2}{ c| }{nnz}    &\multicolumn{2}{ c| }{pobj}   & \multicolumn{2}{c}{Time(s)}  \\
			&   &  & & PMM&pADMM  &  PMM&pADMM  &   PMM&pADMM   \\
			\hline
			E2006.train  & \multirow{2}*{150} & \multirow{2}*{5e-5} & \multirow{2}*{5}
			&  \multirow{2}*{5}     & \multirow{2}*{6}     & \multirow{2}*{0.2605}
			&  \multirow{2}*{0.2620}  & \multirow{2}*{278.3120} & \multirow{2}*{1.7031e+03} \\
			(16087,150360)  & &  &  & & & & & &  \\
			\hline
			E2006.test  & \multirow{2}*{150} &  \multirow{2}*{5e-5} & \multirow{2}*{8}
			& \multirow{2}*{7}  &  \multirow{2}*{13}  & \multirow{2}*{0.2333}     & \multirow{2}*{0.2340} & \multirow{2}*{63.45} & \multirow{2}*{350.61} \\
			(3308,150358) & & &   & & & &  & &  \\
			\hline
			log1p.E2006.train   & \multirow{2}*{1425} & \multirow{2}*{8e-3} & \multirow{2}*{3} &\multirow{2}*{7}
			& \multirow{2}*{33} &\multirow{2}*{0.1461}  & \multirow{2}*{0.2732}
			& \multirow{2}*{6.59e+03} &  \multirow{2}*{1.03e+04}  \\
			(16087,4272227)  & & &  & & & & & &  \\
			\hline
			log1p.E2006.test & \multirow{2}*{1425} &  \multirow{2}*{1e-4} & \multirow{2}*{1} & \multirow{2}*{3}
			&  \multirow{2}*{44}  & \multirow{2}*{0.0014}
			&  \multirow{2}*{0.11}  & \multirow{2}*{296.5630}     & \multirow{2}*{2.2866e+03}  \\
			(3308,4272226)  & & &  & & & &  & & \\
			\hline
			abalone7   & \multirow{2}*{21} & \multirow{2}*{1e-3} & \multirow{2}*{1}
			&  \multirow{2}*{3}  &  \multirow{2}*{8}    & \multirow{2}*{1.4477}     & \multirow{2}*{2.5991}
			& \multirow{2}*{344.8810}     & \multirow{2}*{4.2898e+03} \\
			(4177,6435)   & & &   & & & &  & & \\
			\hline
			bodyfat7  & \multirow{2}*{3876}  & \multirow{2}*{1e-2}  & \multirow{2}*{3}
			&  \multirow{2}*{3}     & \multirow{2}*{2}     & \multirow{2}*{0.0020}
			&  \multirow{2}*{0.0051}  & \multirow{2}*{342.0210}     & \multirow{2}*{1.7581e+03}\\
			(252,116280)  & &  &  & & & & & &   \\
			\hline
			housing7    & \multirow{2}*{150} &  \multirow{2}*{5e-3}  & \multirow{2}*{1}
			&  \multirow{2}*{16}     & \multirow{2}*{16}     & \multirow{2}*{3.3016}
			&  \multirow{2}*{4.8431}     & \multirow{2}*{265.4630}  & \multirow{2}*{5.7502e+03} \\
			(506,77520)  & &  &  & & & & & &   \\
			\hline
			mpg7   & \multirow{2}*{11} &  \multirow{2}*{5e-3}  & \multirow{2}*{1}
			&  \multirow{2}*{3}     & \multirow{2}*{9}     & \multirow{2}*{1.6688}
			&  \multirow{2}*{1.1006} & \multirow{2}*{1.1142} & \multirow{2}*{0.2730} \\
			(392,3432)  & &  &  & & & &  & &  \\
			\hline
			pyrim5  & \multirow{2}*{150} &  \multirow{2}*{2e-3}  & \multirow{2}*{10}
			&  \multirow{2}*{2}     & \multirow{2}*{9}     & \multirow{2}*{0.0024}
			&  \multirow{2}*{0.0012}   &\multirow{2}*{60.7730} & \multirow{2}*{3.1668e+03} \\
			(74,201376)   & &  &  & & & & & &   \\
			\hline
			space ga9   & \multirow{2}*{50} & \multirow{2}*{4e-3}  & \multirow{2}*{1}
			& \multirow{2}*{4}     & \multirow{2}*{8}  & \multirow{2}*{0.0743}
			& \multirow{2}*{0.0777} & \multirow{2}*{576.1090} & \multirow{2}*{3.9092e+03} \\
			(3107,5005)  & & &  & & & & & &  \\
		%	\hline
			\hline
	\end{tabular}}
\end{sidewaystable}

\begin{sidewaystable}
%\begin{table}[h]
	\caption{Performance of two algorithms on the LIBSVM data set}
	\label{result-lib-l2}
	\centering
	\scalebox{0.85}{
		\begin{tabular}{c|c|c|cc|cc|cc|cc }
			%\hline
			\hline
			\multirow{2}*{Probname}   & \multirow{2}*{ng}     & \multirow{2}*{$\overline{\gamma}$}
			& \multicolumn{2}{ c| }{nnz}   & \multicolumn{2}{ c| }{$\eta_{\rm kkt}$}     &\multicolumn{2}{ c| }{pobj}            & \multicolumn{2}{ c }{Time(s)} \\
			(n,p) &     & & PMM&pADMM  &  PMM&pADMM  &   PMM&pADMM     &  PMM&pADMM \\
			\hline
			E2006.train  & \multirow{2}*{150} & \multirow{2}*{0.9}
			&  \multirow{2}*{1}     & \multirow{2}*{2}     & \multirow{2}*{7.55e-10}
			&  \multirow{2}*{2.13e-04}  & \multirow{2}*{3e-03} & \multirow{2}*{3e-03} & \multirow{2}*{1.176} &\multirow{2}*{24.318}\\
			(16087,150360)  & &    & & & & & &  &  & \\
			\hline
			E2006.test  & \multirow{2}*{150} &  \multirow{2}*{0.9}
			& \multirow{2}*{2}  &  \multirow{2}*{5}  & \multirow{2}*{4.96e-05}     & \multirow{2}*{2.5e-04} & \multirow{2}*{6.5e-03} & \multirow{2}*{6.6e-03} &\multirow{2}*{2.073} & \multirow{2}*{4.55}\\
			(3308,150358) & &    & & & &  & & &  & \\
			\hline
			log1p.E2006.train   & \multirow{2}*{1425} & \multirow{2}*{3} &\multirow{2}*{1}& \multirow{2}*{2} &\multirow{2}*{2.37e-08}  & \multirow{2}*{5.01e-01}
			& \multirow{2}*{5.5e-03}
			&  \multirow{2}*{4.4e-03}   & \multirow{2}*{216.968}     & \multirow{2}*{78.395}\\
			(16087,4272227)  & &   & & & & & &  &  & \\
			\hline
			log1p.E2006.test & \multirow{2}*{1425} &  \multirow{2}*{5} & \multirow{2}*{1}
			&  \multirow{2}*{2}  & \multirow{2}*{1.82e-08}
			&  \multirow{2}*{2.78e-01}  & \multirow{2}*{1.14e-02}     & \multirow{2}*{7.3e-03} &\multirow{2}*{131.983} & \multirow{2}*{19.189} \\
			(3308,4272226)  & &   & & & &  & & &  & \\
			\hline
			abalone7   & \multirow{2}*{21} & \multirow{2}*{1}
			&  \multirow{2}*{3}     & \multirow{2}*{3}     & \multirow{2}*{5.49e-09}
			&  \multirow{2}*{5.4e-03} & \multirow{2}*{0.036}     & \multirow{2}*{0.036}
			&\multirow{2}*{2.951} & \multirow{2}*{6.268}  \\
			(4177,6435)   & &    & & & &  & & &  & \\
			\hline
			bodyfat7  & \multirow{2}*{3876}  & \multirow{2}*{1}
			&  \multirow{2}*{2}     & \multirow{2}*{2}     & \multirow{2}*{5.47e-11}
			&  \multirow{2}*{3.23e-01}  & \multirow{2}*{6.59e-04}     & \multirow{2}*{6.73e-04} & \multirow{2}*{2.074} &\multirow{2}*{0.6430}\\
			(252,116280)  & &    & & & & & &  &  & \\
			\hline
			housing7    & \multirow{2}*{150} &  \multirow{2}*{1}
			&  \multirow{2}*{5}     & \multirow{2}*{5}     & \multirow{2}*{ 2.14e-09}
			&  \multirow{2}*{3.3e-03}     & \multirow{2}*{0.1143}  & \multirow{2}*{0.1143} & \multirow{2}*{5.554} & \multirow{2}*{7.886}\\
			(506,77520)  & &    & & & &   &  & \\
			\hline
			mpg7   & \multirow{2}*{11} &  \multirow{2}*{5}
			&  \multirow{2}*{1}     & \multirow{2}*{3}     & \multirow{2}*{7.81e-06}
			&  \multirow{2}*{5.66e-01} & \multirow{2}*{9.86e-02} & \multirow{2}*{7.27e-01} & \multirow{2}*{64.773} & \multirow{2}*{152.617}\\
			(392,3432)  & &    & & & &  & & &  & \\
			\hline
			pyrim5  & \multirow{2}*{150} &  \multirow{2}*{18}
			&  \multirow{2}*{7}     & \multirow{2}*{7}     & \multirow{2}*{8.21e-09}
			&  \multirow{2}*{5.30e-02}   &\multirow{2}*{8.97e-04} & \multirow{2}*{8.79e-04} & \multirow{2}*{229.067}  & \multirow{2}*{0.182}\\
			(74,201376)   & &    & & & & & &  &  & \\
			\hline
			space ga9   & \multirow{2}*{50} & \multirow{2}*{1}
			& \multirow{2}*{6}     & \multirow{2}*{6}  & \multirow{2}*{2.85e-08}
			& \multirow{2}*{3.3e-03} & \multirow{2}*{1.9e-03} & \multirow{2}*{1.9e-03} &   \multirow{2}*{9.423} & \multirow{2}*{2.572}\\
			(3107,5005)  & &   & & & & & &  &  & \\
		%	\hline
			\hline
	\end{tabular}}
%\end{table}
\end{sidewaystable}

\bigskip
\section{Conclusions and Prospects}\label{section6}

This paper uses the proximal~MM~algorithm to solve the equivalent surrogate problem of the original group zero-norm regularized composite optimization problem, where the subproblems are solved using the semismooth Newton method. The convergence and convergence rate of the algorithm are analyzed and the statistical error bounds of the iterative sequence and its limits generated by the PMM method with respect to the true solution are constructed. The numerical results on the synthetic data set and the~UCI data set contribute to the understanding of theory of the PMM method and have practical applications for statistics and machine learning.

\section*{Acknowledgements}
The first two authors would like to express their sincere thanks to Prof. Shaohua Pan from the School of Mathematics, South China University of Technology for helpful suggestions on this paper. The authors extend their gratitude to Prof. Defeng Sun from The Hong Kong Polytechnic University for generously sharing the Semismooth Newton code and providing valuable insights regarding the numerical experiment.
%%%%%%%%%%%%%%%%%%%%%%%%%%%%%%%%%%%%%%%%%%%%%%%%%%%%%%%%%%%%%%%%%%%%%%%%%%%%

\section{Supplementary theorem}

The following lemma states that under a mild condition, the group zero-norm regularized minimization problem has a nonempty global optimal solution set.

\begin{lemma}\label{existence-global-opt-sol}
	Let $A\in\mathbb{R}^{n\times p}$ and $b\in\mathbb{R}^n$ be given, and let $g:\mathbb{R}^n\rightarrow\mathbb{R}$ be an lsc coercive function with ${\rm inf}_{z\in\mathbb{R}^n}g(z)>-\infty$. Then, for any given $\nu>0$, the following problem
	\begin{equation}\label{coercive}
		\mathop{\min}_{x\in\mathbb{R}^p}\{\nu g(Ax-b)+\|{G}(x)\|_0\}
	\end{equation}
	has a nonempty globally optimal solution set.
\end{lemma}
\begin{proof}
	Notice that the objective function of \eqref{coercive} is lower bounded. So, it has an infimum, say $\alpha^*$. Then there exists a sequence $\{x^k\}\subset\mathbb{R}^p$ such that for each $k$,
	\begin{equation}\label{c1}
		\nu g(Ax^k-b)+\|{G}(x^k)\|_0\leq\alpha^*+\frac{1}{k}~~\forall~k.
	\end{equation}
	If the sequence $\{x^k\}$ is bounded, then by letting $\overline{x}$ be an arbitrary limit point of $\{x^k\}$ and using the lower semicontinuity of $x\rightarrow g(Ax-b)$ and $\|{G}(\cdot)\|_0$, we have
	\begin{equation}\label{c2}
		\nu g(A\overline{x}-b)+\|{G}(\overline{x})\|_0\leq\alpha^*.
	\end{equation}
	This shows that $\overline{x}$ is a global optimal solution of the problem \eqref{coercive}. We next consider the case that $\{x^k\}$ is unbounded. For convenience, we define the following two disjoint index sets:
	\begin{align*}
	&I:=\{i\in\{1,2,\ldots,m\}|\{x_{J_i}^k\} {\rm is~unbounded}\}\\
	&{\rm and}~\overline{I}:=\{1,2,\ldots,m\}\backslash I.
	\end{align*}
	Together with inequality \eqref{c1}, it immediately follows that
	\begin{equation}\label{c3}
		\nu g(Ax^k-b)+|I|+\sum_{i\in\overline{I}}{\rm sign}(\|x_{J_i}^k\|)\leq\alpha^*+\frac{1}{k}
	\end{equation}
	for all sufficiently large $k$. Since $g(Ax-b)$ is coercive, which means that there exists a bounded sequence $\{z^k\}\subset\mathbb{R}^n$ such that $z^k=Ax^k-b$. Let $K=\bigcup_{i\in I}{J_i}$ and $\overline{K}=\{1,2,\ldots,m\}\backslash K$. Clearly, $A_Kx_K^k=b+z^k-A_{\overline{K}}x_{\overline{K}}^k$. Notice that $\{z^k\}$ and $\{x_{\overline{K}}^k\}$ are bounded. We may assume (taking a subsequence if necessary) that $\{z^k\}$ and $x_{\overline{K}}^k$ are convergent, say, $z^k\rightarrow z^*$ and $x_{\overline{K}}^k\rightarrow\xi^*$. Notice that for each $k$, $x_{K}^k$ is a solution of the system $A_{K}y=b+z^k-A_{\overline{K}}x_{\overline{K}}^k$. That is, $\{b+z^k-A_{\overline{K}}x_{\overline{K}}^k\}\subset A_{K}(\mathbb{R}^{|K|})$. Together with the closeness of the set $A_{K}(\mathbb{R}^{|K|})$, we have $b+z^*-A_{\overline{K}}\xi^*\in A_K(\mathbb{R}^{|K|})$. So, there exists $u^*\in\mathbb{R}^{|K|}$ such that $A_Ku^*=b+z^*-A_{\overline{K}}\xi^*$, i.e., $A_Ku^*+A_{\overline{K}}\xi^*-z^*=b$. Now passing the limit to \eqref{c3} and using $z^k=Ax^k-b$, it yields
	\[
	\nu g(z^*)+|I|+\sum_{i\in \overline{I}}{\rm sign}(\|\hat{\xi}^*_{J_i}\|)\leq\alpha^*,
	\]
	where $\hat{\xi}^*\in\mathbb{R}^P$ with $\hat{\xi}^*_K=0$ and $\hat{\xi}^*_{\overline{K}}={\xi}^*$. Let $\hat{u}^*\in\mathbb{R}^P$ with $\hat{u}_K^*=u^*$ and $\hat{u}_{\overline{K}}^*=0$. Along with $\nu g(A_{K}u^*+A_{\overline{K}}\xi^*-b)+\|{G}(\hat{u}^*)\|_0+\|{G}(\hat{\xi}^*)\|_0\leq\nu g(z^*)+|I|+\sum_{i\in \overline{I}}{\rm sign}(\|\hat{\xi}_{J_i}^*\|)$, we conclude that $x^*$ with $x_K^*=u^*$ and $x_{\overline{K}}^*=\xi^*$ is a global optimal solution of \eqref{coercive}.
\end{proof}

\bibliographystyle{IEEEtran}

\newpage
\onecolumn
\appendices
% \section{Supplementary theorem}
%\subsection{Proof Lemma \ref{u-outlemma}}
%\begin{proof}

%\end{proof}

%\subsection{Proof Lemma \ref{u_insolution}}
%\begin{proof} 
%\end{proof}

%\subsection{Proof Theorem \ref{PGM-converge} }
%\begin{proof} 
%\end{proof}

%\subsection{Proof Theorem \ref{the-convergence}}
%\begin{proof}

%\end{proof}

%------------------------------------------------------------------------------------------------

\end{document}